\documentclass[reqno]{amsart}
\usepackage{amscd}
\usepackage[dvips]{graphics}

\def\beq{\begin{equation}}
\def\eeq{\end{equation}}
\def\ba{\begin{array}}
\def\ea{\end{array}}

\newenvironment{abs}{\textbf{Abstract}\mbox{  }}{ }
\newenvironment{key words}{\textbf{Keywords}\mbox{  }}{ }
\newtheorem{thm}{Theorem}[section]
\newtheorem{lm}[thm]{Lemma}
\newtheorem{prop}[thm]{Proposition}
\newtheorem{crl}[thm]{Corollary}

\theoremstyle{definition}
\renewenvironment{proof}{\noindent{\textbf{Proof.}}}{\hfill$\Box$}
\newtheorem{rem}[thm]{Remark}

\theoremstyle{remark}
 
\numberwithin{equation}{section}
\begin{document}
\pagestyle{plain} % \today
% start to revise on March 11. Observe the HLS inequality and simplify some theorem. -MJ, 3-11-2012
%circulate 3-26-2012
%final check 3-15-2012 -MJ

%turn down by CPAM 4-24, rewrite the introduction, and shooting for JAMS
%turn down by JAMS 5-15/2012, modified, and aims in GAFA? 7-04-2012

%final modification. Decide to submit to JFA, 7-13-2012
%correct lemma 3.5, 3.6's citations; 7-23-2012
%rejected from JFA without report

%revised on 9-3-2013, plan to post on Axiv soon.

\title{Sharp Hardy-Littlewood-Sobolev inequality on the upper half space}
%On some singular integral operators on the upper half space
\author  {Jingbo Dou and Meijun Zhu}
\address{Jingbo Dou, School of Statistics, Xi'an University of Finance and
Economics, Xi'an, Shaanxi, 710100, China and Department of Mathematics, The University of Oklahoma, Norman, OK 73019, USA}
\email {jbdou@xaufe.edu.cn}

\address{ Meijun Zhu, Department of Mathematics,
The University of Oklahoma, Norman, OK 73019, USA }
\email { mzhu@ou.edu}

\maketitle

\begin{abs}
There are at least two directions concerning the extension of classical sharp Hardy-Littlewood-Sobolev inequality: (1) Extending the sharp inequality on general manifolds; (2) Extending it for the negative exponent $\lambda=n-\alpha$ (that is for the case of $\alpha>n$).
In this paper we confirm the possibility for the extension along the first direction by establishing the sharp
Hardy-Littlewood-Sobolev inequality on the upper half space (which is conformally equivalent to a ball).   The existences of extremal
functions are obtained; And for certain range of the exponent, we classify all extremal functions
via the method of moving sphere.
%The nonlinear integral equation can be viewed as a
%natural generalization of partial differential equations with nonlinear Neumann boundary condition. Some basic %properties about the
%extension, as well as the corresponding Liouville theorem are obtained. Due to the special domain for the integral %equations, The
%Liouville theorems are obtained via the method of moving sphere.
% %Consider the following integral equations in upper half space $\mathbb{R}^n_+$
%\begin{equation*}
%u(x)=\frac12\int_{\partial\mathbb{R}^n_+}\big(|x-y|^{p}+|\bar{x}-y|^{p}
%\big)u^q(y)dy\quad x\in\mathbb{R}^n_+.
%\end{equation*} where $p,q\in \mathbb{R}, n\ge2$,
%$\bar{x}=(x_1,x_2,\cdots,x_{n-1},-x_n)$ is the reflection of the point $x$ about the hyperplane $x_n=0$.
\end{abs}

\smallskip

\begin{key words}Hardy-Littlewood-Sobolev inequality; Conformal Laplacian operator; Green's  function;  Sharp constant;
Moving sphere method
\end{key words}\\
%\footnotetext{
\textit{\bf Mathematics Subject Classification(2010).}  	35A23, 42B37
%\footnotetext{\textit{Key words and phrases. } Hardy-Littlewood-Sobolev inequality; Conformal Laplacian operator; Green's  function;  Sharp constant;
%Moving sphere method}

%---------------------------------------------------------------------------------
\section{Introduction\label{Section 1}}
The classical sharp Hardy-Littlewood-Sobolev (HLS) inequality (\cite{HL1928, HL1930, So1963, Lieb1983}) states that
\begin{equation}\label{class-HLS}
|\int \int f(x)|x-y|^{-(n-\alpha)} g(y) dx dy|\le N(p,\alpha, n)||f||_p||g||_t
\end{equation}
for all $f\in L^p(\mathbb{R}^n), \, g\in L^t(\mathbb{R}^n),$ $1<p, \, t<\infty$, $0<\alpha <n$ and
$
1/p +1 /t+ (n-\alpha)/n=2.$
Lieb \cite{Lieb1983} proved the existence of the extremal function to the inequality with sharp constant, and computed the best constant in the case of $t=p$. The sharp HLS inequality implies Moser-Trudinger-Onofri and Beckner inequalities \cite{B1993}, as well as Gross's logarithmic Sobolev inequality \cite{Gr}. All these inequalities play significant roles in solving global geometric problems, such as Yamabe problem, Ricci flow problem, etc. Besides the recent extension of the sharp HLS on the Heisenberg group by Frank and Lieb \cite{FL2012},
there are at least two directions concerning the extension of the above sharp HLS inequality: (1) Extending the sharp inequality on general manifolds; (2) Extending it for the negative exponent $\lambda=n-\alpha$ (that is for the case of $\alpha>n$). In this paper, we study the extension of sharp HLS to the upper half space $\mathbb{R}^n.$

We start with manifolds with boundary. One of the simplest manifolds with boundary is the upper half space, or under the conformal equivalence, a ball in $\mathbb{R}^n$. By introducing an extension operator, we establish the sharp
HLS inequality on the upper half space, and prove the existences of extremal
functions; For certain exponent, we classify all extremal functions
via the method of moving sphere, which was introduced in early work of Li and Zhu \cite{LZ1995}. The current work builds a solid foundation for extending the classical sharp HLS on general manifolds. We shall outline (without proof) the general approach for such extensions later in this introduction.

\subsection{Singular integral operator on the upper half space}
Let $f(x)\in C^2_0(\mathbb{R}^n)$ for $n\ge 3$. The pointwise defined potential equation
$$%\begin{equation}\label{1-1}
-\Delta u=f
$$%\end{equation}
is equivalent to, up to a harmonic function,  the following (globally defined) integral equation:
$$%\begin{equation}\label{1-2}
u(x)=\frac1{n(n-2)\omega_n}\int_{\mathbb{R}^n}\frac{f(y)}{|x-y|^{n-2}}dy,
$$%\end{equation}
where, and throughout the paper, $\omega_n$ denotes the volume of the $n-$dimensional unit ball.   Generally, for all $\alpha>0$, a
positive solution $u(x)\in H^{\alpha/2}(\mathbb{R}^n)$ to
\begin{equation}\label{1-3}
(-\Delta)^{\frac \alpha 2} u=f, \ \ \ u\in H^{\alpha/2}(\mathbb{R}^n)
\end{equation}
in the distribution sense is given by
$$
\int_{\mathbb{R}^n}(-\Delta)^{\frac \alpha 4} u(x) (-\Delta)^{\frac \alpha 4} \phi(x) dx =\int_{\mathbb{R}^n} \phi(x) f(x) dx,
$$
 for all $\phi(x)\in C^\infty_0(\mathbb{R}^n)$, where $$\int_{\mathbb{R}^n}(-\Delta)^{\frac \alpha 4} u(x)
(-\Delta)^{\frac \alpha 4} \phi(x) dx =\int_{\mathbb{R}^n}|\xi|^{\alpha} \widehat{u}(\xi) \overline{\widehat{\phi}}(\xi)  d\xi.$$ If $\alpha$
is not an even number, equation (\ref{1-3}) is globally defined. It is also known that it
 is equivalent to the  integral equation (see, e.g., Stein \cite{Stein} on $P_{117}$)
\begin{equation}\label{1-4}
u(x)=\frac1{c(n,\alpha) }\int_{\mathbb{R}^n}\frac{f(y)}{|x-y|^{n-\alpha}}dy,
\end{equation} where $c(n,\alpha)=\pi^\frac n2 2^\alpha\Gamma(\frac \alpha2)/\Gamma(\frac{n-\alpha}2).$ See other related work in Chen and Li
\cite{CL2011}.

For $f=u^{\frac{n+\alpha}{n-\alpha}}$, equation (\ref{1-4}) is also  satisfied by the extremal functions to certain sharp
HLS inequality for the singular integral operator
$$ I_\alpha
f(x)=\frac1{c(n,\alpha)}\int_{\mathbb{R}^n}\frac{f(y)}{|x-y|^{n-\alpha}}dy.
$$
see, e.g., \cite{Lieb1983, CLO2006, Li2004}.

%Very few work on fractional Laplacian operator on a domain different to the whole space $\mathbb{R}^n$ or sphere $\mathbb{S}^n$,
%partially due to the fact that it is nontrivial to define fractional Laplacian on such a domain.

%In this paper, we initiate our study for other singular integral operators, which are also related to the fractional Laplacian operators
%on the upper half space.

Parallel to the potential equation in the whole space, we consider the Laplacian equation on the upper half space with Neumann boundary
condition. Let $f(y)\in C_0^2(\mathbb{R}^{n-1})$ for $n\ge 3$.
% with support in a ball $B_R(0)$.
The pointwise defined partial differential equation
$$%\begin{equation}\label{1-5}
\begin{cases}
-\Delta u(x', x_n)=0, ~~~&\mbox{for\ }~ x_n>0 ~\mbox{and\ } x' \in \mathbb{R}^{n-1},\\
u_{x_n}(x', 0) =-f(x'), ~~ &\mbox{for \ }~ x' \in \mathbb{R}^{n-1}
\end{cases}
$$%\end{equation}
is equivalent to, up to a harmonic function and a constant multiplier, the following  integral equation:
\begin{equation}\label{1-6}
u(x)=\int_{\partial \mathbb{R}^n_+}\frac{f(y)}{|x-y|^{n-2}}dy, \quad ~~~~~~~~~~\forall x=(x',x_n)\in\mathbb{R}^n_+,
\end{equation}
where and throughout the paper,  $|x-y|=\sqrt{|x'-y|^2+x_n^2}$ for $x=(x', x_n)\in \mathbb{R}^n_+$ and $y \in
\partial \mathbb{R}^n_+ =\mathbb{R}^{n-1}.$
Equation (\ref{1-6}) can also be viewed as another type of harmonic extensions of $f(y)$.
%%$\bar{x}=(x_1,x_2,\cdots,x_{n-1},-x_n)$ denotes the
%%reflection of the point $x$ about the hyperplane $x_n=0$, and

Generally, for $\alpha \in (1, n)$, we can introduce an extension operator for $f(y)\in C_0^\infty(\mathbb{R}^{n-1})$ as
\begin{equation}\label{1-7}
E_\alpha f(x)=\int_{\partial \mathbb{R}^n_+}\frac{f(y)}{|x-y|^{n-\alpha}}dy, \quad ~~~~~~~~~~\forall x=(x',x_n)\in\mathbb{R}^n_+.
\end{equation}
It will be clear that the above extension is in fact a Laplacian type extension operator for even number $\alpha$ (see Proposition
\ref{prop1-1} in Section \ref{Section 5}). Other properties about this extension and its relation to fractional Laplacian operator will be
discussed in Section \ref{Section 5}.

In the meantime, we  consider the dual operator $R_\alpha$ for $g(x)\in C^\infty(\mathbb{R}_+^{n})$ with compact set in $B_R(0)$ for
large $R$:
\begin{equation}\label{1-9}
R_\alpha g(y)=\int_{\mathbb{R}_+^{n}} \frac{g(x)}{|x-y|^{n-\alpha}}dx, \quad ~~~~~~~~~~\forall x=(x',x_n)\in\mathbb{R}^n_+, \,
y\in\partial\mathbb{R}^n_+.
\end{equation}
$R_\alpha$ can be viewed as a restriction operator.

The main goal of this paper is to establish the sharp HLS type inequalities for operators $E_\alpha$ and
$R_\alpha$.

\begin{thm}\label{HLSD-theo}
 For any  $1<\alpha<n, 1< p<\frac {n-1}{\alpha-1}$, and
 \begin{equation}\label{pq}
 \frac 1q=\frac{n-1}n(\frac1p-\frac{\alpha-1}{n-1}),
\end{equation}
  %Define
% \begin{equation*}
% Tf(x)=\frac12\int_{\partial\mathbb{R}^n_+}\big(\frac{1}{|x-y|^{n-\alpha}}
% +\frac{1}{|\bar{x}-y|^{n-\alpha}}\big)f(y)dy~~~~~~\forall\,x=(x',x_n)\in\mathbb{R}^n_+.
%\end{equation*}
there is a best constant $C_e(n,\alpha,p)>0$ depending on $n,  \ \alpha  $ and $p$, such that
%\begin{equation}\label{HLSD-1}
%\|E_\alpha f\|_{L^q_W( \mathbb{R}^n_+)}\le C(n,\alpha,p)\|f\|_{L^p(
%\partial\mathbb{R}^n_+)}.
%\end{equation}and
\begin{equation}\label{HLSD-2}
\|E_\alpha f\|_{L^q( \mathbb{R}^n_+)}\le C_e(n,\alpha,p)\|f\|_{L^p(
\partial\mathbb{R}^n_+)},
\end{equation}
and the equality holds for certain extremal functions. Moreover, all extremal functions are radially symmetric (with respect to some
points).

\end{thm}

Note that inequality (\ref{HLSD-2}) is equivalent to the following  HLS inequality.
\begin{thm}\label{HLSB}({\bf HLS inequality on the upper half space}).  For any $1<\alpha<n$, and $1<t, \, p<\infty$ satisfying
\begin{equation}\label{mj-7-13}
\frac{n-1}n \cdot \frac 1p +\frac 1t+\frac{n-\alpha+1}n=2,
\end{equation}
the following sharp inequality holds for all $f\in L^p(\partial \mathbb{R}_+^n), \, g\in L^t(\mathbb{R}_+^n):$
\begin{equation}\label{HLSD-O}
\int_{\mathbb{R}^n_+}\int_{\partial\mathbb{R}^n_+}\frac{g(x) f(y)}{|x-y|^{n-\alpha}}dydx\le C_e (n,\alpha,p)\|f\|_{L^p(\partial\mathbb{R}^n_+)}\|g\|_{L^t(\mathbb{R}^n_+)}.
\end{equation}
%where $t=q'=q/(q-1)$. Thus $t$  satisfies
%\[\frac1t=\frac{n-1}n\big(\frac{n+\alpha-1}{n-1}-
%\frac1p \big).
%\]
\end{thm}

This yields the sharp inequality for $R_\alpha$.

\begin{crl}\label{HLSD-theo-dual} For $1<\alpha<n$, $1<t<\frac {n}{\alpha}$, and
\begin{equation}\label{pq1}
\frac{1}q=\frac{n}{n-1}\big(\frac1t-\frac{\alpha}n\big),
\end{equation}
there is a best constant $C_r(n,\alpha,t)>0$ depending on $n,  \ \alpha  $ and $t$, such that
 %it holds
%\begin{equation}\label{HLSD-3w}
%\|R_\alpha g\|_{L^{q}_W(\partial\mathbb{R}^n_+
%)}\le C(n,\alpha,p)\|g\|_{L^p( \mathbb{R}^n_+)},
%\end{equation}
\begin{equation}\label{HLSD-3}
\|R_\alpha g\|_{L^{q}(\partial\mathbb{R}^n_+)}\le C_r(n,\alpha,t)\|g\|_{L^t(\mathbb{R}^n_+)},
\end{equation}
and the equality holds for certain extremal functions. Moreover,
$C_r(n,\alpha,t)=C_e(n,\alpha,p),$ where $p$ is given via \eqref{mj-7-13}.
%Moreover, all extremal functions are radially symmetric (with respect to some
%points).
\end{crl}

The best constants in (\ref{HLSD-2}) can be classified as
\[ C_e(n,\alpha,p)=\sup\{\|E_\alpha f\|_{L^q(\mathbb{R}^n_+)}\,:\, \,\,\|f\|_{L^p( \partial\mathbb{R}^n_+)}=1\}.
\]
%and
%\[ C_r(n,\alpha,p) =\sup\{\|R_\alpha g\|_{L^q(\partial \mathbb{R}^n_+)}\,:\, \,\|g\|_{L^p(\mathbb{R}^n_+)}=1\},
%\]
The extremal functions to inequality \eqref{HLSD-2}, up to a positive constant multiplier, satisfy the following integral equation:
\begin{equation}\label{equ-1}
f^{p-1}(y)=\int_{\mathbb{R}^n_+} \frac{(E_\alpha f(x))^{q-1}}{|x-y|^{n-\alpha}}dx,\quad\quad \forall\, y\in\partial\mathbb{R}^n_+.
\end{equation}
%\begin{equation}\label{equ-1_1}
%g^{p-1}(x)=\int_{\partial\mathbb{R}^n_+} \frac{(R_\alpha g(y))^{q-1}}{|x-y|^{n-\alpha}}dy,\quad\quad \forall\, x\in\mathbb{R}^n_+.
%\end{equation}
Using the method of moving sphere, we are able to classify all positive solutions to the above equation for certain power $p$, and obtain the precise value for the best constant.

\begin{thm}\label{Classification}
For $1<\alpha<n$,  $ p=\frac{2(n-1)}{n+\alpha-2}$, and $q=\frac{2n}{n-\alpha}$, if $f\in
L^\frac{2(n-1)}{n+\alpha-2}_{loc}(\partial\mathbb{R}^n_+)$  is a positive solution to equation \eqref{equ-1}, then  $f\in
C^\infty(\partial\mathbb{R}^n_+),$ and must be the form of
\begin{equation*}
f(y)= c(n,\alpha)\big(\frac{1}{|y-y_0|^2+d^2}\big)^\frac{n+\alpha-2}2,\quad\quad \forall \, y\in\partial\mathbb{R}^n_+
\end{equation*}  for some constants $c(n,\alpha)$,  $d>0,$ and  $y_0\in\partial\mathbb{R}^n_+.$  Thus,
$
C_e(n,\alpha,p)
$ can be computed explicitly, in particular, for $\alpha=2$,
$$C_e(n,2,\frac{2(n-1)}{n})=n^{\frac{n-2}{2(n-1)}}\omega_{n}^{1-\frac1n-\frac{1}{2(n-1)}}.$$

%(2). For $n\ge2, 1<\alpha<n$,  $ p=\frac{2n}{n+\alpha}$,   and $q=\frac{2(n-1)}{n-\alpha},$ if $g\in
%L^\frac{2n}{n+\alpha}_{loc}(\mathbb{R}^n_+)$   is a solution to equation \eqref{equ-1_1}, then  $g\in C^\infty(
%\overline{\mathbb{R}^n_+}),$ and on $\partial \mathbb{R}^n_+$, $g$  must be the form of
%\begin{equation*}
%g(x',0)= c_1(n,\alpha)\big(\frac{1}{|x'-x'_0|^2+d^2}\big)^\frac{n+\alpha}2,\quad\quad \forall \, x=(x', t)\in\mathbb{R}^{n}_+
%\end{equation*}  for some constants $c_1(n,\alpha)$,  $d>0,$ and  $x'_0\in\mathbb{R}^{n-1}.$
%Thus
%$
%C_r(n,\alpha,p)
%$ can be computed explicitly, in particular, for $\alpha=2$,
%$$C_r(n,2,\frac{2n}{n+2})= n^{\frac{n-2}{2(n-1)}}\omega_{n}^{1-\frac1n-\frac{1}{2(n-1)}}.$$

\end{thm}

We do not know whether similar inequalities hold for $\alpha=1$ or not. On the other hand, a limiting inequality for $\alpha=n$ can be obtained. See more remark late in this introduction and Section 5.2.

 It will be clear in classifying all positive solutions to \eqref{equ-1} via the method of moving sphere, that for $ p=\frac{2(n-1)}{n+\alpha-2}$ and $q=\frac{2n}{n-\alpha}$, inequality (\ref{HLSD-2}) is equivalent to an
integral inequality on a ball $B_r(0)$:
\begin{equation}\label{HLSD-2-ball}
\|\tilde E_\alpha f\|_{L^q( B_r(0)) }\le C_e(n,\alpha,p)\|f\|_{L^p(
\partial B_r(0))},
\end{equation}
where $$ \tilde E_\alpha f(x) =\int_{\partial B_r(0) }\frac{f(y)}{|x-y|^{n-\alpha}}dS_y, \quad ~~~~~~~~~~\forall \, x\in B_r(0),
$$
and one  extremal function is always constant on the boundary.  Moreover, for $\alpha=2$, we can show that $\tilde E_\alpha f(x)$ is a constant function if extremal function $f(x)$ is a constant on the boundary.  These  enable us  to obtain the best constant $C_e(n,\alpha,p)$  from inequality
(\ref{HLSD-2-ball}).

As a byproduct to our study, we also establish the following nonexistence result to \eqref{equ-1} with subcritical exponents.

\begin{thm}\label{Nonexistence}
Let $1<\alpha<n$,  $\frac{2(n-1)}{n+\alpha-2}\le p<2,$ and $2<q\le\frac{2n}{n-\alpha}$ satisfying
\[\frac{1}{p}-\frac{n}{q (n-1)}<\frac{\alpha-1}{n-1}.
\]
If  $f\in L^\frac{(n-1)(2-p)}{\alpha-1}_{loc}(\partial\mathbb{R}^n_+)$
is a nonnegative solution to equation \eqref{equ-1}, then $f(x)=0$.
\end{thm}

\subsection {Sharp HLS on general manifolds} Our current work has a strong implication in the extension of sharp HLS on general manifolds.

For a smooth and bounded domain $\Omega \subset \mathbb{R}^n$, we certainly can introduce $$\tilde E_\alpha f(x) =\int_{\partial \Omega
}\frac{f(y)}{|x-y|^{n-\alpha}}dS_y, \quad ~~~~~~~~~~\forall \, x\in \Omega
$$
for smooth functions $f(y)$ defined on the boundary of $\Omega$. Using the tools built in this paper, one shall be able to show that
\begin{equation}
\|\tilde E_\alpha f\|_{L^q(\Omega) }\le C_\Omega (n,\alpha,p)\|f\|_{L^p(
\partial \Omega )},
\end{equation}
 with the sharp constant $C_\Omega (n,\alpha,p)$ that is related to the geometric property of the domain. Using certain standard bubbling sequence of functions, we can show that
$$ C_e (n,\alpha,p)\le C_\Omega (n,\alpha,p).$$
One may ask: for which type of domains will the strict inequality hold? Note that similar extension operators were studied by Hang, Wang and Yan
\cite{HWY2008}, where they used the classical Riesz potential (corresponding to partial differential equations with Dirichlet boundary condition); And similar question was discussed by them (Conjecture 1.1 in \cite{HWY2009}).

However, the danger for linking $C_\Omega (n,\alpha,p)$ to the the geometric property of the domain is that the geometric property of a general domain does not play essential role in the definition of the extension operator $\tilde E_\alpha$. Similar question can be asked for other kernels on manifolds (for example, Stein and Weiss type kernel \cite{SW1958}).  Broadly, we may ask, how to extend the sharp HLS on general manifolds.

 For $\alpha=2$, we observe: the common property for the kernel to our inequality \eqref{HLSD-O} and that to the classical HLS \eqref{class-HLS} is that both of them, up to a constant multiplier, are the Green's function of corresponding conformal Laplacian operators.

We thus propose (for simplicity we only consider $\alpha=2$ here)  to extend the classical Sharp HLS inequality  as follows. Let $(M^n, g)$ (for $n\ge 3$) be a  compact smooth Riemmanian manifold, and $R_g$ be its scalar curvature. Let $G_y(x)$ be the Green's function to the conformal Lapacian operator $L_g:=-\Delta_g+\frac{n-2}{4(n-1)}R$ with pole at $y\in M^n$, then

\noindent{\it Extension of the sharp HLS on compact Riemannian manifold}. There is a best constant $N(M^n, g, p, 2)>0$, such that
\begin{equation}\label{class-HLS_man}
|\int_{M^n} \int_{M^n} f(x)G_y(x) h(y) dV_x dV_y|\le N(M^n, g,  p, 2)||f||_p||h||_p
\end{equation}
holds for all $f\in L^p(M^n), \, h\in L^p(M^n),$ where $p=2n/(n+2)$. Moreover, $N(M^n, g,  p, 2)\ge N(\mathbb{S}^n, g_0,  p, 2)$ where $(\mathbb{S}^n, g_0)$ is the standard sphere with induced metric $g_0$; And for $n=3, 4, 5,$ the strict inequality holds if $G_y(x)$ has positive mass (or if $(M^n, g)$ is not conformally equivalent to $(\mathbb{S}^n, g_0)$ by the positive mass theorem of Schoen and Yau \cite{SY1988}). It turns out that \eqref{class-HLS_man} is equivalent to the sharp Sobolev inequality on general compact manifold (we thank X. Wang for verifying this fact with the second named author in \cite{WZ2012}). The general extension of sharp HLS inequality on compact manifolds, in particular for $\alpha \ne 2$, is not clear.

\smallskip

  Let $(M^n, \partial M, g)$ ($n\ge 3$) be a  Riemannian manifold with smooth boundary $\partial M$. For simplicity, we consider a bounded domain $\Omega\subset \mathbb{R}^n$ with smooth boundary $\partial \Omega$, whose mean curvature function is $h$.  Let $F_y(x)$ be the Green's function to the conformal Lapacian operator:
$$
\begin{cases}
-\Delta F_y(x)=0, ~~~ x \in \Omega,\\
\frac {\partial F_y}{\partial \nu}+\frac{n-2}2 h F_y =\delta_y, ~~~ x \in \partial \Omega
\end{cases}
$$
with pole $y$ on the boundary, then

\noindent{\it Extension of the sharp HLS on manifold with boundary}. There is a best constant $N_b(\Omega, p, 2)>0$, such that
\begin{equation}\label{class-HLS_man2}
|\int_{\Omega} \int_{\partial \Omega} f(y)F_y(x) g(x)  dS_ydV_x|\le N_b(\Omega,  p, 2)||f||_{L^p(\partial \Omega)}||g||_{L^t(\Omega)}
\end{equation}
holds for all $f\in L^p(\partial \Omega), \, g\in L^t(\Omega),$ where $p= 2(n-1)/n,$ and $t= 2n/(n+2).$
  Moreover, $N_b(\Omega,  p, 2)\ge N_b(B_1(0), p, 2)$ where $B_1(0)$ is the unit ball center at the origin; And for $n=3, 4, 5,$ the strict inequality holds if $\Omega$ is not conformal to $B_1(0). $

  \smallskip

  Now, for a bounded domain $\Omega\subset \mathbb{R}^n$ with smooth boundary $\partial \Omega$, let $K_y(x)=n(n-2)\omega_nF_y(x),$ and define the extension operator
 \begin{equation}\label{ext_op}   E_{2, \Omega} f(x) =\int_{\partial \Omega
}{f(y)}K_y(x)dS_y, \quad ~~~~~~~~~~\forall \, x\in \Omega.
\end{equation}
We shall be able to show that, there is a best constant $\hat C_\Omega (n, 2,p)$, such that
\begin{equation}\label{HLSD-2-domain}
\| E_{2, \Omega} f\|_{L^q(\Omega) }\le \hat C_\Omega (n,2,p)\|f\|_{L^p(
\partial \Omega )},
\end{equation}
and $C_e (n, 2,p)\le \hat C_\Omega (n, 2,p)$. Moreover, for $n=3, 4, 5$,  the equality holds if and only if  $\Omega$ is a ball. We believe that for $\alpha=2$, sharp inequality \eqref{class-HLS_man2} and \eqref{HLSD-2-domain} are equivalent to the sharp trace inequality proved in Escobar \cite{E1992-1}. The general extension of sharp HLS on compact manifolds with boundary for $\alpha\ne 2$ is not clear.
 %More details on such general extension will be given later.

%It shall be clear by now that our current work is certainly the very first, and fundamental step to the above extension of the sharp HLS on general manifold.

\medskip

The classification results in Theorem \ref{Classification} are obtained for continuous functions, based on some basic calculus
 lemmas initially used in Li and Zhu \cite{LZ1995} for $C^1$ functions.  Even though the extremal functions are weak solutions
to the integral equations, similar argument to the proof of Brezis and Kato's lemma  implies that all weak solutions are indeed
smooth. Details on the regularity of solutions are given in Section 4. Interestingly, the same calculus lemmas in Li
and Zhu \cite{LZ1995} for $C^1$ functions were later proved to be true for all continuous functions by Li and Nirenberg (see Appendix B in \cite{Li2004}), and
eventually it is proved to be true for all finite, non-negative measures in $\mathbb{R}^n$ by Frank and Lieb \cite{FL2010}. In \cite{FL2010}, the invariant property is used to derive the absolute continuity for the measure with respect to Lebesgue measure, which is different to our regularity argument.

%Besides extensive studies of HLS inequalities,  Their results are closely related to the study of fractional Lapalacian operator in the whole space \cite{CS2007},  fractional Yamabe problem
%\cite{CG2011,GM2011,GQ2011,HL1999,JX2011,Li2004}, as well as  the study of fractional prescribing curvature problem on $\mathbb{S}^n$ \cite{JLX2011a,JLX2011b}. Besides its own merits, our study of singular integral operators shall have similar impact on the study of Yamabe and prescribing curvature problems with boundary, as our early work \cite{LZ1995}.

For $\alpha =n$, the modern folklore inequality is obtained via taking the limit of the power, see Corollary {\ref{folklore}} below. The other case $\alpha>n$ will be addressed in our future paper \cite{DZ3}, after we explain how to extend the classical HLS inequality \eqref{class-HLS} for negative $\lambda=n-\alpha$ in \cite{DZ2}. As we pointed out at the beginning, this is another direction for extending the classical sharp HLS inequality. We need to point out that sharp Sobolev type inequalities with negative exponent on the standard sphere $\mathbb{S}^n$ did appear in, for example,   Ai, Chou and Wei \cite{ACW2000}, Yang and Zhu \cite{YZ2004}, Hang and Yang \cite{HY2004}, and Ni and Zhu \cite{NZ1}. See also our recent work \cite{DZ2012} for equations with negative exponent.

 Our results are closely related to the study of fractional Lapalacian operator in the whole space \cite{CS2007},  fractional Yamabe problem
\cite{CG2011,GM2011,GQ2011,HL1999,JX2011,Li2004},   fractional prescribing curvature problem on $\mathbb{S}^n$ \cite{JLX2011a,JLX2011b}, and to  the study of Yamabe and prescribing curvature problems with boundary, see, e. g. \cite{E1988, E1992-1,E1992, LZ1995, LZ1997, HL1999}.

Our proof of sharp integral inequalities is essentially along the line of the classical paper by Lieb
\cite{Lieb1983}. Our approach to the classification of extremal function is similar to that in \cite{LZ1995, CLO2006, HWY2008}  and
\cite{Li2004} in spirit, but quite different in details. For instance, we use the method of moving sphere for a system of equations
to classify the extremal functions.

We shall first prove Theorem \ref{HLSD-theo}  in Section 2. The inequalities are proved along the line of
the proof for the classical HLS inequality (see, e.g. Stein \cite{Stein}). The sharp inequalities and the existences
of extremal functions are obtained via the arguments based on symmetrization. The classification of extremal functions via the
method of moving sphere is given in Section 3. In Section 4, we establish the regularity properties for solutions to the integral
equations use the  argument similar to the proof of Brezis and Kato's lemma \cite{BK1979}. In the final section, we discuss the similar
inequalities in a ball, limit case, and the relation between our operators and fractional Lapalacian operators in the whole space. As a byproduct of our classification argument, we also obtain the nonexistence result (Theorem \ref{Nonexistence}) for a subcritical integral equation system, though it is still an open problem to find the extremal function for inequality \eqref{HLSD-2} if $p\ne 2(n-1)/(n+\alpha-2)$ (corresponding to the case $p\ne 2n/(n+\alpha)$ for HLS in whole space).

\section{\textbf{The sharp Hardy-Littlewood-Sobolev inequality on the upper half space} \label{Section 2}}
In this section, we shall establish the sharp  HLS inequalities on the upper half space. We first establish the
following Young inequality on the upper half space $\mathbb{R}^n_+$.

\begin{lm}\label{Young ineq}
 Suppose that  $p,q,r\in[1,\infty]$ are three parameters  satisfying $\frac1p+\frac1q=1+\frac1r.$
 For any  $h\in L^p(\partial\mathbb{R}^n_+),g\in L^q(\mathbb{R}^n_+),$ define
\[ g*h(x)=\int_{\partial\mathbb{R}^n_+}g(x-y)h(y) dy, ~~~~~~~~ \forall \, x \in \mathbb{R}^n_+.
 \]
Let  $\tilde{g}(x')=g(x', 0)$ be the restriction of $g(x)$ on the boundary of the upper half space $\mathbb{R}^n_+$. If
$|g(x',x_n)|\le|\tilde{g}(x')|$ for all  $x'\in  \partial \mathbb{R}^n_+$ and $x_n >0$,
  then
 \begin{equation}\label{Young-1}
 \|g*h\|_{L^r(\mathbb{R}^n_+)}\le\|h\|_{L^p(\partial\mathbb{R}^n_+)}
 \|g\|_{L^q(\mathbb{R}^n_+)}^{q/r}\|\tilde{g}\|^{1-q/r}_{L^q(\partial\mathbb{R}^n_+)}.
 \end{equation}
\end{lm}
\begin{proof}
If $r=\infty,$ then $\frac1p+\frac1q=1.$ By H\"{o}lder inequality and the translation
invariance of Lebesgue measure, we have
\begin{eqnarray*}
|\int_{\partial\mathbb{R}^n_+} g(x'-y,x_n)  h(y)dy|&\le&
\int_{\partial\mathbb{R}^n_+}|\tilde{g}(x'-y)| |h(y)|dy\\
&\le&\|\tilde{g}\|_{L^q(\partial\mathbb{R}^n_+)}\|h\|_{L^p(\partial\mathbb{R}^n_+)}.
\end{eqnarray*} Hence,
\[\|g*h\|_{L^\infty(\mathbb{R}^n_+)}\le\|\tilde{g}\|_{L^q(\partial\mathbb{R}^n_+)}
\|h\|_{L^p(\partial\mathbb{R}^n_+)}.
\]

Next, we consider the case $1\le r<\infty.$
 For any  $a,b\in[0,1]$, $p_1,p_2 \in[0,\infty]$ satisfying $\frac1{p_1}+\frac1{p_2}+\frac1r=1,$ and $x\in \mathbb{R}^n_+$, we have
 \begin{eqnarray*}
 |g*h(x)|&\le&\int_{\partial\mathbb{R}^n_+}|g(x-y)|^{1-a}|h(y)|^{1-b}|g(x'-y,x_n)|^{a}|h(y)|^{b} dy\\
&\le&\int_{\partial\mathbb{R}^n_+}|g(x-y)|^{1-a}|h(y)|^{1-b}|\tilde g(x'-y)|^{a}|h(y)|^{b} dy\\
 &\le&\big(\int_{\partial\mathbb{R}^n_+}|g(x-y)|^{(1-a)r}|h(y)|^{(1-b)r} dy\big)^\frac1r\|\tilde{g}\|^{a}_{L^{ap_1}
 (\partial\mathbb{R}^n_+)}\|h\|^{b}_{L^{bp_2}(\partial\mathbb{R}^n_+)}.\\
 \end{eqnarray*}
  Taking the $r^{th}$ power of the above inequality and integrating on $\mathbb{R}^n_+$, we have
\begin{eqnarray*}
 \int_{\mathbb{R}^n_+}|g*h(x)|^rdx
 &\le& \|\tilde{g}\|^{ar}_{L^{ap_1}
 (\partial\mathbb{R}^n_+)}\|h\|^{br}_{L^{bp_2}(\partial\mathbb{R}^n_+)} \cdot \int_{\mathbb{R}^n_+}
 \int_{\partial\mathbb{R}^n_+}|g(x-y)|^{(1-a)r}|h(y)|^{(1-b)r} dydx \\
 &=& \|\tilde{g}\|^{ar}_{L^{ap_1}
 (\partial\mathbb{R}^n_+)}\|h\|^{br}_{L^{bp_2}(\partial\mathbb{R}^n_+)}
 \cdot\int_{\partial\mathbb{R}^n_+}|h(y)|^{(1-b)r}\big(\int_{\mathbb{R}^n_+}|g(x-y)|^{(1-a)r} dx\big)dy \\
 &=&\|h\|^{(1-b)r}_{L^{(1-b)r}(\partial\mathbb{R}^n_+)}\|g\|^{(1-a)r}_{L^{(1-a)r}(\mathbb{R}^n_+)}
 \|\tilde{g}\|^{ar}_{L^{ap_1}(\partial\mathbb{R}^n_+)}\|h\|^{br}_{L^{bp_2}(\partial\mathbb{R}^n_+)}.
 \end{eqnarray*}
 Now, choosing $p=(1-b)r=bp_2,q=(1-a)r=ap_1$,  we have inequality \eqref{Young-1}.

 To complete the proof, we still need to show that $p$, $q$ and $r$ are arbitrary indices in $[1, \infty)$ satisfying $\frac1{p}+\frac1{q}=1+\frac1r.$
 Since $p=(1-b)r=bp_2,q=(1-a)r=ap_1$, we have $a=1-\frac{q}r,b=1-\frac pr,$ and $p_1=\frac q a, p_2=\frac p b$. We can directly
 verify that
\[1=\frac1{p_1}+\frac1{p_2}+\frac1r=\frac1q+\frac1p-\frac1r.
 \]

Conversely, if $p,q,r \in [1, \infty)$ are given and satisfy $\frac1{p}+\frac1{q}=1+\frac1r$, then $a=1-\frac{q}r \ge 0 $
and $b=1-\frac pr \ge 0$. It is obvious that both $a, b \le 1.$ Also, for $p_1=\frac q a, p_2=\frac p b$, it is easy to check that
\[\frac1{p_1}+\frac1{p_2}+\frac1r=1.
\]

\end{proof}

We are now ready to prove the inequality in Theorem \ref{HLSD-theo} with a constant (which may not be sharp).

\textbf{Proof of Theorem \ref{HLSD-theo}.}

% We first show \eqref{HLSD-2} holds for some constant.

  For $\alpha \in (1,n),$ $p\in [1, \frac {n-1}{\alpha-1})$ and $q$  given
 by  $\frac
1q=(\frac1p-\frac{\alpha-1}{n-1})\cdot \frac{n-1}n,$ we first prove
\begin{equation}\label{HD-1-1}
\|E_\alpha f\|_{L^q_W( \mathbb{R}^n_+)}\le C(n,\alpha,p)\|f\|_{L^p(
\partial\mathbb{R}^n_+)}
\end{equation}
for some constant $C(n,\alpha,p)$. That is, we need to show that there is a constant $C(n,\alpha,p)>0,$ such that
\begin{equation}\label{HD-1}
meas\{x:|E_\alpha f(x)|>\lambda\}\le\big(C(n,\alpha,p)\frac{\|f\|_{L^p(\partial\mathbb{R}^n_+)}}\lambda\big)^q, ~~~~~\forall f\in
L^p(\partial\mathbb{R}^n_+), ~\forall~ \lambda>0.
\end{equation}
Note that inequality (\ref{HD-1-1}) implies, via the Marcinkiewicz interpolation \cite{Stein,ONeil}, that
$$%\begin{equation}\label{HD-1-2}
\|E_\alpha f\|_{L^q( \mathbb{R}^n_+)}\le C_1(n,\alpha,p)\|f\|_{L^p(
\partial\mathbb{R}^n_+)},
$$%\end{equation}
or even slight stronger inequality
\begin{equation}\label{HD-1-3}
\|E_\alpha f\|_{L^q( \mathbb{R}^n_+)}\le C_2(n,\alpha,p)\|f\|_{L^{p,q}(
\partial\mathbb{R}^n_+)}
\end{equation}
 for $p\in (1, \frac {n-1}{\alpha-1})$ and $q$  given by  $\frac
1q=(\frac1p-\frac{\alpha-1}{n-1})\cdot \frac{n-1}n.$

Let $X$ be either $\mathbb{R}^n_+$ or $\partial \mathbb{R}^n_+$.  Recall: for a given measurable function $f(x)$ on $X$ and $1\le
p<\infty$, the weak $L^p$ norm of $f(x)$  is defined by
\[\|f\|_{L^{p}_W(X)}=\inf \{A>0 \,|\,meas \{x \in X \,|\,f(x)>t\}\cdot t^p\le A^p\},
\]
and $L^{p}_W(X)=\{f\,|\, f ~\text{is a measurable function, and}~\|f\|_{L^{p}_W(X)}<\infty \}$.
 More generally, for $1<p<\infty,
1\le q<\infty$, the Lorentz $L^{p.q}(X) $ norm is given by
\[\|f\|_{L^{p,q}(X)}=\big[\int_0^\infty(x^\frac1pf^{**}(x))^q\frac{dx}x\big]^\frac1q.
\]
where $f^{**}=\frac1x\int_0^xf^*(t)dt$, $f^*$ is the nonnegative nonincreasing rearrangement of $|f|$ (see, O'Neail \cite{ONeil} $P_{136}$).  For $1\le
p\le\infty,q=\infty$,
\[\|f\|_{L^{p,\infty}(X)}=\sup_{x>0}x^\frac1pf^{**}(x),
\]
 and the Lorentz space $
L^{p,q}(X)=\{f\,|\,f ~\text{is a measurable function, and}~ \|f\|_{L^{p,q}(X)}<\infty\}.$ It is known that
$L^{p}_W(X)=L^{p,\infty}(X)$ is a special case of such space. See also, Stein \cite{Stein}.

 For any $r>0$, define
 \[E_{\alpha, r}^1 f(x)=\int_{\partial\mathbb{R}^n_+, |y-x|\le r}\frac{f(y)}{|x-y|^{n-\alpha}}dy,
  \]
  and
   \[E_{\alpha, r}^2 f(x)=\int_{\partial\mathbb{R}^n_+, |y-x|\ge  r}\frac{f(y)}{|x-y|^{n-\alpha}}dy.
  \]
Then for any $\lambda>0$,
\begin{equation}\label{HD-2}
meas\{x:|E_\alpha f(x)|>2\lambda\}\le meas\{x:|E_{\alpha,r}^1 f (x)|>\lambda\}+meas\{x:|E_{\alpha,r}^2f(x)|>\lambda\}.
\end{equation}

 It is enough to prove  inequality \eqref{HD-1} with $2\lambda$ in place of $\lambda$ in the left side of the
inequality. We can further assume $\|f\|_{L^p(\partial\mathbb{R}^n_+)}=1.$

%For given $r>0$ and $y\in\partial\mathbb{R}^n_+$, define $G=\{x\in\mathbb{R}^n_+ \ :\,|x-y|\le r\}$, we have
%\[Tf(x)=Tf(x)\chi_G+Tf(x)(1-\chi_G).
%\]

%Note for $x\in\mathbb{R}^n_+$ and $y\in\partial\mathbb{R}^n_+$, $|x-y|=|\bar{x}-y|.$ We only need to prove the inequality for :
% \[E_\alpha^1 f(x)=\int_{\partial\mathbb{R}^n_+}\frac{f(y)}{|x-y|^{n-\alpha}}dy.
%  \]
%
%
%Define the characteristic function on the set $G$ as
%\begin{equation*}
%\chi_G(x)=
%\begin{cases}1,&~~~~~~~~x\in G,\\
%0,&~~~~~~~~x\not\in G.
%\end{cases}
%\end{equation*}
%
%Let $1\le p<\frac{n-1}{n}q<\infty,$ and $\frac{n}{n-1}\frac 1q=\frac1p-\frac{\alpha-1}{n-1},$ the mapping $f\longmapsto Tf$ is of
%the weak-type $(p,q)$, in the sense that

Since $|x-y|\ge |x'-y|$ for any $x=(x',x_n)\in\mathbb{R}^n_+$ and $
 y\in\partial\mathbb{R}^n_+$, using Young inequality
\eqref{Young-1} (with $q=1$), we have
\begin{eqnarray*}
\|E_{\alpha, r}^1 f \|_{L^p(\mathbb{R}^n_+)}&\le&\big(\int_{\mathbb{R}^n_+}
\frac{\chi_r(|x-y|)}{|x-y|^{n-\alpha}}dx\big)^{\frac{1}p}\big(\int_{\partial\mathbb{R}^n_+}
\frac{\chi_r(|x'-y|)}{|x'-y|^{n-\alpha}}dy\big)^{\frac{p-1}p}\|f\|_{L^p(\partial\mathbb{R}^n_+)}\\
&=:&D_1D_2,
\end{eqnarray*}
where $\chi_r(x)=1$ for $|x|\le r$ and $\chi_r(x)=0$ for $|x|>r$, and
\begin{eqnarray*}
D_1&=&\big(\int_{B^+_r(y)}
\frac{1}{|x-y|^{n-\alpha}}dx\big)^{\frac{1}p}=C_1(n,\alpha,p)r^\frac{\alpha}p,\\
D_2&=&\big(\int_{B^{n-1}_r(x')} \frac{1}{|x'-y|^{n-\alpha}}dy\big)^{\frac{p-1}p}=C_2(n,\alpha,p)r^\frac{(\alpha-1)(p-1)}p.
\end{eqnarray*}
We use the fact that $\alpha>1$ in the computation for $D_2$.
It follows that
\begin{eqnarray}\label{HD-3}
meas\{x:|E_{\alpha,r}^1 f|>\lambda\}\le\frac{\|E_{\alpha,r}^1 f \|_{L^p(\mathbb{R}^n_+)}^p}{\lambda^p}\le \frac{C_3(n,\alpha,p)r^{\alpha
p-(p-1)}}{\lambda^p}.
\end{eqnarray}
On the other hand,
\begin{eqnarray*}
\|E_{\alpha, r}^2 f\|_{L^\infty(\mathbb{R}^n_+)}
&\le&\int_{\partial\mathbb{R}^n_+}\frac{1-\chi_r(|x'-y|)}{|x'-y|^{n-\alpha}}|f(y)|dy\nonumber\\
&\le&\big(\int_{\partial\mathbb{R}^n_+}\big(\frac{1-\chi_r(|x'-y|)}{|x'-y|^{n-\alpha}}\big)^{p'}dy\big)^\frac1{p'}\|f\|_{L^p(\partial\mathbb{R}^n_+)}\nonumber\\
&=:&D_3,
\end{eqnarray*} where $\frac1p+\frac1{p'}=1.$ It is easy to see (note: $p<(n-1)/(\alpha-1)$) that
\begin{eqnarray*}
D_3&=&\big(\int_{\partial\mathbb{R}^n_+\backslash B^{n-1}_r(x')}
\big(\frac{1}{|x'-y|^{n-\alpha}}\big)^{p'}dy\big)^{\frac{1}{p'}}=C_4(n,\alpha,p)r^{\frac{(n-1)-(n-\alpha)p'}{p'}}.
\end{eqnarray*}
If we choose $r=(C_4^{-1}(n, \alpha,p)  \lambda)^{\frac{p'}{(n-1)-(n-\alpha)p'}}$, thus
$\lambda=C_4(n,\alpha,p)r^{\frac{(n-1)-(n-\alpha)p'}{p'}}$. It follows that
$
\|E_{\alpha, r}^2f\|_{L^\infty(\mathbb{R}^n_+)}\le\lambda,
$
thus
\[meas\{x:|E_{\alpha, r}^2 f |>\lambda\}=0.
\]
Bringing the above and \eqref{HD-3} into \eqref{HD-2}, we have
\begin{eqnarray*}
meas\{x:|E_\alpha f|>2\lambda\}&\le& meas\{x:|E_{\alpha, r}^1 f |>\lambda\}\le \frac{C_3(n, \alpha) r^{\alpha p-(p-1)}}{\lambda^p}\\
&=&C_5(n, \alpha)\lambda^{\frac{p'[\alpha p-(p-1)]}{(n-1)-(n-\alpha)p'}-p},
\end{eqnarray*} where
\begin{eqnarray*}
\frac{p'[\alpha
p-(p-1)]}{(n-1)-(n-\alpha)p'}-p&
%=&-p\big[\frac{\alpha
%p-(p-1)}{(n-1)(p-1)-(n-\alpha)p}+1\big]
%\\&=&-p\big[\frac{\alpha
%p-(p-1)}{(n-1)-(\alpha-1)p}+1\big]\\&
=&-\frac{np}{(n-1)-(\alpha-1)p}.
\end{eqnarray*}
Let
\[q=\frac{np}{(n-1)-(\alpha-1)p},
\] that is,
\[\frac 1q=\frac{n-1}{n}\big(\frac1p-\frac{\alpha-1}{n-1}\big),
\]
we then obtain (\ref{HD-1}).
% Hence, we conclude \eqref{HLSD-1}, and \eqref{HLSD-2} follows from
% the Marcinkiewicz interpolation theorem (see \cite{Stein}).

% \hfill $\Box$
%
%
%
%
%\textbf{Proof of Theorem \ref{HLSD-theo-dual} for the inequality with a rough constant.}
%We will prove it via a duality argument using inequality \eqref{HD-1-2}.
%%{\bf so we only have the strong inequality?}
%Let $g$  and
%$f$ be two functions defined on $\mathbb{R}^n_+$ and $\mathbb{R}^{n-1}$, respectively.  By H\"{o}lder inequality and
% \eqref{HD-1-2} (noting $\frac p{p-1}>\frac{n}{n-\alpha}$, since $1<p<\frac n\alpha$ ) we have
%\begin{eqnarray*}
%\int_{\partial\mathbb{R}^n_+}(R_\alpha g)(y)f(y)dy
%&=&\int_{\partial\mathbb{R}^n_+}dy\int_{\mathbb{R}^n_+} \frac{g(x)f(y)}{|x-y|^{n-\alpha}}dx\\
%&=&\int_{\mathbb{R}^n_+}(E_\alpha f)(x)g(x)dx\\
%&\le&\|E_\alpha f\|_{L^{\frac p{p-1}}(\mathbb{R}^n_+)}\|g\|_{L^p(\mathbb{R}^n_+)}\\
%&\le&C\|f\|_{L^{m_0}(\partial\mathbb{R}^n_+)}\|g\|_{L^p(\mathbb{R}^n_+)},
%\end{eqnarray*} where $m_0=\frac{(n-1)p}{n(p-1)+(\alpha-1)p}=\frac{(n-1)p}{(n-1)p-n+\alpha p}>1$ (since $1<p<n/\alpha$).
%
%On the other hand, we can choose $f(y)=(R_\alpha g)^{q-1}(y)$ on the above with $q \cdot (m_0-1)=m_0$ and obtain  inequality
%\eqref{HLSD-3}.
%%{\bf only keep \eqref{HLSD-3}?}
%\hfill $\Box$
%
%%
%%{\bf extremal function 2-2-2012}
%%%------------------------------------------------------------------
%%\section{Existence of extremal function for sharp inequalities\label{Section 4}}
%
%
%\smallskip

%\textbf{Continue the proof of Theorem \ref{HLSD-theo} for the inequality with sharp constant.}

 The sharp constant to inequality (\ref{HLSD-2}) is classified by
\[ C_e({n,\alpha,p})=\sup\{\|E_\alpha f\|_{L^q(\mathbb{R}^n_+)}\,:\, f\in L^p{(\partial\mathbb{R}^n_+)},\,\|f\|_{L^p( \partial\mathbb{R}^n_+)}=1\}.
\]
Using  symmetrization argument,  we will show that the above supreme is attained by a radially symmetric (with respect to some
point) function. Moreover, we will show that all extremal functions for inequality (\ref{HLSD-2}) are radially symmetric with respect
to some point.

%minimizing sequence to ..., and prove the convergence of the sequence, thus obtain the existence of the minimizer for the sharp
%inequalities \eqref{HLSD-2}.

For any measurable function $u(x)$ on $\mathbb{R}^n$ vanishing at infinity, we can define its radially symmetric, non-increasing rearrangement
function $u^*$.
 $u^*(x)$ is a nonnegative lower-semicontinuous
 function and has the same distribution as $|u|$.  Moreover, for $1\le p\le \infty$, the following inequality holds
\begin{equation}\label{cov_prop1}
\|u*v\|_{L^p(\mathbb{R}^n)}\le\|u^**v^*\|_{L^p(\mathbb{R}^n)},
\end{equation}
where $u*v(x)=\int_{\mathbb{R}^n}u(y)v(x-y)dy$ is the convolution product. See, for example,  Brascamp, Lieb and Luttinger \cite{BLL1974}.

 Let $\{f_j\}_{j=1}^\infty \in C_0^\infty (\partial \mathbb{R}^n_+) $ be a nonnegative maximizing sequence
with $\|f_j\|_{L^p(\partial\mathbb{R}^{n}_+)}=1$,  $f_j^*$ be the rearrangement of $f_j.$

 Since
\[\|f_j^*\|_{L^p(\partial \mathbb{R}_+^{n})}=\|f_j\|_{L^p(\partial \mathbb{R}_+^{n})}=1,
\]
and by (\ref{cov_prop1}),
\begin{eqnarray*}
\|E_\alpha
(f_j)\|^q_{L^q(\mathbb{R}^n_+)}&=&\int_0^\infty\int_{\partial\mathbb{R}^{n}_+}
\big(\int_{\partial\mathbb{R}^{n}_+}\frac{f_j(y)}{(|x'-y|^2+x_n^2)^{\frac{n-\alpha}2}}dy\big)^qdx'dx_n\\
&\le&\int_0^\infty\int_{\partial\mathbb{R}^{n}_+}
\big(\int_{\partial\mathbb{R}^{n}_+}\frac{f_j^*(y)}{(|x'-y|^2+x_n^2)^{\frac{n-\alpha}2}}dy\big)^qdx'dx_n ~ \ \  \ \ \\
&=&\|E_\alpha (f^*_j)\|^q_{L^q(\mathbb{R}^n_+)},
\end{eqnarray*}
we know that $\{f_j^*\}_{j=1}^\infty$ is also a maximizing sequence.  Without loss of generality, we can assume that this is a sequence of
nonnegative radially symmetric and  non-increasing functions. To avoid that  $f_j$ may converge to a trivial function, we need to modify
the sequence further.

For convenience, denote $e_1=(e'_1,0)=(1,0,\cdots,0,0)\in\mathbb{R}^n$, and
\begin{equation*}
a_j:=sup_{\lambda>0} \lambda^{-\frac{n-1}p}f_j(\frac {e'_1}\lambda).
\end{equation*}
Note that  for $y\in \partial \mathbb{R}_+^{n}\setminus \{0\}$,
\[0\le f_j(y)=f_j(|y|e'_1)=|y|^{-\frac{n-1}p}|y|^{\frac{n-1}p}f_j(|y|e'_1)\le a_j|y|^{-\frac{n-1}p},
\]
and
\[\|f_j\|_{L^{p,\infty}(\partial \mathbb{R}_+^{n})}\le C a_j.
\]
Thus, for $\frac 1q=(\frac1p-\frac{\alpha-1}{n-1})\cdot \frac{n-1}n$, by \eqref{HD-1-3}  we have
% {\bf will check the definition $L^{p,q}$}
\begin{eqnarray*}
\|E_\alpha (f_j)\|_{L^q(\mathbb{R}^n_+)}&\le& C(n,\alpha,p)\|f_j\|_{L^{p,q}(\partial\mathbb{R}^{n}_+)}\\
&\le&
C(n,\alpha,p)\|f_j\|_{L^{p,p}(\partial\mathbb{R}^{n}_+)}^{\frac{n-1-(\alpha-1)p}n}
\|f_j\|_{L^{p,\infty}(\partial\mathbb{R}^{n}_+)}^{\frac{1+(\alpha-1)p}n}\\
&\le& C(n,\alpha,p)a_j^{\frac{1+(\alpha-1)p}n}.
\end{eqnarray*}
This implies $a_j\ge C(n,\alpha,p)>0.$  Define $ f^{\lambda}_j(x) = {\lambda}^{-\frac{n-1}p}f(\frac {x}\lambda)$.  Thus, we may choose $\lambda_j$ such that
$$f^{\lambda_j}_j(e'_1)={\lambda_j}^{-\frac{n-1}p}f(\frac {e'_1}{\lambda_j})  \ge C(n,\alpha,p)>0.$$
Also, it is easy to verify that
\[\|f_j^{\lambda_j}\|_{L^p(\partial \mathbb{R}_+^{n})}=\|f_j\|_{L^p(\partial \mathbb{R}_+^{n})},
\quad \text{and}\quad \|E_\alpha (f_j^{\lambda_j})\|_{L^q(\mathbb{R}^n_+)}=\|E_\alpha (f_j)\|_{L^q(\mathbb{R}^n_+)}.
\]
Thus, $\{f_j^{\lambda_j}\}_{j=1}^\infty$ is also a maximizing sequence.  Therefore, we can further assume that the nonnegative and
radially non-increasing maximizing sequence $\{f_j\}_{j=1}^\infty$ with $\|f_j\|_{L^p(\partial \mathbb{R}_+^{n})}=1$ also satisfies
$f _j(e'_1)\ge C(n,\alpha,p)>0$ for all $j\ge 1$. It follows that
\begin{equation}|f_j(y)|\le \omega_{n-1}^{-\frac1p}|y|^{-\frac{n-1}p}, \ \ \ \ \ \ \ \forall y \in \partial \mathbb{R}_+^{n}.
\label{conv-9-3}
\end{equation}
Up to a subsequence, we can find a nonnegative, radially nonincreasing function $f$ such that $f_j\to f$ a.e..  Hence we conclude
that $f(y)\ge C(n,\alpha,p)>0$ for $|y|\le1,f_j\rightharpoonup f$ in $L^p(\partial \mathbb{R}_+^{n})$, and $\|f\|_{L^p(\partial
\mathbb{R}_+^{n})}\le1$. From Brezis and Lieb's Lemma \cite{BL1983},  we see
\begin{eqnarray*}
\int_{\partial \mathbb{R}_+^{n}}\big||f_j|^p-|f|^p -|f_j-f|^p\big|dy \to 0.
\end{eqnarray*} So
\begin{eqnarray}\label{Ex-1}
\|f_j-f\|^p_{L^p(\partial \mathbb{R}_+^{n})}&=&\|f_j\|^p_{L^p(\partial \mathbb{R}_+^{n})}-\|f\|^p_{L^p(\partial \mathbb{R}_+^{n})}
+o(1)\nonumber\\
&=&1-\|f\|^p_{L^p(\partial \mathbb{R}_+^{n})}+o(1).
\end{eqnarray}
 Note $\alpha\in (1, n)$.  Also from \eqref{conv-9-3} we have: for any fixed $x\ne 0$, and $|y|>2|x|$,
 $$|\frac{f_j(y)}{|x-y|^{n-\alpha}}-\frac{f(y)}{|x-y|^{n-\alpha}}|\le \frac C{|y|^{n-\alpha+(n-1)/p}}.$$
   Thus   $(E_\alpha f_j)(x)\to(E_\alpha f)(x)$ for almost all  $x\in\mathbb{R}^{n}_+$ by dominated convergence theory.  It follows that
\begin{eqnarray*}
C_e(n,\alpha,p)+o(1) &=&\|E_\alpha f_j\|^q_{L^q(\mathbb{R}^{n}_+)}\\
&=&\|E_\alpha f\|^q_{L^q(\mathbb{R}^{n}_+)}+\|E_\alpha (f_j-f)\|^q_{L^q(\mathbb{R}^{n}_+)}+o(1)\\
&\le&C_e^q({n,\alpha,p})\|f\|^q_{L^p(\partial \mathbb{R}_+^{n})}+C_e^q({n,\alpha,p})\|f_j-f\|^q_{L^p(\partial \mathbb{R}_+^{n})}+o(1).
\end{eqnarray*} So
\begin{eqnarray}\label{Ex-2}
1&\le&\|f\|^q_{L^p(\partial \mathbb{R}_+^{n})}+\|f_j-f\|^q_{L^p(\partial \mathbb{R}_+^{n})}+o(1).
\end{eqnarray}
Combining \eqref{Ex-1} and \eqref{Ex-2} and letting $j\to \infty$, we obtain
\begin{eqnarray*}
1&\le&\|f\|^q_{L^p(\partial \mathbb{R}_+^{n})}+ \big(1-\|f\|^p_{L^p(\partial \mathbb{R}_+^{n})}\big)^\frac qp.
\end{eqnarray*}
From the fact that $q>p$ and $f\not\equiv 0,$ we have $\|f\|_{L^p(\partial \mathbb{R}_+^{n})}=1$.  Hence $f_j\to f$ in $L^p(\partial
\mathbb{R}_+^{n})$ and $f$ is a maximizer. This shows the existence of an extremal function to the sharp inequality \eqref{HLSD-2}.

Next we show the radial symmetry and  monotonicity hold for all extremal functions.  Assume that $f\in L^p(\partial \mathbb{R}_+^{n})$
is a maximizer,  so is $|f|$. Thus  $\|E_\alpha f\|_{L^q(\mathbb{R}^{n}_+)}=\|E_\alpha |f|\|_{L^q(\mathbb{R}^{n}_+)}$,
%On the other hand, $|(E_\alpha f)|(x)\le (E_\alpha |f|)(x)$ for $x\in \mathbb{R}^n_+$,
which implies $f\ge0$ or $f\le0$. Without loss of generality, we only consider the case of $f\ge0$. Then the Euler-Lagrange equation
for $f(x)$ is, up to a constant multiplier, given by equation \eqref{equ-1}
%\begin{equation*}
%f^{p-1}(y)=\int_{\mathbb{R}^n_+} \frac{(E_\alpha f(x))^{q-1}}{|x-y|^{n-\alpha}}dx,\quad\quad \forall\, y\in\partial\mathbb{R}^n_+,
%\end{equation*} where $c$ is some constant.  Due to $\|f\|_{L^p(\partial\mathbb{R}^{n}_+)}=1$, we have
%$$c=\|E_\alpha |f|\|_{L^q(\mathbb{R}^{n}_+)}=c_{n,\alpha,p}^q.$$
%After scaling by a positive constant, it has
\begin{equation*}
f^{p-1}(y)=\int_{\mathbb{R}^n_+} \frac{(E_\alpha f(x))^{q-1}}{|x-y|^{n-\alpha}}dx,\quad\quad \forall\, y\in\partial\mathbb{R}^n_+.
\end{equation*}
On the other hand, (see, e.g. \cite{Lieb1977} $P_{103}$), if $u$ is nonnegative, radially symmetric, and strictly decreasing in the radial direction, $v$ is nonnegative,
$1<p<\infty$ and
\begin{equation*}\label{cov_prop2}
\|u*v\|_{L^p(\mathbb{R}^n)}=\|u*v^*\|_{L^p(\mathbb{R}^n)}<\infty,
\end{equation*} then $v(x)=v^*(x-x_0)$ for some $x_0\in\mathbb{R}^n.$
It follows from this fact and the Euler-Lagrange equation satisfied by $f(y)$, that $f(y)=f^*(y-y_0)=f^*(|y-y_0|)$ for
some $y_0 \in
\partial\mathbb{R}^{n}_+,$ where $f^*(r)$ is decreasing in $r$.

%\hfill $\Box$
%
%\medskip
%
%  It seems that the above rearrangement argument does not work for the  existence of extremal function to inequality (\ref{HLSD-3}). However,  the standard concentration compactness principle can be used to show the existence of extremal function. More details will be given in our forthcoming paper \cite{DZ2012}.
%
%\medskip
%
%{\bf will add more later here:
%
%For $\alpha=n$, we can obtain the following inequalities via taking limit in inequality (\ref{HLSD-2}) and (\ref{HLSD-3}).
%\begin{crl}\label{folklore}
%For...
%\end{crl}
%The case of $\alpha>n$ will be discussed in our forthcoming paper \cite{DZ3}, after we establish the analog extension to the classical HLS inequality in \cite{DZ2}.
%
%}
% remove these on 3-11-2012 MJ
%{\bf I am here now 2-2-2012}

%~~~~~~~~~~~~~~~~~~~~~~~~~~~~~~~~~~~~~~~~~~~~~~~~~~~~~~~~~~~~~~~~~~~~~~~~~~~~
\section{\textbf{Classification of positive solutions for integral equations } \label{Section 3}}
 In this section, we classify all nonnegative solutions to  integral equation \eqref{equ-1}
 for $p=2(n-1)/(n+\alpha-2)$ and $q=2n/(n-\alpha).$
  %and \eqref{equ-1_1}. We first focus on integral equation \eqref{equ-1}.

%{\bf rewording later: The proof of part
% (2) in  Theorem \ref{Classification} is almost the same as that of part (1), we thus focus on the proof on part (1) and skip the
% details for part (2).}

Let $f$ be a nonnegative function satisfying \eqref{equ-1}. Define $u(y)=f^{p-1}(y),v(x)=E_\alpha f(x)$, $\theta=\frac1{p-1},$ and
$\kappa=q-1.$ Then the single equation \eqref{equ-1} can be rewritten as an integral system
\begin{equation}\label{sys-1}
\begin{cases}
u(y)=\int_{\mathbb{R}^{n}_+}\frac{v^{\kappa}(x)}{|x-y|^{n-\alpha}}dx,&\quad y\in\partial\mathbb{R}^{n}_+,\\
v(x)=\int_{\partial\mathbb{R}^{n}_+}\frac{u^{\theta}(y)}{|x-y|^{n-\alpha}}dy,&\quad x\in\mathbb{R}_+^{n}
\end{cases}
\end{equation}with $\frac1{\kappa+1}=\frac{n-1}{n}\big( \frac{n-\alpha}{n-1}-\frac1{\theta+1}\big).$
Note: if $f\in L^p_{loc}(\partial\mathbb{R}^{n}_+)$, then $u\in L^{\theta+1}_{loc}(\partial\mathbb{R}^{n}_+).$ In next section, we
will show that if $(u, v)$ is a pair of nonnegative solutions to system \eqref{sys-1} with $u\in
L^{\theta+1}_{loc}(\partial\mathbb{R}^{n}_+),$ then $u\in C^\infty (\partial \mathbb{R}^{n}_+)$ and $v\in
C^\infty(\overline{\mathbb{R}^{n}_+}).$ From now on in this section, we assume both $u,\,  v$ are smooth functions.
Theorem \ref{Classification} follows from  the following theorem.

\begin{thm}\label{sys-critical}
Let $(u,v)$ be a pair of positive smooth solutions to system \eqref{sys-1}. For $1<\alpha<n,$ if
$$%\begin{equation}{\label{pq}}
\theta=\frac{n+\alpha-2}{n-\alpha}, \kappa=\frac{n+\alpha}{n-\alpha},
$$% \end{equation}
 then $u,v$ must be the following forms on $\partial\mathbb{R}^{n}_+$:
\begin{eqnarray*}
u(\xi)=c_1\big(\frac{1}{|\xi-\xi_0|^2+d^2}\big)^\frac{n-\alpha}2,\\
v(\xi,0)=c_2\big(\frac{1}{|\xi-\xi_0|^2+d^2}\big)^\frac{n-\alpha}2,
\end{eqnarray*} where $c_1, \, c_2, \, d>0, \xi_0\in\partial \mathbb{R}^n_+.$

\end{thm}

We prove the above theorem via the method of moving sphere, which is introduced by Li and Zhu in \cite{LZ1995}.

For $R>0,$ denote
$$B_R(x)=\{y\in\mathbb{R}^n\,: \,|y-x|<R, x\in\mathbb{R}^n\}, \, B_R^{n-1}(x)=\{y\in\partial\mathbb{R}^{n}_+\,: \,|y-x|<R, x\in\partial\mathbb{R}^{n}_+\},$$
$$B_R^+(x)=\{y=(y_1,y_2,\cdots,y_n)\in B_R(x)\,: \, y_n>0,x\in\partial\mathbb{R}^{n}_+\},$$
$$\Sigma_{x,R}^n=\mathbb{R}^n_+\backslash\overline{B_R^+(x)}, \,~
\Sigma_{x,R}^{n-1}=\partial\mathbb{R}^{n}_+\backslash\overline{B_R^{n-1}(x)}.$$
For $x\in\partial \mathbb{R}^n_+$ and $\lambda>0,$ we define the following transform:
\[
\omega_{x,\lambda}(\xi)=\big(\frac\lambda{|\xi-x|}\big)^{n-\alpha}\omega(\xi^{x,\lambda}),\quad\xi\in
\overline{\mathbb{R}^n_+}\,\backslash\{x\},
\]
where
$$\xi^{x,\lambda}=x+\frac{\lambda^2(\xi-x)}{|\xi-x|^2}$$
is the Kelvin transformation of $\xi$ with respect to $B_\lambda^+ (x)$. Also we write $\omega^k_{x,\lambda}(\xi):=
(\omega_{x,\lambda}(\xi))^k$.

First we have the following lemma.
\begin{lm}\label{Kelvin formula}
Let $1<\alpha<n, 0<\theta,\kappa<\infty$ and $(u,v)$ be a pair of positive solutions to system \eqref{sys-1}. Then, for any $x\in
\partial\mathbb{R}^{n}_+$,
\begin{eqnarray}
u_{x,\lambda}(\xi)&=&\int_{\mathbb{R}^{n}_+}\frac{v_{x,\lambda}^{\kappa}(\eta)}
{|\xi-\eta|^{n-\alpha}}\big(\frac\lambda{|\eta-x|}\big)^{\tau_1}d\eta,\quad \forall \, \xi\in\partial\mathbb{R}^{n}_+,\label{K-1}\\
v_{x,\lambda}(\eta)&=&\int_{\partial\mathbb{R}^{n}_+}\frac{u_{x,\lambda}^\theta(\xi)}{|\xi-\eta|^{n-\alpha}}
\big(\frac\lambda{|\xi-x|}\big)^{\tau_2}d\xi,\quad \forall \, \eta\in\mathbb{R}_+^{n},\label{K-2}
\end{eqnarray} where $\tau_1= n+\alpha-\kappa(n-\alpha),\tau_2=n+\alpha-2-\theta(n-\alpha).$
Moreover,
\begin{eqnarray}
u_{x,\lambda}(\xi)-u(\xi)&=&\int_{\Sigma_{x,\lambda}^n}P(x,\lambda;\xi,\eta)
\big[\big(\frac{\lambda}{|\eta-x|}\big)^{\tau_1} v_{x,\lambda}^{\kappa}(\eta)-v^{\kappa}(\eta)\big]d\eta,~~~~~~~~~~~~~~~~~\label{K-3}\\
v_{x,\lambda}(\eta)-v(\eta)&=&\int_{\Sigma_{x,\lambda}^{n-1}}P(x,\lambda;\eta,\xi)
\big[\big(\frac{\lambda}{|\xi-x|}\big)^{\tau_2}u_{x,\lambda}^\theta (\xi)-u^\theta(\xi)\big]d\xi,~~~~~~~~~~~~~~~~~~~~~~~~\label{K-4}
\end{eqnarray} where
\[
P(x,\lambda;\xi,\eta)=\frac1{|\xi-\eta|^{n-\alpha}}
-\big(\frac\lambda{|\xi-x|}\big)^{n-\alpha}\frac1{|\xi^{x,\lambda}-\eta|^{n-\alpha}},
\]and
\[
P(x,\lambda;\xi,\eta)>0,~~~~~~~~~~~~\text{for}~\forall \,\xi\in\Sigma_{x,\lambda}^{n-1},\, \eta\in\Sigma_{x,\lambda}^n,\lambda>0.
\]
\end{lm}

The proof of Lemma \ref{Kelvin formula} is similar to that in \cite{Li2004}.  Similar computation will also be used to derive an  analog inequality (\ref{HLSD-2-ball}) on a ball from inequality (\ref{HLSD-2}) in Section \ref{subsect 5.1}.

\begin{proof}
For any $x \in\partial\mathbb{R}^n_+,\eta\in\mathbb{R}^n_+$ and $\lambda>0$, let
$$y=\eta^{x,\lambda}=x+\frac{\lambda^2(\eta-x)}{|\eta-x|^2}.$$
The $n$ dimensional volume  forms in $y$ variable and $\eta$ variable are related by
\[dy=\big(\frac{\lambda}{|\eta-x|}\big)^{2n}d\eta. \]
To simplify the calculations, we write
\begin{eqnarray*}
A^+(\xi^{x,\lambda})&=&\int_{\Sigma^{n}_{x,\lambda}}\frac{v^\kappa(y)}{|\xi^{x,\lambda}-y|^{n-\alpha}}dy ,\quad
A^-(\xi^{x,\lambda})=\int_{B^+_\lambda(x)}\frac{v^\kappa(y)}{|\xi^{x,\lambda}-y|^{n-\alpha}}dy.
\end{eqnarray*}
Thus, from \eqref{sys-1} we can rewrite $u$ as follows:
\[
 u(\xi^{x,\lambda})=A^+(\xi^{x,\lambda})+A^-(\xi^{x,\lambda})
\]
for $x,\xi\in \partial \mathbb{R}_+^n$. Direct
calculation yields
\begin{eqnarray*}
A^+(\xi^{x,\lambda})&=&\int_{\Sigma^{n}_{x,\lambda}}\frac{v^\kappa(y)}{|\xi^{x,\lambda}-y|^{n-\alpha}}dy
=\int_{B^+_\lambda(x)}\frac{v^\kappa(\eta^{x,\lambda})}{|\xi^{x,\lambda}-\eta^{x,\lambda}|^{n-\alpha}}
\big(\frac{\lambda}{|\eta-x|}\big)^{2n}d\eta\\
&=&\int_{B^+_\lambda(x)}\frac{v^\kappa_{x,\lambda}(\eta)}{|\xi^{x,\lambda}-\eta^{x,\lambda}|^{n-\alpha}}
\big(\frac{\lambda}{|\eta-x|}\big)^{2n-\kappa(n-\alpha)}d\eta.
\end{eqnarray*}
Note that:
\[\frac{|\eta-x|}\lambda\frac{|\xi-x|}\lambda|\xi^{x,\lambda}-\eta^{x,\lambda}|=|\xi-\eta|.
\]
Thus,
\begin{eqnarray*}
A^+_{x,\lambda}(\xi)&:=&\big(\frac\lambda{|\xi-x|}\big)^{n-\alpha}A^+(\xi^{x,\lambda})\\
&=&\big(\frac\lambda{|\xi-x|}\big)^{n-\alpha}
\int_{B^+_\lambda(x)}\frac{v^\kappa_{x,\lambda}(\eta)}{|\xi^{x,\lambda}-\eta^{x,\lambda}|^{n-\alpha}}
\big(\frac{\lambda}{|\eta-x|}\big)^{2n-\kappa(n-\alpha)}d\eta\\
&=&\int_{B^+_\lambda(x)}\frac{v^\kappa_{x,\lambda}(\eta)}{|\xi-\eta|^{n-\alpha}}\big(\frac{\lambda}{|\eta-x|}\big)^{\tau_1}
d\eta.
\end{eqnarray*}
Similarly, we have
\begin{eqnarray*}
A^-_{x,\lambda}(\xi)&:=&\big(\frac\lambda{|\xi-x|}\big)^{n-\alpha}A^-(\xi^{x,\lambda})
=\int_{\Sigma^{n}_{x,\lambda}}\frac{v^\kappa_{x,\lambda}(\eta)}{|\xi-\eta|^{n-\alpha}}\big(\frac{\lambda}{|\eta-x|}\big)^{\tau_1}
d\eta.
\end{eqnarray*}
Hence,
\begin{eqnarray*}
u_{x,\lambda}(\xi) &=& A^+_{x,\lambda}(\xi)+A^-_{x,\lambda}(\xi)
= \int_{\mathbb{R}^n_+}\frac{v^\kappa_{x,\lambda}(\eta)}{|\xi-\eta|^{n-\alpha}}\big(\frac{\lambda}{|\eta-x|}\big)^{\tau_1}
d\eta.
\end{eqnarray*}  Identity \eqref{K-1} is established.

%Now, we show the formula \eqref{Kelvin-2}.
It also follows that
\begin{eqnarray*}
u_{x,\lambda}(\xi)-u(\xi)
&=&A^+_{x,\lambda}(\xi)+A^-_{x,\lambda}(\xi)
-\big(A^+(\xi)+A^-(\xi)\big)\\
&=& (A^-_{x,\lambda}(\xi)- A^+(\xi) )+(A^+_{x,\lambda}(\xi)-A^-(\xi))\\
&=&\int_{\Sigma^{n}_{x,\lambda}} \frac1{|\xi-\eta|^{n-\alpha}} \big[\big(\frac{\lambda}{|\eta-x|}\big)^{\tau_1} v^\kappa_{x,\lambda}(\eta)-v^\kappa(\eta)\big]d\eta\\
& &+(A^+_{x,\lambda}(\xi)-A^-(\xi)).
\end{eqnarray*}
On the other hand,
\begin{eqnarray*}
A^+_{x,\lambda}(\xi) &=&\big(\frac\lambda{|\xi-x|}\big)^{n-\alpha}A^+(\xi^{x,\lambda})\\
&=&\big(\frac\lambda{|\xi-x|}\big)^{n-\alpha}
\int_{\Sigma^{n}_{x,\lambda}}\frac{v^\kappa(y)}{|\xi^{x,\lambda}-y|^{n-\alpha}}dy,
\end{eqnarray*} and
%from  $A^-_{x,\lambda}(\xi)$ and  the fact $(\xi^{x,\lambda})^{x,\lambda}=\xi$, it follows
\begin{eqnarray*}
A^-(\xi)&=&A^-((\xi^{x,\lambda})^{x,\lambda})=\big(\frac\lambda{|\xi^{x,\lambda}-x|}\big)^{\alpha-n}A^-_{x,\lambda}(\xi^{x,\lambda})\\
&=&\big(\frac\lambda{|\xi^{x,\lambda}-x|}\big)^{\alpha-n}
\int_{\Sigma^{n}_{x,\lambda}}\frac{v^\kappa_{x,\lambda}(\eta)}{|\xi^{x,\lambda}-\eta|^{n-\alpha}}\big(\frac{\lambda}{|\eta-x|}\big)^{\tau_1}
d\eta\\
&=&\big(\frac\lambda{|\xi-x|}\big)^{n-\alpha}
\int_{\Sigma^{n}_{x,\lambda}}\frac{v^\kappa_{x,\lambda}(\eta)}{|\xi^{x,\lambda}-\eta|^{n-\alpha}}\big(\frac{\lambda}{|\eta-x|}\big)^{\tau_1}
d\eta.
\end{eqnarray*}
Combining the above computations, we have identity \eqref{K-3}.
\eqref{K-2} and \eqref{K-4} can be obtained in the same way.

 Now, for $\xi\in
\Sigma^{n-1}_{x,\lambda},\eta\in
\Sigma^{n}_{x,\lambda} $  and $\lambda>0$, we have
\begin{eqnarray*}
P(x,\lambda;\xi,\eta)&= &|\xi-\eta|^{\alpha-n}-\big(\frac{|\xi-x|}\lambda\big)^{\alpha-n}|\xi^{x,\lambda}-\eta|^{\alpha-n}\\
&=&|\xi-\eta|^{\alpha-n}-\frac1{\lambda^{\alpha-n}}\big(\big|\frac{\lambda^2(\xi-x)}{|\xi-x|}+|\xi-x|(x-\eta)\big|^2\big)^{\frac {\alpha-n}2}\\
&=&|\xi-\eta|^{\alpha-n}-\frac1{\lambda^{\alpha-n}}\big[ \lambda^4 +2\lambda^2 (\xi-x)(x-\eta)+(\xi-x)^2(x-\eta)^2\big]^{\frac {\alpha-n}2}\\
&>&|\xi-\eta|^{\alpha-n}-\frac1{\lambda^{\alpha-n}}\big[ (x-\eta)^4 +2(x-\eta)^2 (\xi-x)(x-\eta)+(\xi-x)^2(x-\eta)^2\big]^{\frac {\alpha-n}2}\\
%&=&|\xi-\eta|^{\alpha-n}-\big(\frac{|x-\eta|}{\lambda}\big)^{\alpha-n}\big[(x-\eta)^2 +2(\xi-x)(x-\eta)+(\xi-x)^2\big]^{\frac {\alpha-n}2}\\
&=&|\xi-\eta|^{\alpha-n}-\big(\frac{|x-\eta|}{\lambda}\big)^{\alpha-n}\big[ (x-\eta+\xi-x)^2\big]^{\frac {\alpha-n}2}\\
&=&|\xi-\eta|^{\alpha-n}\big(1-\big(\frac{|x-\eta|}{\lambda}\big)^{{\alpha-n}}\big)>0.
\end{eqnarray*}
Lemma \ref{Kelvin formula} is proved.
\end{proof}

Let $$\gamma:=\frac{n+\alpha-2}{n-\alpha}, \, \, \, \beta:=\frac{n+\alpha}{n-\alpha}.$$ It is clear in Lemma \ref{Kelvin formula}
that $\tau_1=\tau_2=0$ if and only if $\theta=\gamma$ and $\kappa=\beta.$ From now on in this section, we assume that $\theta=\gamma$
and $\kappa=\beta.$

%Now take $p=\frac{2(n-1)}{n+\alpha-2},q=\frac{2n}{n-\alpha}$, then $\theta=\frac{n+\alpha-2}{n-\alpha},
%\kappa=\frac{n+\alpha}{n-\alpha}$. If $f\in L^{\frac{2(n-1)}{n+\alpha-2}}_{loc}(\partial\mathbb{R}^{n}_+)$ then $u\in
%L^{\frac{2(n-1)}{n-\alpha}}_{loc}(\partial\mathbb{R}^{n}_+)$. Moreover, from previous section, we know $v\in
%L^{\frac{2n}{n-\alpha}}_{loc}(\overline{\mathbb{R}^{n}_+}).$ Hence, Theorem \ref{Classification} is equivalent to the following
%theorem.
%
%To prove this theorem, we employ the method of  moving sphere.
Define
\[\Sigma_{x,\lambda}^u=\{\xi\in\Sigma_{x,\lambda}^{n-1}\,|\,u(\xi)<u_{x,\lambda}(\xi)\},\quad\text{and}\quad
\Sigma_{x,\lambda}^v=\{\eta\in\Sigma_{x,\lambda}^{n}\,|\,v(\eta)<v_{x,\lambda}(\eta)\}.
\]
\begin{lm}\label{K-lm-1}
Assume the same conditions on $n, \alpha, \theta,$ and $\kappa$ as those in Theorem \ref{sys-critical}. Then for any $x\in\partial\mathbb{R}^{n}_+$,
there exists $\lambda_0(x)>0$ such that: $\forall \ 0<\lambda<\lambda_0(x),$
\begin{eqnarray*}
u_{x,\lambda}(\xi)&\le& u(\xi),\quad  a.e. \, \mbox{in} \, \, \Sigma_{x,\lambda}^{n-1},\\
v_{x,\lambda}(\eta)&\le& v(\eta),\quad  a.e. \, \mbox{in} \, \, \Sigma_{x,\lambda}^n.
\end{eqnarray*}
\end{lm}
\begin{proof}
For  $\xi\in\Sigma_{x,\lambda}^{u}$, we have, via  \eqref{K-3} and  mean value theorem, that
\begin{eqnarray*}
0\le u_{x,\lambda}(\xi)-u(\xi)&=&\int_{\Sigma_{x,\lambda}^n}P(x,\lambda;\xi,\eta)
\big[v_{x,\lambda}^{\beta}(\eta)-v^{\beta}(\eta)\big]d\eta,\\
&\le &\int_{\Sigma_{x,\lambda}^v }P(x,\lambda;\xi,\eta)
\big[v_{x,\lambda}^{\beta}(\eta)-v^{\beta}(\eta)\big]d\eta,\\
&\le&\int_{\Sigma_{x,\lambda}^v }\frac{v_{x,\lambda}^{\beta}(\eta)-v^{\beta}(\eta)}
{|\xi-\eta|^{n-\alpha}}d\eta,\\
&=&\beta\int_{\Sigma_{x,\lambda}^v }\frac{\phi_\lambda(v)^{\beta-1}(v_{x,\lambda}(\eta)-v(\eta))}
{|\xi-\eta|^{n-\alpha}}d\eta,\\
&\le&\beta\int_{\Sigma_{x,\lambda}^v}\frac{v_{x,\lambda}^{\beta-1}(\eta)(v_{x,\lambda}(\eta)-v(\eta)) }{|\xi-\eta|^{n-\alpha}}d\eta,
\end{eqnarray*} where $v(\eta)\le\phi_\lambda(v)\le v_{x,\lambda}(\eta)$ on $\Sigma_{x,\lambda}^v$.

For $t \in( 1, \frac n\alpha),$  take $k=\frac{(n-1)t}{n-\alpha t}, s=\frac{nt}{n-\alpha t}$.  Using inequality \eqref{HLSD-3}, the above inequality yields
\begin{eqnarray*}
\|(u_{x,\lambda}-u)_+\|_{L^k(\partial\mathbb{R}^{n}_+)}
&\le&c\|v_{x,\lambda}^{\beta-1}(v_{x,\lambda}-v)_+\|_{L^t(\mathbb{R}^{n}_+)},
\end{eqnarray*}
where $f_+(x)=\max(f(x), 0)$.  Note $\beta-1= 2\alpha/(n-\alpha)$. By H\"{o}lder inequality, we have
\begin{eqnarray*}
\|v_{x,\lambda}^{\beta-1}(v_{x,\lambda}-v)\|_{L^t(\Sigma_{x,\lambda}^v)}
\le\|v_{x,\lambda}\|^{\beta-1}_{L^{\frac{2n}{n-\alpha}}(\Sigma_{x,\lambda}^v)} \|v_{x,\lambda}-v\|_{L^s(\Sigma_{x,\lambda}^v)}.
\end{eqnarray*}
Thus
\begin{eqnarray}\label{K-5}
\|u_{x,\lambda}-u\|_{L^k(\Sigma_{x,\lambda}^{u})} \le
c\|v_{x,\lambda}\|^{\beta-1}_{L^{\frac{2n}{n-\alpha}}(\Sigma_{x,\lambda}^v)} \|v_{x,\lambda}-v\|_{L^s(\Sigma_{x,\lambda}^v)}.
\end{eqnarray}
On the other hand, for any $\eta\in\Sigma_{x,\lambda}^{v}$, we know from \eqref{K-4} that
\begin{eqnarray*}
v_{x,\lambda}(\eta)-v(\eta) &\le&\gamma\int_{\Sigma_{x,\lambda}^u}\frac{u_{x,\lambda}^{\gamma-1}(\xi)(u_{x,\lambda}(\xi)-u(\xi))
}{|\eta-\xi|^{n-\alpha}}d\xi,
\end{eqnarray*}
Note for $1<t<\frac n\alpha$, $s=\frac {nt}{n-\alpha t}>\frac{n}{n-\alpha}$ and $\frac{(n-1)s}{n+(\alpha-1)s}=\frac{(n-1)t}{n-t}$.
Similarly, using HLS inequality \eqref{HLSD-2} and H\"{o}lder inequality, we have
 \begin{eqnarray}\label{K-6}
\|v_{x,\lambda}-v\|_{L^s(\Sigma_{x,\lambda}^v)}
&\le&c\|u_{x,\lambda}^{\gamma-1}(u_{x,\lambda}-u)\|_{L^{\frac{(n-1)s}{n+(\alpha-1)s}}
(\Sigma_{x,\lambda}^u)}\nonumber\\
&\le&c\|u_{x,\lambda}\|^{\gamma-1}_{L^{\frac{2(n-1)}{n-\alpha}} (\Sigma_{x,\lambda}^u)}\|u_{x,\lambda}-u\|_{L^{k}
(\Sigma_{x,\lambda}^u)}.
\end{eqnarray}
 Combining \eqref{K-5} with \eqref{K-6}, we obtain
\begin{eqnarray}
& & \|u_{x,\lambda}-u\|_{L^k(\Sigma_{x,\lambda}^{u})}\nonumber\\
&\le&c\|v_{x,\lambda}\|^{\beta-1}_{L^{\frac{2n}{n-\alpha}}(\Sigma_{x,\lambda}^v)}
 \|u_{x,\lambda}\|^{\gamma-1}_{L^{\frac{2(n-1)}{n-\alpha}}
(\Sigma_{x,\lambda}^u)}\|u_{x,\lambda}-u\|_{L^{k}
(\Sigma_{x,\lambda}^u)}\nonumber\\
&\le&c\|v_{x,\lambda}\|^{\beta-1}_{L^{\frac{2n}{n-\alpha}}(\Sigma_{x,\lambda}^n)}
 \|u_{x,\lambda}\|^{\gamma-1}_{L^{\frac{2(n-1)}{n-\alpha}}
(\Sigma_{x,\lambda}^{n-1})}\|u_{x,\lambda}-u\|_{L^{k}
(\Sigma_{x,\lambda}^u)}\nonumber\\
&=&c\|v\|^{\beta-1}_{L^{\frac{2n}{n-\alpha}}(B^+_{\lambda}(x))}
 \|u\|^{\gamma-1}_{L^{\frac{2(n-1)}{n-\alpha}}
(B_\lambda^{n-1}(x))}\|u_{x,\lambda}-u\|_{L^{k} (\Sigma_{x,\lambda}^u)}.~~~~~~~~~~~  ~~~~~~~~~~~~\label{K-7}
\end{eqnarray}

Similarly, we have
\begin{eqnarray}
\|v_{x,\lambda}-v\|_{L^s(\Sigma_{x,\lambda}^v)} &\le&c\|v\|^{\beta-1}_{L^{\frac{2n}{n-\alpha}}(B^+_{\lambda}(x))}
 \|u\|^{\gamma-1}_{L^{\frac{2(n-1)}{n-\alpha}}
(B_\lambda^{n-1}(x))}\|v_{x,\lambda}-v\|_{L^s(\Sigma_{x,\lambda}^v)}.\nonumber\\~~~~~~~~~~~ ~~~~~~~~~~~~\label{K-8}
\end{eqnarray}
Since $u\in L^{\frac{2(n-1)}{n-\alpha}}_{loc}(\partial\mathbb{R}^{n}_+)$ and $v\in L^{\frac{2n}{n-\alpha}}_{loc}
(\mathbb{R}^{n}_+),$ we can choose $\lambda_0$ small enough such that for $0<\lambda<\lambda_0$, we have
\[c\|v\|^{\beta-1}_{L^{\frac{2n}{n-\alpha}}(B^+_{\lambda}(x))}
 \|u\|^{\gamma-1}_{L^{\frac{2(n-1)}{n-\alpha}}
(B_\lambda^{n-1}(x))}\le\frac12.
\]
%{\bf note, we assume $u, v$ are smooth}
 Combining the above with \eqref{K-7} and \eqref{K-8}, we get
\begin{eqnarray*}
\|u_{x,\lambda}-u\|_{L^k(\Sigma_{x,\lambda}^{u})}&\le&\frac12\|u_{x,\lambda}-u\|_{L^k(\Sigma_{x,\lambda}^{u})},\\
\|v_{x,\lambda}-v\|_{L^s(\Sigma_{x,\lambda}^v)} &\le&\frac12\|v_{x,\lambda}-v\|_{L^s(\Sigma_{x,\lambda}^v)},
\end{eqnarray*}
which imply that $\|u_{x,\lambda}-u\|_{L^k(\Sigma_{x,\lambda}^{u})}
=\|v_{x,\lambda}-v\|_{L^s(\Sigma_{x,\lambda}^v)}=0$.
%Since $u, \ v$ are continuous, we know t
That is, the measures for both $\Sigma_{x,\lambda}^{u}$ and $\Sigma_{x,\lambda}^v$ are zero. We complete the proof of the lemma.
\end{proof}

\medskip

Define, module sets with zero measure,
\[\bar{\lambda}(x)=sup\{\mu>0\,|\,u_{x,\lambda}(\xi)\le u (\xi), \mbox{and} \ v_{x,\lambda}(\eta)\le v
(\eta), \  \forall \lambda\in (0, \mu), \forall \xi\in\Sigma_{x,\lambda}^{n-1}, \forall \eta\in \Sigma_{x,\lambda}^{n}\}.
\]
We will show: if the sphere stops, then we have symmetric properties for solutions.

\begin{lm}\label{K-lm-2}For some $x_0\in
\partial\mathbb{R}^{n}_+$,
if $\bar{\lambda}(x_0)<\infty$, then
\begin{eqnarray*}
u_{x_0,\bar{\lambda}(x_0)}&=&u\quad \text{on}~\partial\mathbb{R}^{n}_+,\\
v_{x_0,\bar{\lambda}(x_0)}&=&v\quad \text{on}~\mathbb{R}^{n}_+.
\end{eqnarray*}
\end{lm}
\begin{proof} \ Denote $\bar{\lambda}=\bar{\lambda}(x_0).$
We only need to show
\begin{eqnarray*}
u_{x_0,\bar{\lambda}}(\xi)&=&u(\xi)\quad
\text{for \ all }~\xi\in\Sigma_{x_0,\bar{\lambda}}^{n-1},\\
v_{x_0,\bar{\lambda}}(\eta)&=&v(\eta)\quad \text{for \ all}~\eta\in \Sigma_{x_0,\bar{\lambda}}^{n}.
\end{eqnarray*}
By the definition of $\bar{\lambda}$, we have
\begin{eqnarray*}
u_{x_0,\bar{\lambda}}(\xi)&\le&u(\xi)\quad
\text{for \ all}~\xi\in\Sigma_{x_0,\bar{\lambda}}^{n-1},\\
v_{x_0,\bar{\lambda}}(\eta)&\le&v(\eta)\quad \text{for \ all}~\eta\in \Sigma_{x_0,\bar{\lambda}}^{n}.
\end{eqnarray*}
If $u_{x_0,\bar{\lambda}}(\xi)\not\equiv u(\xi)$ or $v_{x_0,\bar{\lambda}}(\eta)\not\equiv v(\eta)$, we know from \eqref{K-3} and
\eqref{K-4} (using $x_0,\bar{\lambda}$ to replace $x, \lambda$), that
\begin{eqnarray*}
u_{x_0,\bar{\lambda}}(\xi)&<&u(\xi)\quad
\text{for \ all}~\xi\in\Sigma_{x_0,\bar{\lambda}}^{n-1},\\
v_{x_0,\bar{\lambda}}(\eta)&<&v(\eta)\quad \text{for \ all}~\eta\in \Sigma_{x_0,\bar{\lambda}}^{n}.
\end{eqnarray*}
%Since  $u\in L^{\frac{2(n-1)}{n-\alpha}}_{loc}(\partial\mathbb{R}^{n}_+)$ and $v\in L^{\frac{2n}{n-\alpha}}_{loc}
%(\mathbb{R}^{n}_+),$ we can choose $R(>\bar{\lambda})$ sufficiently large so that
%\begin{eqnarray}\label{K-3R}
%\int_{\partial\mathbb{R}^n_+\cap B^{n-1}_R(x_0)}u^\frac{2(n-1)}{n-\alpha}(\xi)d\xi<\varepsilon_0,\quad \int_{\mathbb{R}^n_+\cap
%B^{+}_R(x_0)}v^\frac{2n}{n-\alpha}(\eta)d\eta<\varepsilon_0,
%\end{eqnarray}
%where $\varepsilon_0$ is small enough.
%%{\bf we assume $u, v$ are smooth}
Thus, for a given large $R$, $\varepsilon_0$ and any $\delta>0$ there exist $c_1,c_2$ such that
 \begin{eqnarray*}
u(\xi)-u_{x_0,\bar{\lambda}}(\xi)&>&c_1\quad
\text{for}~\xi\in\Sigma_{x_0,\bar{\lambda}+\delta}^{n-1}\cap B^{n-1}_R(x_0),\\
v(\eta)-v_{x_0,\bar{\lambda}}(\eta)&>&c_2\quad \text{for}~\eta\in \Sigma_{x_0,\bar{\lambda}+\delta}^{n}\cap B^{+}_R(x_0).
\end{eqnarray*}
By (\ref{K-3}) and (\ref{K-4}), we know that we can choose $\varepsilon<\delta$ sufficiently small so
that for $\lambda\in[\bar{\lambda},\bar{\lambda}+\varepsilon)$,

\begin{eqnarray*}
u(\xi)&\ge& u_{x_0,\lambda}(\xi)\quad
\text{for}~\xi\in\Sigma_{x_0,\bar{\lambda}+\delta}^{n-1}\cap B^{n-1}_R(x_0),\\
v(\eta)&\ge& v_{x_0,\lambda}(\eta)\quad \text{for}~\eta\in \Sigma_{x_0,\bar{\lambda}+\delta}^{n}\cap B^{+}_R(x_0).
\end{eqnarray*}
 These imply that  $\Sigma_{x_0,\lambda}^u$ and $\Sigma_{x_0,\lambda}^v$ have no
intersection with
\[\Sigma_{x_0,\bar{\lambda}+\delta}^{n-1}\cap B^{n-1}_R(x_0) \quad\text{and}\quad
\Sigma_{x_0,\bar{\lambda}+\delta}^{n}\cap B^{+}_R(x_0),
\]respectively.
Thus, $\Sigma_{x_0,\lambda}^u$ is contained in the union of
\[\partial\mathbb{R}^n_+\setminus B^{n-1}_R(x_0)\quad
\text{and}\quad\Sigma_{x_0,\lambda}^{n-1}\backslash\Sigma_{x_0,\bar{\lambda}+\delta}^{n-1};
\] And
$\Sigma_{x_0,\lambda}^v$ is contained in the union of
\[\mathbb{R}^n_+\setminus B^{+}_R(x_0)\quad \text{and}\ \Sigma_{x_0,\lambda}^{n}\backslash\Sigma_{x_0,\bar{\lambda}+\delta}^{n}.
\]

For simplicity, we write $\Omega_{\lambda,R}^{n-1}=(\partial\mathbb{R}^n_+\setminus
B^{n-1}_R(x_0))\cup(\Sigma_{x_0,\lambda}^{n-1}\backslash\Sigma_{x_0,\bar{\lambda}+\delta}^{n-1}),
\Omega_{\lambda,R}^{n}=(\mathbb{R}^n_+\setminus
B^{+}_R(x_0))\cup(\Sigma_{x_0,\lambda}^{n}\backslash\Sigma_{x_0,\bar{\lambda}+\delta}^{n}).$ Moreover, denote
$(\Omega_{\lambda,R}^{n})^*$ and $(\Omega_{\lambda,R}^{n-1})^*$ as the reflection of $\Omega_{\lambda,R}^{n}$ and
$\Omega_{\lambda,R}^{n-1}$ under the Kelvin transformation with respect to  the sphere $\{x: |x-x_0|=\lambda\},$ respectively. That is,
$(\Omega_{\lambda,R}^{n})^*=B_{\varepsilon_1}^+(x_0)\cup(B_{\lambda}^+(x_0)\backslash B_{\lambda^2/(\bar \lambda +\delta)}^+(x_0)),$  and
$(\Omega_{\lambda,R}^{n-1})^*=B_{\varepsilon_1}^{n-1}(x_0)\cup(B_{\lambda}^{n-1}(x_0)\backslash B_{\lambda^2/(\bar \lambda +\delta)}^{n-1}(x_0))$,
where $\varepsilon_1=\lambda/R$ is small as $R\to \infty.$

Similar to \eqref{K-7} and \eqref{K-8}, for $\lambda\in[\bar{\lambda},\bar{\lambda}+\varepsilon),$ we have
\begin{eqnarray}
\|u_{x,\lambda}-u\|_{L^k(\Sigma_{x_0,\lambda}^{u})}
&\le&c\|v_{x,\lambda}\|^{\beta-1}_{L^{\frac{2n}{n-\alpha}}(\Sigma_{x_0,\lambda}^v)}
 \|u_{x,\lambda}\|^{\gamma-1}_{L^{\frac{2(n-1)}{n-\alpha}}
(\Sigma_{x_0,\lambda}^{u})}\nonumber\\
& &\times\|u_{x,\lambda}-u\|_{L^{k} (\Sigma_{x_0,\lambda}^u)},\nonumber\\
&\le&c\|v\|^{\beta-1}_{L^{\frac{2n}{n-\alpha}}((\Omega_{\lambda,R}^{n})^*)}
 \|u\|^{\gamma-1}_{L^{\frac{2(n-1)}{n-\alpha}}
((\Omega_{\lambda,R}^{n-1})^*)}\nonumber\\
& &\times\|u_{x,\lambda}-u\|_{L^{k} (\Sigma_{x_0,\lambda}^u)},
~~~~~~~~~~~  ~~~~~~~~~~~~\label{K-9}\\
\|v_{x,\lambda}-v\|_{L^s(\Sigma_{x_0,\lambda}^v)} &\le&c\|v\|^{\beta-1}_{L^{\frac{2n}{n-\alpha}}((\Omega_{\lambda,R}^{n})^*)}
 \|u\|^{\gamma-1}_{L^{\frac{2(n-1)}{n-\alpha}}
((\Omega_{\lambda,R}^{n-1})^*)}
\nonumber\\
& &\times\|v_{x,\lambda}-v\|_{L^s(\Sigma_{x_0,\lambda}^v)}, ~~~~~~~~~~~ ~~~~~~~~~~~~\label{K-10}
\end{eqnarray}
%From \eqref{K-3R}, and Kelvin transformation, for some $\delta$ small enough,
Since $u\in
L^{\frac{2(n-1)}{n-\alpha}}_{loc}(\partial\mathbb{R}^{n}_+)$ and $v\in L^{\frac{2n}{n-\alpha}}_{loc} (\mathbb{R}^{n}_+),$
%{\bf replace by smoothness assumption}  3-3-2012
we have
\begin{eqnarray*}
\int_{(\Omega_{\lambda,R}^{n-1})^*}u^\frac{2(n-1)}{n-\alpha}(\xi)d\xi<\varepsilon_0 (\varepsilon, \delta)\,\quad
\int_{(\Omega_{\lambda,R}^{n})^*}v^\frac{2n}{n-\alpha}(\eta)d\eta<\varepsilon_0(\varepsilon, \delta).
\end{eqnarray*}
%Moreover, since  $u\in L^{\frac{2(n-1)}{n-\alpha}}_{loc}(\partial\mathbb{R}^{n}_+)$ and $v\in L^{\frac{2n}{n-\alpha}}_{loc}
%(\mathbb{R}^{n}_+),$
%{\bf replace by smoothness assumption} 3-3-2012
Choose $\varepsilon_0$ small enough (via choosing $\varepsilon, \, \delta$ small enough) such that for $
\lambda\in[\bar{\lambda},\bar{\lambda}+\varepsilon) $,
\[c\|v\|^{\beta-1}_{L^{\frac{2n}{n-\alpha}}((\Omega_{\lambda,R}^{n-1})^*)}
 \|u\|^{\gamma-1}_{L^{\frac{2(n-1)}{n-\alpha}}
((\Omega_{\lambda,R}^{n-1})^*)}<\frac12.
\]
Substituting the above into \eqref{K-9} and \eqref{K-10}, we obtain
\[\|u_{x,\lambda}-u\|_{L^k(\Sigma_{x_0,\lambda}^{u})}=\|v_{x,\lambda}-v\|_{L^s(\Sigma_{x_0,\lambda}^v)}=0.
\]
 %This implies that $\Sigma_{x_0,\lambda}^{u}$ and
% $\Sigma_{x_0,\lambda}^v$ must be empty.
 Thus we conclude that
 \begin{eqnarray*}
u_{x_0, {\lambda}}(\xi)&\le&u(\xi)\quad
\text{for \ all}~\xi\in\Sigma_{x_0, {\lambda}}^{n-1},\\
v_{x_0, {\lambda}}(\eta)&\le&v(\eta)\quad \text{for \ all}~\eta\in \Sigma_{x_0, {\lambda}}^{n}
\end{eqnarray*}for  $
\lambda\in[\bar{\lambda},\bar{\lambda}+\varepsilon) $,  which contradicts the definition of $\bar{\lambda}$.
\end{proof}

\medskip

The following three calculus key lemmas are needed for carrying out moving sphere procedure. Under  stronger assumptions ($f\in
C^1(\mathbb{R}^n_+)$), these lemmas were early proved by Li and Zhu \cite{LZ1995}, and  Li and Zhang \cite{LZ2003}. The first two lemmas, due to Li and Nirenberg,  are adopted from Li \cite{Li2004}.
\begin{lm}\label{Li-1} (Lemma~$5.7$~ in~\cite{Li2004}) For $n\ge 1$ and $\mu \in \mathbb{R}$,
if $f$ is a function defined on $\mathbb{R}^n$ and valued in $(-\infty,+\infty)$ satisfying
\[
\big(\frac\lambda{|y-x|}\big)^\mu f\big(x+\frac{\lambda^2(y-x)}{|y-x|^2}\big)\le f(y),\quad\forall\lambda>0, ~
|y-x|\ge\lambda,x,y\in\mathbb{R}^n,
\]
then $f(x)=$ constant.
\end{lm}
\begin{lm}\label{Li-2} (Lemma~$5.8$~ in~\cite{Li2004}) Let $n\ge 1$ and $\mu \in \mathbb{R}$,
 and $f\in C^0(\mathbb{R}^n).$  Suppose that for every $x\in\mathbb{R}^n$, there exists $\lambda>0$ such that
\[
\big(\frac\lambda{|y-x|}\big)^\mu f\big(x+\frac{\lambda^2(y-x)}{|y-x|^2}\big)= f(y), \ \ \quad ~ \forall
y\in\mathbb{R}^n\setminus\{x\}.
\] Then there are $a\ge0,d>0$ and $\bar{x}\in\mathbb{R}^n$, such that
\[
f(x)\equiv\pm a\big(\frac1 {d+|x-\bar{x}|^2}\big)^\frac{\mu}2.
\]
\end{lm}

\begin{lm}\label{LZ-2} For $n\ge 1$ and $\mu \in \mathbb{R}$,
if $f$ is a function defined on $\mathbb{R}^n_+$ and valued in $(-\infty,+\infty)$ satisfying
\[
\big(\frac\lambda{|y-x|}\big)^\mu f\big(x+\frac{\lambda^2(y-x)}{|y-x|^2}\big)\le f(y),\quad\forall\lambda>0, ~
|y-x|\ge\lambda,y\in\mathbb{R}^n_+,x\in\partial\mathbb{R}^n_+,
\]
then
\[f(z)=f(z',t)=f(0,t),  ~~~~~~~~~~~~~~~~~~~~~\forall\,z=(z',t)\in\mathbb{R}^n_+.
\]
\end{lm}
\begin{proof} The proof of this lemma is similar to that of Lemma 7.1 in \cite{Li2004}.
For any $z=(z',z_n)\in \mathbb{R}^n_+$, choose $y^i=(y',y^i_n)$ with $y'\ne z'$,  $y^i_n>z_n$ and $y^i_n \to z_n$ as $i\to \infty$.
Choose $b^i>1,$ so that
\[x^i:=(x'^i,0)=x^i(b^i)=y^i+b^i(z-y^i)\in\partial\mathbb{R}^n_+.\]
Also define
\[\lambda_i:=\lambda_i(b^i)=\sqrt{|z-x^i||y^i-x^i|}.
\]
Then,
\[z=x^i+\frac{\lambda_i^2(y^i-x^i)}{|y^i-x^i|^2},
\] and by the assumption of Lemma \ref{LZ-2},
\[
\big(\frac{\lambda_i}{|y^i-x^i|}\big)^\mu f(z)\le f(y^i).
\]
Since
\begin{eqnarray*}
\lim_{i\to\infty}\frac{\lambda_i}{|y^i-x^i|}=\lim_{i\to\infty}\sqrt{\frac{|z-x^i|}{|y^i-x^i|}}=1, \ \ \mbox{and} \ \
\lim_{i\to\infty}y^i=(y', z_n),
\end{eqnarray*} we obtain $f(z',z_n)\le f(y', z_n)$. Since $y'$ and $z'$ are arbitrary, we have the lemma.

\end{proof}

\medskip

\textbf{Proof of Theorem \ref{sys-critical}.}

{\bf Case 1.} If there exists some $x_0\in \partial\mathbb{R}^{n}_+$ such that $\bar{\lambda}(x_0)<\infty,$ then
$\bar{\lambda}(x)<\infty$ for all $x\in\partial\mathbb{R}^{n}_+$.

For any $x \in\partial\mathbb{R}^{n}_+,$ from the definition of $\bar{\lambda}(x)$, we know $\forall \, \lambda\in (0, \bar{\lambda}(x)),$
 \[u_{x,\lambda}(\xi)\le u(\xi),\quad\quad \forall \xi\in\Sigma_{x, {\lambda}}^{n-1}.
 \]
It implies
\begin{equation}\label{K-11}
a:=\liminf_{|\xi|\to\infty}\big(|\xi|^{n-\alpha}u(\xi)\big)
\ge\liminf_{|\xi|\to\infty}\big(|\xi|^{n-\alpha}u_{x,\lambda}(\xi)\big)=\lambda^{n-\alpha}u(x),~~~~~~~~~~~\forall\,
\lambda\in(0, \bar{\lambda}(x)).
\end{equation}
On the other hand, since $\bar{\lambda}(x_0)<\infty$, using Lemma \ref{K-lm-2} we have
\begin{equation}\label{K-12}
a=\liminf_{|\xi|\to\infty}\big(|\xi|^{n-\alpha}u(\xi)\big)
=\liminf_{|\xi|\to\infty}\big(|\xi|^{n-\alpha}u_{x_0,\bar{\lambda}}(\xi)\big)=\bar{\lambda}^{n-\alpha}u(x_0)<\infty.
\end{equation}
 Combining \eqref{K-11} with \eqref{K-12} we obtain $\bar{\lambda}(x)<\infty$ for all $x\in\partial\mathbb{R}^{n}_+$.
 Applying Lemma \ref{K-lm-2} again, we know
 \[u_{x,\bar{\lambda}}(\xi)=u(\xi),~~~~~~~~~~~ ~~~~~~~~~~~\forall\, x,\xi\in \partial\mathbb{R}^{n}_+.
 \]
 By Lemma \ref{Li-2}, we have: for all $x \in\partial\mathbb{R}^{n}_+$,
 \begin{equation*}
 u(\xi)=c_1\big(\frac{1}{|\xi-\xi_0|^2+d^2}\big)^\frac{n-\alpha}2
\end{equation*} for some $c_1,\, d>0$ and $ \xi_0\in\partial\mathbb{R}^{n}_+.$

 Bringing the above into the second equation in system \eqref{sys-1}, we can show,  for $\xi\in\partial\mathbb{R}^{n}_+$, that
 \begin{equation*}
 v(\xi,0)=c_2\big(\frac{1}{|\xi-\xi_0|^2+d^2}\big)^\frac{n-\alpha}2
\end{equation*} for some $c_2,\, d>0$ and $ \xi_0\in\partial\mathbb{R}^{n}_+.$ More computational details can be found in the proof of Lemma 6.1 in Li \cite{Li2004}.
%{\bf NOT convinced here yet 3-16-2012 From system \eqref{sys-1}, it is easy to verify
%\[c_1^\frac{n+\alpha-2}2 C_e(n,\alpha,p)=c_2^{\frac{n-\alpha}2}.
%\]}

{\bf Case 2.}  $\bar{\lambda}(x)=\infty $ for all $x \in\partial\mathbb{R}^{n}_+$.

 Then for any given $x\in
\partial\mathbb{R}^{n}_+,$
\[u_{x,\lambda}(\xi)\le u(\xi),~~~~~~~~~~~~~~~~~~~~~~\forall\, \xi\in\Sigma_{x,
{\lambda}}^{n-1}.
\]
Using Lemma \ref{Li-1}, we conclude that $u=C_0$  is a constant.

On the other hand, for any $x\in
\partial\mathbb{R}^{n}_+,$
\[v_{x,\lambda}(\eta)\le v(\eta),~~~~~~~~~~~~~~~~~~~~~~\forall\, \eta\in\Sigma_{x,
{\lambda}}^{n}.
\]
From Lemma \ref{LZ-2},   we conclude that $v$ only depends on $t$. Thus, we have
\begin{equation*}
v(\eta',t)=v(0,t)=\int_{\partial\mathbb{R}^{n}_+}\frac{C_0^\gamma}{(|y|^2+t^2)^{\frac{n-\alpha}{2}}} dy,
\end{equation*} However,
\begin{eqnarray*}
v(0,t)&=&\int_{\partial\mathbb{R}^{n}_+}\frac{C_0^\gamma}{(|y|^2+t^2)^{\frac{n-\alpha}{2}}} dy\\
%&=&t^{\alpha-1}\int_{\partial\mathbb{R}^{n}_+}\frac{C_0^\gamma}{(|\xi|^2+1)^{\frac{n-\alpha}{2}}} d\xi~~~
%~~~~~~(y=t\xi)\\
&=&C t^{\alpha-1}\int_{0}^{\infty}\frac{\rho^{n-2}}{(\rho^2+1)^{\frac{n-\alpha}{2}}}d\rho.
\end{eqnarray*}
Since $\alpha>1$, we conclude that $v(0,t)$  is infinite for $t\ne0$.  Contradiction.\hfill$\Box$

\begin{rem}\label{Re-R}
It is interesting to point out that Theorem \ref{sys-critical} can also be directly proved from the fact that $\bar \lambda (x) <\infty $ for all $x\in\partial\mathbb{R}^{n}_+$ via Theorem 1.4 in \cite{FL2010} without using the $C^0$ regularity of solutions.
\end{rem}

\section{\textbf{Regularity of solutions to integral equation} \label{Section 4}}
In this section, we address the regularity properties of solutions to integral equation \eqref{equ-1}.

\begin{thm} \label{Regularity-1}
Let $ 1<\alpha<n$ and $1<p<(n-1)/(\alpha-1)$.  Suppose that $f\in L^p_{loc}(\partial\mathbb{R}^{n}_+)$ is nonnegative solution to  \eqref{equ-1} with $\frac1q=\frac{n-1}n\big(\frac1p-\frac{\alpha-1}{n-1})$. Then $f\in
C^\infty(\partial\mathbb{R}^{n}_+).$
%Moreover, if $f\in
%L^p(\partial\mathbb{R}^{n}_+)$, then $f(y)\to0$ as $|y|\to\infty.$
\end{thm}
%
%To this end, let $f$ be a nonnegative function satisfying \eqref{equ-1}, and
%define $u(y)=f^{p-1}(y),v(x)=E_\alpha f(x)$, and $\theta=\frac1{p-1},\kappa=q-1,$ then the single equation become an integral system
%\begin{equation}\label{sys-1}
%\begin{cases}
%u(y)=\int_{\mathbb{R}^{n}_+}\frac{v^{\kappa}(x)}{|x-y|^{n-\alpha}}dx,&\quad y\in\partial\mathbb{R}^{n}_+,\\
%v(x)=\int_{\partial\mathbb{R}^{n}_+}\frac{u^{\theta}(y)}{|x-y|^{n-\alpha}}dy,&\quad x\in\mathbb{R}_+^{n}
%\end{cases}
%\end{equation}with then $\frac1{\kappa+1}=\frac{n-1}{n}\big( \frac{n-\alpha}{n-1}-\frac1{\theta+1}\big).$
%Moreover, If $f\in L^p_{loc}(\partial\mathbb{R}^{n}_+)$, then $u\in L^{\theta+1}_{loc}(\partial\mathbb{R}^{n}_+).$

Theorem \ref{Regularity-1} is equivalent to the following.
\begin{thm}\label{Regularity-2}
Assume $1<\alpha<n, $ $\frac{\alpha -1}{n-\alpha}<\theta<\infty$ and $0<\kappa<\infty$ given by
\begin{equation}\label{k-t}
\frac1{\kappa+1}=\frac{n-1}{n}\big(\frac{n-\alpha}{n-1}-\frac1{\theta+1}\big).
 \end{equation} If (u,v) is a pair of positive solutions of \eqref{sys-1} with $u\in
L^{\theta+1}_{loc}(\partial\mathbb{R}^{n}_+),$ then $u\in C^\infty(\partial\mathbb{R}^{n}_+)$ and $v\in
C^\infty(\overline{\mathbb{R}^{n}_+}).$
%Moreover, $u\in
%L^{\theta+1}(\partial\mathbb{R}^{n}_+),$ then $u(y)\to0$ as $|y|\to
%\infty,$ and $v(x)\to0$ as $|x|\to \infty.$
\end{thm}

To prove this theorem, we first establish two local regularity results, which are similar to Lemma A.1 in \cite{BK1979},  Theorem
$1.3$ in \cite{Li2004} and Proposition $5.2$ and $5.3$ in \cite{HWY2008}, and quite similar in spirit to  Brezis and Kato's Lemma
\cite{BK1979} and the regularity result in \cite{BR1992}.

We use the notations from previous section. And later in this section, for $x=0$, we write
$B_R=B_R(0),B_R^{n-1}=B_R^{n-1}(0),B_R^+=B_R^+(0),\Sigma_{R}^n=\Sigma_{0,R}^n$ and $ \Sigma_{R}^{n-1}=\Sigma_{0,R}^{n-1}$.

%For $R>0, x\in\mathbb{R}^n,$ denote
%$$B_R(x)=\{y\in\mathbb{R}^n\,|\,|y-x|<R\},B_R^{n-1}(x)=\{y\in\partial\mathbb{R}^{n}_+\,|\,|y-x|<R\},$$
%$$B_R^+(x)=\{y=(y_1,y_2,\cdots,y_n)\in B_R(x)\,|\, y_n>0\},$$
%$$\Sigma_{x,R}^n=\mathbb{R}^n_+\backslash\overline{B_R^+(x)},
%\Sigma_{x,R}^{n-1}=\partial\mathbb{R}^{n}_+\backslash\overline{B_R^{n-1}(x)}.$$
%Suppose $x=0,$ write $B_R=B_R(0),B_R^{n-1}=B_R^{n-1}(0),B_R^+=B_R^+(0),\Sigma_{R}^n=\Sigma_{0,R}^n,
%\Sigma_{R}^{n-1}=\Sigma_{0,R}^{n-1}.$
 \begin{prop}\label{Reg-prop-1}
 Assume $1<\alpha<n,\,1<a,b\le\infty,1\le r<\infty$ and $\frac{n}{n-\alpha}<p<q<\infty $ satisfy
 \begin{equation}\label{cond-1}\frac\alpha n<\frac rq+\frac1a<\frac rp+\frac1a<1\quad\text{and}\quad\frac{n}{ar}+\frac{n-1}b=\frac\alpha r+(\alpha-1).
 \end{equation}
 Suppose that  $v,h\in L^{p}(B^+_R), V\in L^{a}(B^+_R), \mbox{and} \, \ ~ U\in L^{b}(B^{n-1}_R)$ are
 all nonnegative functions with $h|_{B^+_{R/2}}\in L^{q}(B^+_{R/2})$,
and
 \[v(x)\le\int_{B^{n-1}_R}\frac{U(y)}{|x-y|^{n-\alpha}}
 \big[\int_{B^{+}_R}\frac{V(z)v^r(z)}{|z-y|^{n-\alpha}}dz\big]^\frac1rdy+h(x),\quad\quad \forall x\in B^{+}_R.
 \]
  There is a $\varepsilon=\varepsilon(n,\alpha,p,q,r,a,b)>0$, and $C(n,\alpha,p,q,a,b,r,\varepsilon)>0$ such
 that if
 \[\|U\|_{L^{b}(B^{n-1}_R)}\|V\|^\frac1r_{L^{a}(B^+_R)}\le \varepsilon(n,\alpha,p,q,r,a,b),
 \]
 then
 \[\|v\|_{L^{q}(B^+_{R/4})}\le C(n,\alpha,p,q,a,b,r,\varepsilon)\big(R^{\frac{n}{q}-\frac{n}{p}}\|v\|_{L^{p}(B^+_{R})}+\|h\|_{L^{q}(B^+_{R/2})}).
 \]
 \end{prop}
\begin{proof} After rescaling, we may assume $R=1$.

We first consider the case that  $v, \, ~ h \in L^{q}(B^+_{1})$. Denote
\[u(y)=\int_{B^{+}_1}\frac{V(x)v^r(x)}{|x-y|^{n-\alpha}}dx\quad\text{for}\quad y\in B^{n-1}_1,
\]
and  define
\begin{equation*}
\psi(y)=\begin{cases}u(y)\quad &y\in B^{n-1}_1,\\
0\quad &y\not\in  B^{n-1}_1,
\end{cases}
\end{equation*}and
\begin{equation*}
\phi(x)=\begin{cases}v(x)\quad &x\in B^{+}_1,\\
0\quad &x\not\in  B^{+}_1.
\end{cases}
\end{equation*}
For any $s_1, s_2 \in (1, n/\alpha)$, we know from inequality \eqref{HLSD-3}, that
\begin{eqnarray*}
\|\psi\|_{L^{p_1}(\partial \mathbb{R}^n_+)}\le c(n,\alpha,s_1,r)\|V\phi^r\|_{L^{s_1}(\mathbb{R}^n_+)},\\
\|\psi\|_{L^{q_1}(\partial\mathbb{R}^n_+)}\le c(n,\alpha,s_2,r)\|V\phi^r\|_{L^{s_2}(\mathbb{R}^n_+)},
\end{eqnarray*} where $$\frac1{p_1}=\frac n{n-1}\big(\frac1{s_1}-\frac\alpha n\big),\quad
\frac1{q_1}=\frac n{n-1}\big(\frac1{s_2}-\frac\alpha n\big).$$
 In particular, if we choose $s_1, \, s_2$ so that $\frac1{s_1}=\frac rp+\frac1a, \frac1{s_2}=\frac rq+\frac1a$,
we have, via  H\"{o}lder inequality, that
\begin{eqnarray}
\|u\|_{L^{p_1}(B^{n-1}_1)}&\le& c(n,\alpha,p,a,r)\|V\|_{L^a(B^{+}_1)}\|v\|^r_{L^p(B^{+}_1)},\label{R-1}\\
\|u\|_{L^{q_1}(B^{n-1}_1)}&\le& c(n,\alpha,q,a,r)\|V\|_{L^a(B^{+}_1)}\|v\|^r_{L^q(B^{+}_1)},\label{R-2}
\end{eqnarray}
where $p_1$ and $q_1$ satisfy
\begin{eqnarray*}
\frac1{p_1}=\frac n{n-1}\big(\frac rp +\frac1a-\frac\alpha n\big),\quad
\frac1{q_1}=\frac n{n-1}\big(\frac rq+\frac1a-\frac\alpha n\big).
\end{eqnarray*}
The existence of $p_1, \, q_1$ is guaranteed by \eqref{cond-1}.
Let $0<\varrho<\delta\le\frac12.$ For $x\in B^{+}_\delta$, we have
\begin{eqnarray*}
v(x)&\le&\int_{B^{n-1}_{\frac{\delta+\varrho}2}}\frac{U(y)u^\frac1r(y)}{|x-y|^{n-\alpha}}dy
+\int_{B^{n-1}_1\backslash B^{n-1}_{\frac{\delta+\varrho}2}}\frac{U(y)u^\frac1r(y)}{|x-y|^{n-\alpha}}dy+h(x)\\
&=:& J_1(x)+J_2(x)+h(x).
\end{eqnarray*}
From \eqref{cond-1}, we know $\frac{1}q=\frac{n-1}n\big(\frac1b+\frac1{q_1r}-\frac{\alpha-1}{n-1}\big).$ Using inequality \eqref{HLSD-2} and H\"{o}lder
inequality, we have
\begin{eqnarray*}
 \|J_1\|_{L^q(B^{+}_{\varrho})}\le c(n,\alpha,q,r,b)\|U\|_{L^b(B^{n-1}_{1})}\|u\|^\frac1r_{L^{q_1}(B^{n-1}_{\frac{\delta+\varrho}2})}.
\end{eqnarray*}
Let $\frac1{m_1}=1-\frac1b-\frac1{p_1r}.$ Since $p>n/(n-\alpha)$, we know $m_1>1$. From H\"{o}lder inequality and \eqref{R-1}, we know
\begin{eqnarray*}
J_2(x)&\le& \frac{c(n,\alpha)}{(\delta-\varrho)^{n-\alpha}}\int_{B^{n-1}_1\backslash B^{n-1}_{\frac{\delta+\varrho}2}} {U(y)u^\frac1r(y)} dy\\
&\le& \frac{c(n,\alpha,p,r,b)}{(\delta-\varrho)^{n-\alpha}}|B_1^{n-1}|^{\frac1{m_1}}
\|U\|_{L^b(B^{n-1}_{1})}\|u\|^\frac1r_{L^{p_1}(B^{n-1}_{1})}\\
&\le& \frac{c(n,\alpha,p,r,a,b)}{(\delta-\varrho)^{n-\alpha}}|B^{n-1}_1|^{\frac1{m_1}}
\|U\|_{L^b(B^{n-1}_{1})}\|V\|^\frac1r_{L^a(B^{+}_1)}\|v\|_{L^p(B^{+}_1)}.
\end{eqnarray*}
Combining the above and using  Minkowski inequality we have
\begin{eqnarray}\label{R-3}
\|v\|_{L^q(B^{+}_{\varrho})}&\le& c(n,\alpha,q,r,b)\|U\|_{L^b(B^{n-1}_{1})}\|u\|^\frac1r_{L^{q_1}(B^{n-1}_{\frac{\delta+\varrho}2})}
+\frac{c(n,\alpha,p,q,r,a,b)}{(\delta-\varrho)^{n-\alpha}}|B^{n-1}_1|^{\frac1{m_1}}|B^+_\varrho|^{\frac1{q}} \nonumber\\
& & \times \|U\|_{L^b(B^{n-1}_{1})}\|V\|^\frac1r_{L^a(B^{+}_1)}\|v\|_{L^p(B^{+}_1)}+\|h\|_{L^q(B^{+}_{\frac12})}.
\nonumber\\~~~~~~~~~~~~~~~~~
\end{eqnarray}
On the other hand, for $y\in B^{n-1}_{\frac{\delta+\varrho}2}$,
we have
\begin{eqnarray*}
u(y)&=&\int_{B^{+}_{\delta}}\frac{V(x)v^r(x)}{|x-y|^{n-\alpha}}dx
+\int_{B^{+}_1\backslash B^{n-1}_{\delta}}\frac{V(x)v^r(x)}{|x-y|^{n-\alpha}}dx\\
&=:& K_1(y)+K_2(y).
\end{eqnarray*}
Using \eqref{R-2} on $B^{n-1}_{\frac{\delta+\varrho}2}$, we have
\begin{eqnarray*}
 \|K_1\|_{L^{q_1}(B^{n-1}_{\frac{\delta+\varrho}2})}\le c(n,\alpha,q,r,a)\|V\|_{L^a(B^{+}_{1})}\|v\|^r_{L^{q}(B^{+}_{\delta})}.
\end{eqnarray*}
For $\frac1{m_2}=1-\frac1a-\frac rp,$
\begin{eqnarray*}
K_2(y)&\le& \frac{c(n,\alpha)}{(\delta-\varrho)^{n-\alpha}}\int_{B^{+}_1\backslash B^{+}_{\delta}} {V(x)v^r(x)} dx\\
&\le& \frac{c(n,\alpha,p,a,r)}{(\delta-\varrho)^{n-\alpha}}|B^+_1|^{\frac1{m_2}}
\|V\|_{L^a(B^{+}_{1})}\|v\|^r_{L^p(B^{+}_{1})}.
\end{eqnarray*}
Combining the above and using  Minkowski inequality we have
\begin{eqnarray}\label{R-4}
\|u\|_{L^{q_1}(B^{n-1}_{\frac{\delta+\varrho}2})}&\le& c(n,\alpha,q,r,a)\|V\|_{L^a(B^{+}_{1})}\|v\|^r_{L^{q}(B^{+}_{\delta})}
+\frac{c(n,\alpha,p,q,r,a,b)}{(\delta-\varrho)^{n-\alpha}}\nonumber\\
& &\times|B^+_1|^{\frac1{m_2}}|B^{n-1}_\delta|^{\frac1{q_1}}
\|V\|_{L^a(B^{+}_{1})}\|v\|^r_{L^p(B^{+}_{1})}.
\end{eqnarray}
Bringing \eqref{R-4} into \eqref{R-3} and choosing $\varepsilon$ small enough in condition
\[
\|U\|_{L^{b}(B^{n-1}_1)}\|V\|^\frac1r_{L^{a}(B^+_1)}\le \varepsilon(n,\alpha,p,q,r,a,b),
\]
we have
\begin{eqnarray*}
\|v\|_{L^q(B^{+}_{\varrho})}&\le& \frac12\|v\|_{L^{q}(B^{+}_{\delta})}
+c(n,\alpha,p,q,r,a,b,\varepsilon)\nonumber\\
& &\times
\big(\frac{|B^{n-1}_1|^{\frac1{m_1}}|B^+_1|^{\frac1{q}}}{(\delta-\varrho)^{n-\alpha}}
+\frac{|B^+_1|^{\frac 1{rm_2}}|B^{n-1}_1|^{\frac 1{rq_1}}}{(\delta-\varrho)^{\frac{n-\alpha}r}}\big)
\|v\|_{L^p(B^{+}_1)}+\|h\|_{L^q(B^{+}_{\frac12})}.
\end{eqnarray*} Then the usual iteration procedure (see, e.g. Lemma 4.1 in Chen and Wu \cite{CW91}, $P_{27}$) yields
\begin{eqnarray*}
\|v\|_{L^q(B^{+}_{\frac14})}&\le& c(n,\alpha,p,q,r,a,b,\varepsilon)  (\|v\|_{L^p(B^{+}_1)}+\|h\|_{L^q(B^{+}_{\frac12})}).
\end{eqnarray*}
%It is easy to verify  $\frac {n-1}{m_1}+\frac nq=n-\alpha+\frac nq-\frac np,\frac1r\big(\frac {n}{m_2}+\frac
%{n-1}{q_1}\big)=\frac{n-\alpha}r+\frac nq-\frac np$ by the definitions of $m_1,m_2,p_1,q_1,a,b.$
%Hence, we conclude the estimate of
%Proposition \ref{Reg-prop-1}.

For general $v, ~ h \in L^p(B_1^+),$ we follow the argument given in \cite{CL2010}. See, also, \cite{HWY2008}.  Let
$0\le\zeta(x)\le1$ be the measurable function such that
\begin{eqnarray*}
v(x)=\zeta(x)\int_{B^{n-1}_1}\frac{U(y)}{|x-y|^{n-\alpha}}
\big[\int_{B^{+}_1}\frac{V(z)v^r(z)}{|z-y|^{n-\alpha}}dz\big]^\frac1rdy+\zeta(x)h(x)\quad\text{for~ any}\quad x\in \mathbb{R}^{n}_+.
\end{eqnarray*}
Define the map $T$ by
\begin{eqnarray*}
T(\varphi)(x)=\zeta(x)\int_{B^{n-1}_1}\frac{U(y)}{|x-y|^{n-\alpha}}
\big[\int_{B^{+}_1}\frac{V(z)|\varphi(z)|^r}{|z-y|^{n-\alpha}}dz\big]^\frac1rdy.
\end{eqnarray*}
Similar to the above estimates, using inequality \eqref{HLSD-2} and H\"{o}lder inequality, we have
\begin{eqnarray*}
\|T(\varphi)\|_{L^p(B^+_1)}&\le& c(n,\alpha,p,r,a,b)\|U\|_{L^{b}(B^{n-1}_1)}\|V\|^\frac1r_{L^{a}(B^+_1)}\|\varphi\|_{L^p(B^+_1)}
\le\frac12\|\varphi\|_{L^p(B^+_1)},\\
\|T(\varphi)\|_{L^q(B^+_1)}&\le& c(n,\alpha,q,r,a,b)\|U\|_{L^{b}(B^{n-1}_1)}\|V\|^\frac1r_{L^{a}(B^+_1)}\|\varphi\|_{L^q(B^+_1)}
\le\frac12\|\varphi\|_{L^q(B^+_1)},
\end{eqnarray*} for $\varepsilon$ small enough.
Furthermore, for $\varphi,\psi\in L^p(B^+_1)$, it follows from  Minkowski inequality that
\begin{eqnarray*}
|T(\varphi)(x)-T(\psi)(x)|\le T(|\varphi-\psi|)(x),\quad\text{for}~x\in B^+_1.
\end{eqnarray*} Hence,
\begin{eqnarray*}
\|T(\varphi)-T(\psi)\|_{L^p(B^+_1)}\le \|T(|\varphi-\psi|)\|_{L^p(B^+_1)}\le\frac12\|\varphi-\psi\|_{L^p(B^+_1)}.
\end{eqnarray*}
Similarly, for $\varphi,\psi\in L^q(B^+_1)$, we have
\begin{eqnarray*}
\|T(\varphi)-T(\psi)\|_{L^q(B^+_1)}\le\frac12\|\varphi-\psi\|_{L^q(B^+_1)}.
\end{eqnarray*}
Define $h_j(x)=\min\{h(x),j\}.$ Then  we conclude from the contraction mapping theorem that we may find a unique $v_j\in L^q(B^+_1)$
such that
\begin{eqnarray*}
v_j(x)&=&T(v_j)(x)+\zeta(x)h_j(x)\\
&=&\zeta(x)\int_{B^{n-1}_1}\frac{U(y)}{|x-y|^{n-\alpha}}
\big[\int_{B^{+}_1}\frac{V(z)v_j^r(z)}{|z-y|^{n-\alpha}}dz\big]^\frac1rdy+\zeta(x)h_j(x).
\end{eqnarray*}
Using the {\it a priori} estimate for $v_j$ (noting that $h_j \in L^q(B^+_1)$), we have
\begin{eqnarray}\label{R-5}
\|v_j\|_{L^q(B^{+}_{\frac14})}&\le& c(n,\alpha,p,q,r,a,b,\varepsilon)\big(
\|v_j\|_{L^p(B^{+}_1)}+\|h_j\|_{L^q(B^{+}_{\frac12})}\big).
\end{eqnarray}
Observe that
\[v(x)=T(v)(x)+\zeta(x)h(x).
\]
We have
\begin{eqnarray*}
\|v_j-v\|_{L^p(B^+_1)}&\le&\|T(v_j)-T(v)\|_{L^p(B^+_1)}+\|h_j-h\|_{L^p(B^+_1)}\\
&\le& \frac12\|v_j-v\|_{L^p(B^+_1)}+\|h_j-h\|_{L^p(B^+_1)}.
\end{eqnarray*} This implies
\begin{eqnarray*}
\|v_j-v\|_{L^p(B^+_1)}\le2\|h_j-h\|_{L^p(B^+_1)}\to 0,
\end{eqnarray*}  as $j\to\infty.$ Note $ h_j\to h$ in $L^q(B^+_{1/2})$.
Sending $j$ to $\infty$ in \eqref{R-5}, we obtain Proposition \ref{Reg-prop-1}.
\end{proof}

\medskip

The dual local regularity result is the following.
\begin{prop}\label{Reg-prop-2}
Assume $1<\alpha<n,\,1<a,b\le\infty,1\le r<\infty$ and $\frac{n-1}{n-\alpha}<p<q<\infty $ satisfy
\begin{equation}\label{Cond-2}\frac{\alpha-1}{n-1}<\frac rq+\frac1a<\frac rp+\frac1a<1 \quad\text{and}\quad\frac{n-1}{ar}+\frac{n}b=\frac{\alpha-1}r+\alpha.
\end{equation}
Suppose  $u,g\in L^{p}(B^{n-1}_R), U\in L^{a}(B^{n-1}_R),V\in L^{b}(B^+_R)$ are all nonnegative functions with $g|_{B^{n-1}_{R/2}}\in
L^{q}(B^{n-1}_{R/2})$, and
\[u(y)\le\int_{B^{+}_R}\frac{V(x)}{|x-y|^{n-\alpha}}
\big[\int_{B^{n-1}_R}\frac{U(z)u^r(z)}{|z-y|^{n-\alpha}}dz\big]^\frac1rdx+g(y),\quad \forall  y\in B^{n-1}_R.
\]
 There is a $\varepsilon=\varepsilon(n,\alpha,p,q,r,a,b)>0$ and $C(n,\alpha,p,q,a,b,r,\varepsilon)>0$, such
 that if
\[\|U\|^\frac1r_{L^{a}(B^{n-1}_R)}\|V\|_{L^{b}(B^+_R)}\le \varepsilon(n,\alpha,p,q,r,a,b),
\]
then
\[\|u\|_{L^{q}(B^{n-1}_{R/4})}\le C(n,\alpha,p,q,a,b,r, \varepsilon)\big(R^{\frac{n-1}{q}-\frac{n-1}{p}}\|u\|_{L^{p}(B^{n-1}_{R})}+\|g\|_{L^{q}(B^{n-1}_{R/2})}).
\]
\end{prop}
\begin{proof}
After rescaling, we may assume $R=1$. We may further assume  $u,g\in L^{q}(B^{n-1}_{1})$, since  similar argument to that in the
proof of Proposition \ref{Reg-prop-1} yields the same estimate under the assumption of  $u,g\in L^{p}(B^{n-1}_{1})$.

 Denote
\[v(x)=\int_{B^{n-1}_1}\frac{U(y)u^r(y)}{|x-y|^{n-\alpha}}dy\quad\text{for}\quad x\in B^{+}_1.
\]
Let $p_1$ and $q_1$ be the numbers given by
\begin{eqnarray*}
\frac1{p_1}=\frac {n-1}n\big(\frac rp +\frac1a-\frac{\alpha-1}{n-1}\big),\quad
\frac1{q_1}=\frac {n-1}n\big(\frac rq+\frac1a-\frac{\alpha-1}{n-1}\big).
\end{eqnarray*}
Condition \eqref{Cond-2} indicates that $p_1, \, q_1>1$.
Similar to Proposition \ref{Reg-prop-1}, using \eqref{HLSD-2}, we have
\begin{eqnarray}
\|v\|_{L^{p_1}(B^{+}_1)}&\le& c(n,\alpha,p,a,r)\|U\|_{L^a(B^{n-1}_1)}\|u\|^r_{L^p(B^{n-1}_1)},\label{R-6}\\
\|v\|_{L^{q_1}(B^{+}_1)}&\le& c(n,\alpha,q,a,r)\|U\|_{L^a(B^{n-1}_1)}\|u\|^r_{L^q(B^{n-1}_1)}.\label{R-7}
\end{eqnarray}
Let $0<\varrho<\delta\le\frac12.$ For $y\in B^{n-1}_\delta$,
\begin{eqnarray*}
u(x)&\le&\int_{B^{+}_{\frac{\delta+\varrho}2}}\frac{V(x)v^\frac1r(x)}{|x-y|^{n-\alpha}}dx
+\int_{B^{+}_1\backslash B^{+}_{\frac{\delta+\varrho}2}}\frac{V(x)v^\frac1r(x)}{|x-y|^{n-\alpha}}dx+g(y)\\
&=:& J_3(y)+J_4(y)+g(y).
\end{eqnarray*}
 From \eqref{Cond-2}, we know $\frac{1}q=\frac n{n-1}\big(\frac1b+\frac1{q_1r}-\frac{\alpha}{n}\big)$. Using inequality \eqref{HLSD-3}, we have
\begin{eqnarray*}
 \|J_3\|_{L^q(B^{n-1}_{\varrho})}\le c(n,\alpha,q,r,b)\|V\|_{L^b(B^{+}_{1})}\|v\|^\frac1r_{L^{q_1}(B^{+}_{\frac{\delta+\varrho}2})}.
\end{eqnarray*}
Let $\frac1{m_3}=1-\frac1b-\frac1{p_1r}$ (note: $p>(n-1)/(n-\alpha)$ implies $m_3>1$). From H\"{o}lder inequality and \eqref{R-6}, it follows
\begin{eqnarray*}
J_4(y)&\le& \frac{c(n,\alpha)}{(\delta-\varrho)^{n-\alpha}}\int_{B^{+}_1\backslash B^{+}_{\frac{\delta+\varrho}2}} {V(x)v^\frac1r(x)} dx\\
&\le& \frac{c(n,\alpha,p,b,r)}{(\delta-\varrho)^{n-\alpha}}|B_1^+|^{\frac1{m_3}}
\|V\|_{L^b(B^{+}_{1})}\|v\|^\frac1r_{L^{p_1}(B^{+}_{1})}\\
&\le& \frac{c(n,\alpha,p,q,r,a,b)}{(\delta-\varrho)^{n-\alpha}}|B^{+}_1|^{\frac1{m_3}}
\|V\|_{L^b(B^{+}_{1})}\|U\|^\frac1r_{L^a(B^{n-1}_1)}\|u\|_{L^q(B^{n-1}_1)}.
\end{eqnarray*}
Combining the above and using  Minkowski inequality, we have
\begin{eqnarray}\label{R-8}
\|u\|_{L^q(B^{n-1}_{\varrho})}&\le& c(n,\alpha,q,r,b)\|V\|_{L^b(B^{+}_{1})}\|v\|^\frac1r_{L^{q_1}(B^{+}_{\frac{\delta+\varrho}2})}
+\frac{c(n,\alpha,p,q,r,a,b)}{(\delta-\varrho)^{n-\alpha}} |B^{+}_1|^{\frac1{m_3}}|B^{n-1}_\varrho|^{\frac1{q}}\nonumber\\
& &\times \|V\|_{L^b(B^{+}_{1})}\|U\|^\frac1r_{L^a(B^{n-1}_1)}\|u\|_{L^q(B^{n-1}_1)}+\|g\|_{L^q(B^{n-1}_{\frac12})}.
\end{eqnarray}
On the other hand, for $x\in B^{+}_{\frac{\delta+\varrho}2}$,
we have
\begin{eqnarray*}
v(x)&=&\int_{B^{n-1}_{\delta}}\frac{U(y)u^r(y)}{|x-y|^{n-\alpha}}dy
+\int_{B^{n-1}_1\backslash B^{n-1}_{\delta}}\frac{U(y)u^r(y)}{|x-y|^{n-\alpha}}dy\\
&=:& K_3(x)+K_4(x).
\end{eqnarray*}
Using \eqref{R-7} on $B^{+}_{\frac{\delta+\varrho}2}$, we have
\begin{eqnarray*}
 \|K_3\|_{L^{q_1}(B^{+}_{\frac{\delta+\varrho}2})}\le c(n,\alpha,q,r,a)\|U\|_{L^a(B^{n-1}_{1})}\|u\|^r_{L^{q}(B^{n-1}_{\delta})}.
\end{eqnarray*}
For $\frac1{m_4}=1-\frac1a-\frac rp $, from H\"{o}lder inequality, it yields
\begin{eqnarray*}
K_4(x)&\le& \frac{c(n,\alpha)}{(\delta-\varrho)^{n-\alpha}}\int_{B^{n-1}_1\backslash B^{n-1}_{\delta}} {U(y)u^r(y)} dy\\
&\le& \frac{c(n,\alpha,p,a,r)}{(\delta-\varrho)^{n-\alpha}}|B^{n-1}_1|^{\frac1{m_4}}
\|U\|_{L^a(B^{n-1}_{1})}\|u\|^r_{L^p(B^{n-1}_{1})}.
\end{eqnarray*}
Combining the above and using  Minkowski inequality, we have
\begin{eqnarray}\label{R-9}
\|v\|_{L^{q_1}(B^{+}_{\frac{\delta+\varrho}2})}&\le& c(n,\alpha,q,r,a)\|U\|_{L^a(B^{n-1}_{1})}\|u\|^r_{L^{q}(B^{n-1}_{\delta})}
+\frac{c(n,\alpha,p,q,r,a,b)}{(\delta-\varrho)^{n-\alpha}}\nonumber\\
& &\times|B^{n-1}_1|^{\frac1{m_4}}|B^{+}_\delta|^{\frac1{q_1}}
\|U\|_{L^a(B^{n-1}_{1})}\|u\|^r_{L^p(B^{n-1}_{1})}.
\end{eqnarray}
Bringing (\ref{R-9}) into (\ref{R-8}), for $\varepsilon$  small enough in
\[
\|V\|_{L^{b}(B^+_1)}\|U\|^\frac1r_{L^{a}(B^{n-1}_1)}\le \varepsilon(n,\alpha,p,q,r,a,b)
\]
we have
\begin{eqnarray*}
\|u\|_{L^q(B^{n-1}_{\varrho})}&\le& \frac12\|u\|_{L^{q}(B^{n-1}_{\delta})}
+c(n,\alpha,p,q,r,a,b,\varepsilon)\nonumber\\
& &\times
\big(\frac{|B^+_1|^{\frac1{m_3}}|B^{n-1}_1|^{\frac1{q}}}{(\delta-\varrho)^{n-\alpha}}
+\frac{|B^{n-1}_1|^{\frac 1{rm_4}}|B^+_1|^{\frac 1{rq_1}}}{(\delta-\varrho)^{\frac{n-\alpha}r}}\big)
\|u\|_{L^p(B^{n-1}_1)}+\|g\|_{L^q(B^{n-1}_{\frac12})}.
\end{eqnarray*} Using the standard iteration, we arrive
\begin{eqnarray*}
\|u\|_{L^q(B^{n-1}_{\frac14})}&\le& c(n,\alpha,p,q,r,a,b)(\|u\|_{L^p(B^{n-1}_1)}+\|g\|_{L^q(B^{n-1}_{\frac12})}).
\end{eqnarray*}
%
%It is easy to verify  $\frac n{m_3}+\frac{n-1}q=n-\alpha+\frac {n-1}q-\frac {n-1}p,\frac1r\big(\frac {n-1}{m_4}+\frac
%{n}{q_1}\big)=\frac{n-\alpha}r+\frac {n-1}q-\frac {n-1}p$ by the definitions of $m_3,m_4,p_1,q_1,a,b.$  With the above a priori
%estimate, we may proceed in the same way as in the proof Proposition \ref{Reg-prop-1} to finish the proof of Proposition
%\ref{Reg-prop-2}.
\end{proof}

%{\bf I am here now 2-09-2012}

 \textbf{Proof of Theorem \ref{Regularity-2}.}  For $R>0$, define
\begin{eqnarray*}
u_R(y)&=&\int_{\Sigma_R^{n}}\frac{v^\kappa(x)}{|x-y|^{n-\alpha}}dx,\quad
v_R(x)=\int_{\Sigma_R^{n-1}}\frac{u^\theta(y)}{|x-y|^{n-\alpha}}dy.
\end{eqnarray*}
Thus, from system \eqref{sys-1}, we have
\begin{eqnarray*}
u(y)&=&\int_{B_R^{+}}\frac{v^\kappa(x)}{|x-y|^{n-\alpha}}dx+u_R(y),\\
v(x)&=&\int_{B_R^{n-1}}\frac{u^\theta(y)}{|x-y|^{n-\alpha}}dy+v_R(x).
\end{eqnarray*}
We prove this  theorem in two steps.

\textbf{Step $1$.}  We show: if $u\in L^{\theta+1}_{loc}(\partial\mathbb{R}^{n}_+),$  then $ v\in
L^{\kappa+1}_{loc}(\overline{\mathbb{R}^{n}_+}), v_R\in  L_{loc}^\infty(B^{+}_{R}\cup B^{n-1}_{R}), $ and $u_R\in
 L_{loc}^{\infty}(B^{n-1}_R).$

We firstly show that if $u\in L^{\theta+1}_{loc}(\partial\mathbb{R}^{n}_+),$ then
\[
v\in L^{\kappa+1}_{loc}(\overline{\mathbb{R}^{n}_+})\quad\text{and}\quad v_R\in  L_{loc}^\infty(B^{+}_{R}\cup B^{n-1}_{R}).
\]
In fact, since $u\in L^{\theta+1}_{loc}(\partial\mathbb{R}^{n}_+),$ we have $u<\infty$ a.e. in $\partial\mathbb{R}^{n}_+.$ This
implies $v<\infty$ a.e. in $\mathbb{R}^{n}_+$.  Hence, there exists an $x_0\in B_\frac{R}2^+$, such that $v(x_0)<\infty,$ that is,
\begin{eqnarray*}
\int_{\Sigma_R^{n-1}}\frac{u^\theta(y)}{|y|^{n-\alpha}}dy
\le c\int_{\Sigma_R^{n-1}}\frac{u^\theta(y)}{|x_0-y|^{n-\alpha}}dy\le cv(x_0)<\infty.
\end{eqnarray*}
For $0<\delta<1, x\in B^{+}_{\delta R}$, it holds
\begin{eqnarray*}
v_R(x)&\le&\frac{c(n,\alpha) }{(1-\delta)^{n-\alpha}}\int_{\Sigma_R^{n-1}}\frac{u^\theta(y)}{|y|^{n-\alpha}}dy.
\end{eqnarray*} This shows $v_R\in L_{loc}^\infty(B^{+}_{R}\cup B^{n-1}_{R} ).$ On the other hand,
using inequality \eqref{HLSD-2} with $\frac1{\kappa+1}=\frac{n-1}n\big(\frac{n-\alpha}{n-1}-\frac1{\theta+1}\big),$ we have
\begin{eqnarray*}
\big[\int_{R^n_+}\big(\int_{B_R^{n-1}}\frac{u^\theta(y)}{|x-y|^{n-\alpha}}dy \big)^{\kappa+1}dx\big]^\frac1{\kappa+1}\le
c\|u\|_{L^{\theta+1}(B^{n-1}_{R})}<\infty.
\end{eqnarray*}
This implies $v\in L^{\kappa+1}_{loc}(B^{+}_{R}\cup B^{n-1}_{R}).$  Since $R$ is arbitrary, we have $v\in
L^{\kappa+1}_{loc}(\overline{\mathbb{R}^n_+})$.

Next, we show $u_R\in  L^{\infty}_{loc}(B^{n-1}_R).$ Indeed, since $u(y) <\infty$ a.e., there is a $y_0\in B^{n-1}_\frac{R}2$, such
that $u(y_0)<\infty$. Thus
\begin{eqnarray*}
\int_{\Sigma_R^{n}}\frac{v^\kappa(x)}{|x|^{n-\alpha}}dx
\le c\int_{\Sigma_R^{n}}\frac{v^\kappa(x)}{|x-y_0|^{n-\alpha}}dx\le cu(y_0)<\infty.
\end{eqnarray*}
For $0<\delta<1, y\in B^{n-1}_{\delta R}$, we have
\begin{eqnarray*}
u_R(y)\le\frac{c(n,\alpha)}{(1-\delta)^{n-\alpha}}\int_{\Sigma_R^{n}} \frac{v^\kappa(v)}{|x|^{n-\alpha}}dx<\infty.
\end{eqnarray*}
That is, $u_R\in  L^{\infty}_{loc}(B^{n-1}_R)$.

\textbf{Step $2$.} We show $u\in C^\infty(\partial\mathbb{R}^{n}_+)$ and $v\in C^\infty(\overline{\mathbb{R}^{n}_+})$. To do this,
we discuss two cases.

\textbf{ Case $1$.}
$\frac{\alpha-1}{n-\alpha}<\theta<\frac{n+\alpha-2}{n-\alpha}.$

Since $\theta<\frac{n+\alpha-2}{n-\alpha},$ we can see from (\ref{k-t}), that $\kappa>\frac{n+\alpha }{ n-\alpha }$ and
\begin{eqnarray*}
(\theta+1)+\frac{(n-1)(\kappa+1)}{n}=\big(1-\frac\alpha n\big)(\kappa+1)(\theta+1).
\end{eqnarray*}
That is,
\begin{eqnarray*}
\kappa\theta-\frac\alpha n(\kappa+1)\theta-1=\frac{\alpha-1} n(\kappa+1)>0.
\end{eqnarray*}
Thus, we have
\[\big(\kappa-\frac\alpha n(\kappa+1)\big)\theta>1.\]
This implies  $\kappa-\frac\alpha n(\kappa+1)>\frac1\theta.$ On the
other hand, since  $\kappa>\frac{n+\alpha}{ n-\alpha }$, we have $\kappa-\frac\alpha n(\kappa+1)>1$. Hence,  we can choose a
fixed number $r$ such that
\[
1<\kappa-\frac\alpha n(\kappa+1)\le r\le \kappa,\quad\text{and}\quad r>\frac1\theta.
\] %(For example, we may choose $r=\kappa-\frac\alpha n(\kappa+1)$.)
We have
\begin{eqnarray*}
u^\frac1r(y)&\le&\big(\int_{B_R^{+}}\frac{v^\kappa(x)}{|x-y|^{n-\alpha}}dx\big)^\frac1r+u_R^\frac1r(y), \end{eqnarray*} which yields
\begin{eqnarray*}
v(x)&=&\int_{B_R^{n-1}}\frac{u^{\theta-\frac1r}(y)u^\frac1r(y)}{|x-y|^{n-\alpha}}dy+v_R(x)\\
&\le&\int_{B_R^{n-1}}\frac{u^{\theta-\frac1r}(y)}{|x-y|^{n-\alpha}}
\big(\int_{B_R^{+}}\frac{v^{\kappa-r}(x)v^{r}(x)}{|x-y|^{n-\alpha}}dx\big)^\frac1r dy+h_R(x),
\end{eqnarray*}
where
\[h_R(x)=\int_{B_R^{n-1}}\frac{u^{\theta-\frac1r}(y)u_R^\frac1r(y)}{|x-y|^{n-\alpha}}
dy+v_R(x).
\]

 %\le \int_{B_R^{n-1}}\frac{u^{\theta}(y)}{|x-y|^{n-\alpha}} dy+v_R(x)
%By  step $1$, we have $h_R\in L^{\kappa+1}_{loc}(B^{+}_{R}\cup B^{n-1}_{R}).$ {\bf I am here 2-09-2012}

Since $u_R\in L^\infty_{loc}(\partial\mathbb{R}^{n}_+)$, for any  $ x\in B_{ R}^+ $,  we have
\begin{eqnarray*}
\int_{B_R^{n-1}}\frac{u^{\theta-\frac1r}(y)u_R^\frac1r(y)}{|x-y|^{n-\alpha}} dy&\le&\|u_R\|^\frac1r_{L^\infty(B_R^{n-1})}
\int_{B_R^{n-1}}\frac{u^{\theta-\frac1r}(y)}{|x-y|^{n-\alpha}}dy.\\%
%& &
%+\frac{c(n,\alpha)}{(1-\delta)^{n-\alpha}R^{n-\alpha}}\int_{B_R^{n-1}\backslash B^{n-1}_{\frac{(1+\delta)R}2}}{u^{\theta}(y)}dy\\
%&\le&\|u_R\|^\frac1r_{L^\infty(B_{\frac{(1+\delta)R}2}^{n-1})}
%\int_{B_{\frac{(1+\delta)R}2}^{n-1}}\frac{u^{\theta-\frac1r}(y)}{|x-y|^{n-\alpha}}dy\\
%& &
%+\frac{c(n,\alpha)}{(1-\delta)^{n-\alpha}R^{n-\alpha}}|B_R^{n-1}|^\frac{1}{\theta+1}
%\|u\|^\theta_{L^{\theta+1}(B_R^{n-1})}\\
%&\le&\|u_R\|^\frac1r_{L^\infty(B_{\frac{(1+\delta)R}2}^{n-1})}
%\int_{B_{\frac{(1+\delta)R}2}^{n-1}}\frac{u^{\theta-\frac1r}(y)}{|x-y|^{n-\alpha}}dy\\
%& &
%+\frac{c(n,\alpha)}{(1-\delta)^{n-\alpha}R^{n-\alpha-(n-1)/(\theta+1)}}
%\|u\|^\theta_{L^{\theta+1}(B_R^{n-1})}.
\end{eqnarray*}
Note $u\in L^{\theta+1}_{loc}(\partial\mathbb{R}^{n}_+)$. We know, via inequality \eqref{HLSD-2}, that $h_R\in L^{q_0}(B^+_R\cup
B_R^{n-1})$ with $q_0$ given by
\begin{eqnarray*}
\frac 1{q_0} &=& \frac {n-1}n(\frac{\theta-1/r}{\theta +1}-\frac{\alpha -1}{n-1})\\
&=& \frac {n-1}n(\frac{n-\alpha }{n-1}-\frac{1+1/r}{\theta +1})\\ %details skip for publication, 3-18, mj
&=& \frac {n-1}n(\frac{n-\alpha }{n-1}-\frac{1}{\theta +1})-\frac{n-1}{rn(\theta +1)}\\
&=&\frac 1{\kappa+1}-\frac{n-1}{rn(\theta +1)}.
\end{eqnarray*}
%Thus $q_0>\kappa+1.$ In fact, f
For $\varepsilon>0$ small enough, we can choose $r= \kappa-\frac\alpha
n(\kappa+1)+\varepsilon>1+\varepsilon$ so
 that
\begin{eqnarray*}
q_0
&=&\big(\frac 1{\kappa+1}-\frac{n-1}{rn(\theta +1)}\big)^{-1}\\
&=&\frac{rn(\kappa+1)}{n(r+1)-(n-\alpha)(\kappa+1)}\\
&=&\frac{rn(\kappa+1)}{(n-\alpha)(\kappa+1)+n\varepsilon-(n-\alpha)(\kappa+1)}\\
&=&\frac{rn(\kappa+1)}{n\varepsilon}>\frac{\kappa+1}{\varepsilon}
\end{eqnarray*}
can be any large number when we choose $\varepsilon$ small enough. In the above derivation (in the second equality), we used the equation:
$\frac{1}{\theta+1}=\frac{n-\alpha}{n-1}-\frac{n}{n-1}\frac{1}{\kappa+1}$, which can be deduced from \eqref{k-t} easily. Hence, it
follows that $h_R\in L^{q }(B^+_R\cup B_R^{n-1})$ for any $q<\infty.$

Now, in Proposition \ref{Reg-prop-1}, take
\begin{eqnarray*}
U(y)&=&u^{\theta-\frac1r}(y),\quad V(x)=v^{\kappa-r}(x),\\
a&=&\frac{\kappa+1}{\kappa-r},\quad b=\frac{\theta+1}{\theta-\frac1r},\quad p=\kappa+1 >\frac n{n-\alpha}.
\end{eqnarray*}
Since $u \in L^{\theta+1}_{loc}(\partial\mathbb{R}^{n}_+)$ and $v\in L^{\kappa+1}_{loc}(\overline{\mathbb{R}^n_+}),$ we have $U\in
L^{b}(B^{n-1}_R)$ and $V\in L^{a}(B^{+}_R)$. Moreover, it is easy to verify via \eqref{k-t}, that
\[\frac{n}{ra}+\frac{n-1}b=\frac\alpha r+(\alpha-1),\quad\text{and}\quad \frac rp+\frac1a=\frac{\kappa}{\kappa+1}<1.
\]
For
\[\kappa+1<q<\infty,
\]
it is obvious that $\frac{r}q+\frac1a>\frac{\alpha}{n}.$  We know from Proposition \ref{Reg-prop-1} that $v|_{B^+_\frac R4}\in
L^q(B^+_\frac R4)$ for small enough $R$. Hence, we can choose a $q$ satisfying $n\kappa/\alpha<q<\infty$ such that
\begin{eqnarray*}
u(y)&=&\int_{B^+_\frac R4}\frac{v^\kappa(x)}{|x-y|^{n-\alpha}}dx+u_{\frac R4}(y),\\
&\le&\big(\int_{B^+_\frac R4}|x-y|^{\frac{(\alpha- n)q}{q-\kappa}}dx\big)^\frac{q-\kappa}{q}\|v\|^\kappa_{L^q(B^+_\frac R4)}+u_{\frac R4}(y)\\
&\le& c(n,\alpha,q) R^{\frac{n (q-\kappa)}{ q}-(n-\alpha)}\|v\|^\kappa_{L^q(B^+_\frac R4)}+u_{\frac R4}(y)<\infty.
\end{eqnarray*} %{\bf I am here 2-12-2012}
This also implies $u|_\frac R8\in L^\infty(B^{n-1}_\frac R8)$.
Since every point can be viewed as a center, we have $u\in L^\infty_{loc}(\partial\mathbb{R}^{n}_+)$,
and hence $v\in L^\infty_{loc}(\overline{\mathbb{R}^{n}_+})$.

 For any $R>0,$ since
\begin{eqnarray*}
\int_{\Sigma_R^{n-1}}\frac{u^\theta(y)}{|y|^{n-\alpha}}dy<\infty,\quad
\int_{\Sigma_R^{n}}\frac{v^\kappa(x)}{|x|^{n-\alpha}}dx<\infty,
\end{eqnarray*}
we know $v_R\in C^\infty(B^+_R\cup B^{n-1}_R)$ and $u_R\in  C^\infty(B^{n-1}_R)$. The first derivative of
$\int_{B_R^{+}}\frac{v^\kappa(x)}{|x-y|^{n-\alpha}}dx$ is at least H\"{o}lder continuous in $B^{n-1}_R$ (since $\alpha>1$). Since $R$ is arbitrary, we know $u$
is $C^{1, \tau}$ continuous on $\partial\mathbb{R}^{n}_+$, and hence $v$ is $C^{1, \tau}$ continuous on
${\partial\mathbb{R}^{n}_+}.$ Direct computation also shows  $v$ is $C^{1, \tau}$ continuous in $\overline{\mathbb{R}^{n}_+}$. By bootstrap, we conclude that $u\in C^\infty(\partial\mathbb{R}^{n}_+)$ and $v\in
C^\infty(\overline{\mathbb{R}^{n}_+})$.

\smallskip

\textbf{Case $2.$}  For $\frac{n+\alpha-2}{n-\alpha}\le
\theta<\infty.$ In this case, from (\ref{k-t}),  it is easy to check
$\frac\alpha{n-\alpha}<\kappa\le\frac{n+\alpha}{n-\alpha},$ and
\begin{eqnarray*}
(\kappa+1)+\frac{n(\theta+1)}{n-1}=\big(1-\frac{\alpha-1}{n-1}\big)(\kappa+1)(\theta+1).
\end{eqnarray*}
That is,
\begin{eqnarray*}
\kappa\theta-\frac{\alpha-1}{n-1}(\theta+1)\kappa-1=\frac{\alpha} {n-1}(\theta+1)>0.
\end{eqnarray*}
Thus, we have
\[\big(\theta-\frac{\alpha-1}{n-1}(\theta+1)\big)\kappa>1.\]
This implies $\theta-\frac{\alpha-1}{n-1}(\theta+1)>\frac1\kappa.$ On the other hand, since  $\theta\ge\frac{n+\alpha-2}{ n-\alpha
}$, we have $\theta-\frac{\alpha-1}{n-1}(\theta+1)\ge1$. Hence,  we can choose a fixed number $r$ such atht
\[1\le\theta-\frac{\alpha-1}{n-1}(\theta+1)\le r\le\theta,\quad\text{and}\quad r>\frac1\kappa,
\] and then
\begin{eqnarray*}
v^\frac1r(x)&\le&\big(\int_{B_R^{n-1}}\frac{u^{\theta}(y)}{|x-y|^{n-\alpha}}dy\big)^\frac1r+v_R^\frac1r(x).
\end{eqnarray*}
Thus,
\begin{eqnarray*}
u(y)&=&\int_{B_R^{+}}\frac{v^{\kappa-\frac1r}(x)}{|x-y|^{n-\alpha}}
\big(\int_{B_R^{n-1}}\frac{u^\theta(y)}{|x-y|^{n-\alpha}}dy\big)^\frac1rdx+g_R(y),
\end{eqnarray*}
where
\[g_R(y)=\int_{B_R^{+}}\frac{v^{\kappa-\frac1r}(x)v_R^\frac1r(x)}{|x-y|^{n-\alpha}} dx
+u_R(y).\\
\]
%From the result of \textbf{Step $1$},
%By inequality \eqref{HLSD-3}, we see $g_R\in L^{\theta+1}(B^{n-1}_R).$ On the other hand, f
For  any $y\in B^{n-1}_{R}$,
\begin{eqnarray*}
\int_{B_R^{+}}\frac{v^{\kappa-\frac1r}(x)v_R^\frac1r(x)}{|x-y|^{n-\alpha}} dx
%&\le&\|v_R\|^\frac1r_{L^{\infty}(B^{+}_R)}\int_{B_{\frac{(1+\delta) R}2}^{+}}\frac{v^{\kappa-\frac1r}(x)}{|x-y|^{n-\alpha}} dx
%+\int_{B_R^{+}\backslash B_{\frac{(1+\delta) R}2}^{+}}\frac{v^{\kappa-\frac1r}(x)v_R^\frac1r(x)}{|x-y|^{n-\alpha}} dx\\
\le \|v_R\|^\frac1r_{L^{\infty}(B^{+}_R)}\int_{B^{+}_R }\frac{v^{\kappa-\frac1r}(x)}{|x-y|^{n-\alpha}} dx
%+\frac{c(n,\alpha,\kappa)}{(1-\delta)^{n-\alpha}R^{n-\alpha-\frac{n}{\kappa+1}}} \|v\|^\kappa_{L^{\kappa+1}(B^{+}_R)}.
\end{eqnarray*}
By inequality \eqref{HLSD-3}, we have $g_R\in L^{q_1}(B^{n-1}_R)$ with $\theta+1< q_1\le\infty,$ where
$q_1=\frac{(n-1)r(\theta+1)}{(n-1)(r+1)-(n-\alpha)(\theta+1)}.$ As in case 1: $q_1$ can be chosen as any larger number.

Now in Proposition \ref{Reg-prop-2}, take
\begin{eqnarray*}
& &U(y)=u^{\theta-r}(y),\quad V(x)=v^{\kappa-\frac1r}(x),\\
& &a=\frac{\theta+1}{\theta-r}\quad
b=\frac{\kappa+1}{\kappa-\frac1r},\quad p=\theta+1.
\end{eqnarray*}
Since $u \in L^{\theta+1}_{loc}(\partial\mathbb{R}^{n}_+)$ and $v\in
L^{\kappa+1}_{loc}(\overline{\mathbb{R}^n_+}),$ we have $U\in
L^{a}(B^{n-1}_R)$ and $V\in L^{b}(B^{+}_R)$, and it is easy to
verify that
\[
\frac{n-1}{ar}+\frac nb= \frac{\alpha-1}r+\alpha,\quad\frac
rp+\frac1a=\frac{\theta}{\theta+1}<1,
\]
For any $\theta+1< q<\infty$, it follows from Proposition
\ref{Reg-prop-2} that $u\in L^{q}(B^{n-1}_R)$.  Since every point
can be viewed as a center, we have $u\in
L^q_{loc}(\partial\mathbb{R}^{n}_+)$.
%$v\in L^{p_1}_{loc}(\overline{\mathbb{R}^{n}_+})$ with $p_1=
%\frac{(n-1)q}{n\theta-\alpha q}$. Furthermore, we can choose
%$\theta+1<q<\infty$ such that
%\begin{eqnarray*}
%u(y)&=&\int_{B^+_\frac R4}\frac{v^\kappa(x)}{|x-y|^{n-\alpha}}dx+u_{\frac R4}(y),\\
%&\le&\big(\int_{B^+_\frac R4}|x-y|^{\frac{(n-\alpha)p_1}{
%p_1-\kappa}}dx\big)^\frac{
%p_1-\kappa}{p_1}\|v\|^\kappa_{L^{p_1}(B^+_\frac R4)}+u_{\frac R4}(y)\\
%&\le& c(n,\alpha,p_1) R^{\frac{n(p_1-\kappa)}{p_1}-(n-\alpha)
%}\|v\|^\kappa_{L^{p_1}(B^+_\frac R4)}+u_{\frac R4}(y)<\infty.
%\end{eqnarray*}
%This also implies $u|_\frac R8\in L^\infty(B^{n-1}_\frac R8)$.
%Since every point
%can be viewed as a center, we have $u\in
%L^\infty_{loc}(\partial\mathbb{R}^{n}_+),$ and then $v\in
%L^\infty_{loc}(\overline{\mathbb{R}^{n}_+})$.

Now, similar to the argument in previous case, we have $v\in L^\infty_{loc}(\overline{\mathbb{R}^{n}_+})$, and then $u\in
L^\infty_{loc}(\partial\mathbb{R}^{n}_+)$. It follows that $u\in C^\infty(\partial\mathbb{R}^{n}_+)$ and $v\in
C^\infty(\overline{\mathbb{R}^{n}_+})$.

%\textbf{Step $3$.} If $u\in L^{\theta+1}(\partial\mathbb{R}^{n}_+),$
%we have $v\in L^{\kappa+1}(\mathbb{R}^{n}_+)$.  Arguing as step $1$
%and $2$, we can obtain $u\in L^\infty (\partial\mathbb{R}^{n}_+)$,
%and hence $v\in L^\infty (\overline{\mathbb{R}^{n}_+})$. By
%interpolation we have $u\in L^\gamma (\partial\mathbb{R}^{n}_+)$ and
%$v\in L^\gamma(\overline{\mathbb{R}^{n}_+})$ for any $1\le \gamma\le
%\infty$. Denote
%\begin{eqnarray*}
%u&=&\frac1{|x|^{n-\alpha}}*\big(v^\kappa(x)\chi_{\mathbb{R}^{n}_+}\big)\\
%&=&\big(\frac1{|x|^{n-\alpha}}\chi_{B_1}\big)*\big(v^\kappa(x)\chi_{\mathbb{R}^{n}_+}\big)
%+\big(\frac1{|x|^{n-\alpha}}\chi_{\mathbb{R}^{n}\backslash B_1}\big)
%*\big(v^\kappa(x)\chi_{\mathbb{R}^{n}_+}\big).
%\end{eqnarray*}
%Since $\frac1{|x|^{n-\alpha}}\chi_{B_1}\in L^{\frac
%n{n-\alpha}-\varepsilon}(\mathbb{R}^{n})$ and
%$\frac1{|x|^{n-\alpha}}\chi_{\mathbb{R}^{n}\backslash B_1}\in
%L^{\frac n{n-\alpha}+\varepsilon}(\mathbb{R}^{n})$ for
%$\varepsilon>0$ small, we see $u$ is continuous and $u(y)\to0$ as
%$|y|\to \infty.$ Similarly, $v(x)\to0$ as $|x|\to \infty.$

\hfill$\Box$

%
%
%\textbf{Proof of Theorem \ref{sys-nonexistence}.}
%According to
%Lemma \ref{K-lm-no-2},  $\bar{\lambda}(x)=\infty$ for all
%$x\in\partial \mathbb{R}^n_+$, that is, for all $\lambda>0$ and $x\in\partial \mathbb{R}^n_+$,
%\begin{eqnarray*}
%u_{x, {\lambda}}(\xi)&\le&u(\xi)\quad
%\text{for}~\xi\in\Sigma_{x, {\lambda}}^{n-1},\\
%v_{x, {\lambda}}(\eta)&\le&v(\eta)\quad \text{for}~\eta\in
%\Sigma_{x, {\lambda}}^{n}.
%\end{eqnarray*}As proof process $b$ of Theorem
%\ref{sys-critical}, we have a contradiction with \eqref{sys-1}. This
%completes the proof. \hfill$\Box$

%~~~~~~~~~~~~~~~~~~~~~~~~~~~~~~~~~~~~~~~~~~~~~~~~~~~~~~~~~~~~~~~~~~~~~~~~~~~~
\section{\textbf{Miscellaneous} \label{Section 5}}
\medskip

In this section, we shall include some related results concerning the computation of the sharp constants, operators on a bounded
domain, inequality for limit case ($\alpha=n$), fractional Laplacian operators and some non-existence to a system of integral equations.

\subsection{Integral inequalities in a bounded domain}\label{subsect 5.1}
For a smooth and bounded domain $\Omega \subset \mathbb{R}^n$, we introduce the following operators
\begin{equation}\label{5-1}
 \widetilde{ E_\alpha} f(x) =\int_{\partial \Omega }\frac{f(y)}{|x-y|^{n-\alpha}}dS_y, \quad ~~~~~~~~~~\forall
\, x\in \Omega
\end{equation}
and
\begin{equation}\label{5-2}
\widetilde{ R_\alpha} g(y) =\int_{ \Omega }\frac{g(x)}{|x-y|^{n-\alpha}}dV_x, \quad ~~~~~~~~~~\forall \, y\in \partial \Omega.
\end{equation}
From Theorem \ref{HLSD-theo}, we first show

\begin{crl}\label{ball-1} Assume $1<\alpha<n$. For any  $f\in L^\frac{2n}{n+\alpha-2}
(\partial B_{R}),g\in L^\frac{2n}{n+\alpha} (B_{R})$
with $R>0$,
\begin{eqnarray}
\|\widetilde{E_\alpha} f\|_{L^\frac{2n}{n-\alpha}(B_{R})}
&\le& C_e(n,\alpha,\frac{2(n-1)}{n+\alpha-2})\|f\|_{L^\frac{2(n-1)}{n+\alpha-2}(\partial B_{R})},\label{B-1}\\
\|\widetilde{R_\alpha} g\|_{L^\frac{2(n-1)}{n-\alpha}(\partial B_{R})} &\le&
C_r(n,\alpha,\frac{2n}{n+\alpha})\|g\|_{L^\frac{2n}{n+\alpha}(B_{R})}.\label{B-2}
\end{eqnarray} Moreover the best constants $C_e(n,\alpha, \frac{2(n-1)}{n+\alpha-2})$  is given by
\begin{eqnarray*}
C_e(n,\alpha, \frac{2(n-1)}{n+\alpha-2})&=&(n\omega_{n})^{-\frac{n+\alpha-2}{2(n-1)}}\big(\int_{B_{1}(x_1)}\big(\int_{\partial B_{1}(x_1)}\frac{1}{|\xi-z|^{n-\alpha}}d{S_z}\big)^\frac{2n}
{n-\alpha}d\xi\big)^\frac{n-\alpha}{2n}.\\
%C_r(n,\alpha,p)&=&\big(\int_{\partial B_{1}(x_1)}
%\big(\int_{B_{1}(x_1)}\frac1{|\xi-z|^{n-\alpha}}\big(\int_{\partial B_{1}(x_2)}\frac{1}
%{|\xi-\zeta|^{n-\alpha}}d S_\zeta\big)^{\frac{n+\alpha}{n-\alpha}}
%\\
%& &\cdot d\xi\big)^\frac{2(n-1)}{n-\alpha}dz\big)^\frac{n-\alpha}{2(n-1)}
%\big(
%\int_{B_{1}(x_1)}\big(\int_{\partial B_{1}(x_1)}\frac{1}
%{|\xi-z|^{n-\alpha}}dS_z\big)^{\frac{2n}{n-\alpha}}d\xi\big)^{-\frac{n+\alpha}{2n}}.
\end{eqnarray*} where  $x_1=(0,-1),$  and $C_r(n,\alpha,\frac{2n}{n+\alpha})=C_e(n,\alpha, \frac{2(n-1)}{n+\alpha-2}).$ In particular, for $\alpha=2$,
\begin{equation}\label{sharpC1}
C_e(n,2, \frac{2(n-1)}{n})=C_r(n,2,\frac{2n}{n+2})= n^{\frac{n-2}{2(n-1)}}\omega_{n}^{1-\frac1n-\frac{1}{2(n-1)}}.
\end{equation}
\end{crl}
\begin{proof} For given $\mu\ne0$,
define
\[\omega_{\mu,x_0,\lambda}(y)=\big(\frac\lambda{|y-x_0|}\big)^\mu \omega(y^{x_0,\lambda}),\quad y^{x_0,\lambda}=\frac{\lambda^2(y-x_0)}{|y-x_0|^2}+x_0,
\] where $y\in \overline{\mathbb{R}^n_+},x_0=(x'_0,-\lambda),\lambda>0.$
Directly computation  shows
\begin{eqnarray*}
\int_{\partial\mathbb{R}^{n}_+}|f(y)|^\frac{2(n-1)}{n+\alpha-2}dy
&=&\int_{\partial B_{\frac \lambda2}(x_1)}|f(z^{x_0,\lambda})|^\frac{2(n-1)}
{n+\alpha-2}\big(\frac\lambda{|y-x_0|}\big)^{2(n-1)}d{S_z}\quad (\mbox{let} \, \, y=z^{x_0,\lambda})\\
&=&\int_{\partial B_{\frac \lambda2}(x_1)}|f_{\mu, x_0,\lambda}(z)|^\frac{2(n-1)}{n+\alpha-2}
\big(\frac\lambda{|y-x_0|}\big)^{2(n-1)-\mu\frac{2(n-1)}{n+\alpha-2}}d{S_z},
\end{eqnarray*}
and
\begin{eqnarray*}
& &\int_{\mathbb{R}^{n}_+}|E_\alpha f(x)|^\frac{2n}{n-\alpha}dx\\
 &=&\int_{B_{\frac \lambda2}(x_1)}\big|\int_{\partial B_{\frac
\lambda2}(x_1)}\frac{f(z^{x_0,\lambda})}{|\xi^{x_0,\lambda}-z^{x_0,\lambda}|^{n-\alpha}}
\big(\frac\lambda{|y-x_0|}\big)^{2(n-1)}d{S_z}
\big|^\frac{2n}{n-\alpha}\big(\frac\lambda{|\xi-x_0|}\big)^{2n}d\xi\\
&=&\int_{B_{\frac \lambda2}(x_1)}\big|\int_{\partial B_{\frac
\lambda2}(x_1)}\frac{f_{\mu,x_0,\lambda}(z)}{|\xi-z|^{n-\alpha}}\big(\frac\lambda{|y-x_0|}\big)^{\gamma_1}d{S_z}
\big|^\frac{2n}{n-\alpha}\big(\frac\lambda{|\xi-x_0|}\big)^{\gamma_2}d\xi,
%&=&\int_{B_{\frac \lambda2}(x_1)}\big|\int_{\partial B_{\frac \lambda2}(x_1)}\frac{f_{x_0,\lambda}(z)}{|\xi-z|^{n-\alpha}}dz
%\big|^\frac{2n}{n-\alpha}d\xi\\
%&=&\int_{B_{\frac \lambda2}(x_1)}|E_\alpha f_{x_0,\lambda}(\xi)|^\frac{2n}{n-\alpha}d\xi.
\end{eqnarray*} where $x_1=(x_0,-\lambda/2)$, $\gamma_1=2(n-1)-\mu-(n-\alpha),$ and $\gamma_2=2n-(n-\alpha)\frac {2n}{n-\alpha}=0.$

In the case of $\mu=n+\alpha-2$, we will denote $f_{\mu, x, \lambda}$ as $f_{x, \lambda}$. Thus our notation is consistent with
those in Section \ref{Section 3}. In fact, it is clear from the above that for $\mu={n+\alpha-2}$,
\[\|f\|_{L^\frac{2(n-1)}{n+\alpha-2}(\partial\mathbb{R}^{n}_+)}
=\|f_{x_0,\lambda}\|_{L^\frac{2(n-1)}{n+\alpha-2}(\partial B_{\frac \lambda2}(x_1))}
\]
and
\[\|E_\alpha f\|_{L^\frac{2n}{n-\alpha}(\mathbb{R}^{n}_+)}=\|\widetilde{E_\alpha} (f_{x_0,\lambda})\|_{L^\frac{2n}{n-\alpha}(B_{\frac \lambda2}(x_1))}.
\]
%For $\mu={n+\alpha-2}$, by the  it follows
%\[\|\widetilde{E_\alpha} f_{x_0,\lambda}\|_{L^\frac{2n}{n-\alpha}(B_{\frac \lambda2}(x_1))}
%\le C_e(n,\alpha)\|f_{x_0,\lambda}\|_{L^\frac{2(n-1)}{n+\alpha-2}(\partial B_{\frac \lambda2}(x_1))}.
%\]
Note: $(f_{x_0,\lambda})_{x_0,\lambda}=f$,  inequality \eqref{B-1} follows from inequality \eqref{HLSD-2}.

In the case of $\mu={n+\alpha}$, we have
\[\|g\|_{L^\frac{2n}{n+\alpha}(\mathbb{R}^{n}_+)}
= \|g_{\mu, x_0,\lambda}\|_{L^\frac{2n}{n+\alpha}(B_{\frac \lambda2}(x_1))}
\]
and
\[\|R_\alpha g\|_{L^\frac{2(n-1)}{n-\alpha}(\partial \mathbb{R}^{n}_+)}=\|\widetilde{R_\alpha}
(g_{\mu,x_0,\lambda})\|_{L^\frac{2(n-1)}{n-\alpha}(\partial B_{\frac \lambda2}(x_1))}.
\]
Inequality \eqref{B-2} follows from inequality \eqref{HLSD-3}.

Now we compute the best constant $C_e(n,\alpha, p)$. From Theorem \ref{Classification}, we know that
\[f(y)=\big(\frac{\lambda}{|y-x_0|}\big)^{n+\alpha-2}, \quad\forall\,y\in\partial\mathbb{R}^{n}_+
\]
is an extremal function to inequality \eqref{HLSD-2} for any $\lambda>0$ and $x_0=(0, -\lambda)$. Let $x_1=(0, -\lambda/2)$ and $\mu=n+\alpha-2$.  Then for
$y\in \partial B_{\frac \lambda2}(x_1)$,
\begin{eqnarray*}
f_{\mu,x_0,\lambda}(y)&=&\big(\frac{\lambda}{|y-x_0|}\big)^{n+\alpha-2}
\big(\frac{\lambda}{|y^{x_0,\lambda}-x_0|}\big)^{n+\alpha-2}\\
&=&\big(\frac{\lambda}{|y-x_0|}\big)^{n+\alpha-2} \big(\frac{|y-x_0|}{\lambda}\big)^{n+\alpha-2}=1;
\end{eqnarray*}
And,
\begin{eqnarray*}
\|\widetilde{E_\alpha} (f_{\mu,x_0,\lambda})\|^\frac{2n}{n-\alpha}_{L^{\frac{2n}{n-\alpha}}( B_{\frac \lambda2}(x_1))}
&=&\int_{ B_{\frac \lambda2}(x_1)}\big(\int_{\partial B_{\frac \lambda2}(x_1)}\frac{f_{\mu,x_0,\lambda}(z)}{|\xi-z|^{n-\alpha}}d{S_z}\big)^\frac{2n}{n-\alpha}d\xi\\
&=&\int_{ B_{\frac \lambda2}(x_1)}\big(\int_{\partial B_{\frac \lambda2}(x_1)}\frac{1}{|\xi-z|^{n-\alpha}}d{S_z}\big)^\frac{2n}{n-\alpha}d\xi\\
&=&(\frac\lambda2)^{\frac{n+\alpha-2}2\cdot \frac{2n}{n-\alpha}}\int_{B_{1}(x_2)}\big(\int_{\partial B_{1}(x_2)}\frac{1}{|\xi-z|^{n-\alpha}}d{S_z}\big)^\frac{2n}{n-\alpha}d\xi,
\end{eqnarray*} where $x_2=\big(\frac{\lambda}2\big)^{-1}x_1=(0,-1)$.
On the other hand,
\begin{eqnarray*}
\|f\|_{L^\frac{2(n-1)}{n+\alpha-2}(\partial\mathbb{R}^{n}_+)}
&=&\|f_{\mu,x_0,\lambda}\|_{L^\frac{2(n-1)}{n+\alpha-2}(\partial B_{\frac \lambda2}(x_1))}\\
&=&\big(\int_{\partial B_{\frac \lambda2}(x_1)}d S_z\big)^\frac{n+\alpha-2}{2(n-1)}\\
&=&\big(n\omega_{n}\big(\frac{\lambda}2\big)^{n-1}\big)^\frac{n+\alpha-2}{2(n-1)}.
\end{eqnarray*}
Thus, we have
\begin{eqnarray*}
C_e(n,\alpha, \frac{2(n-1)}{n+\alpha-2})&=&\big(\int_{B_{\frac \lambda2}(x_1)}|E_\alpha f_{\mu,x_0,\lambda}(\xi)|^\frac{2n}{n-\alpha}
d\xi\big)^\frac{n-\alpha}{2n}\cdot\|f_{\mu,x_0,\lambda}\|_{L^\frac{2(n-1)}{n+\alpha-2}(\partial B_{\frac \lambda2}(x_1))}^{-1}\\
&=&(n\omega_{n})^{-\frac{n+\alpha-2}{2(n-1)}}\big(\int_{B_{1}(x_2)}\big(\int_{\partial B_{1}(x_2)}\frac{1}{|\xi-z|^{n-\alpha}}d{S_z}\big)^\frac{2n}{n-\alpha}d\xi\big)^\frac{n-\alpha}{2n}.
%&=&\|E_\alpha f\|_{L^\frac{2n}{n-\alpha}(\mathbb{R}^{n}_+)}.
\end{eqnarray*}

For general $\alpha>1$, it is not easy to obtain the  precise value for the sharp constant.
%{\bf limit exists for $\alpha \to 1?$}
However, for $\alpha=2$, we can
identify it.

Observe: for $\alpha=2$, if $\lambda=2, \, x_0=(0, -\lambda), \, x_1=(x_0, -\lambda/2)$, then $f_{x_0,\lambda}=1$ on $\partial B_{\frac \lambda2}(x_1)$ and $\widetilde{E_2} (f_{x_0,\lambda})(\xi)$ is a harmonic function in  $B_{\frac \lambda2}(x_1)$.
%Also, it is
%easy to see that $\widetilde{E_2} (f_{x_0,\lambda})(\xi)$ is a constant on $\partial B_{\frac \lambda2}(x_1))$.  Without loss generality, choose $\lambda=2$.
Thus
$\widetilde{E_2} (f_{x_0,\lambda})(\xi)$ is a constant, in particular
$$
\widetilde{E_2} (f_{x_0,\lambda})(\xi)=\widetilde{E_2} (f_{x_0,\lambda})(\xi)|_{\xi=x_1}=n\omega_{n}.
$$
It follows that
\begin{equation*}
C_e(n,2, \frac{2(n-2)}{n})=\omega_{n}^{1-\frac1n-\frac{1}{2(n-1)}} n^{\frac{n-2}{2(n-1)}}.
\end{equation*}

%Similarly,
%\[\widetilde{R_2}(g_{x_0,\lambda})(z)=\omega_{n }.
%\]
%We have
%\begin{equation}\label{sharpC2}
%C_r(n,2,\frac{2n}{n+2})=\omega_{n}^{1-\frac1n-\frac{1}{2(n-1)}} n^{\frac{n-2}{2(n-1)}}.
%\end{equation}
\end{proof}

\medskip

%{\bf will check later, 2-27-2012}

 We remark here that for $\alpha \ne 2$, $\widetilde{E_\alpha} (1)$ may not be constant in the ball $B_{\frac \lambda2}(x_1)$, thus
solution $E_\alpha (f)(x)$ to \eqref{equ-1} may not be, up to a constant multiplier, in the form of
\begin{equation}\label{standard_buble}
(\frac {1}{|x'-x_0|^2+(x_n+\lambda)^2})^{\frac{n+\alpha-2}2}.
\end{equation}
For example, for $\alpha =4$ and $n\ge5$, let $ U_4(x)=\widetilde{E_\alpha} (1)$. Use the same notations as those in the proof of
Corollary \ref{ball-1}, then for $ x \in B_{\frac \lambda2}(x_1)$,
\begin{eqnarray*}
V(x):=\Delta_x U_4(x)=2(4-n)\int_{\partial B_{\frac \lambda2}(x_1)}\frac{1}{|x-z|^{n-2}}dz.
\end{eqnarray*}
 Clearly $V$ satisfies
\begin{equation}\label{PI-13}
\begin{cases}
\Delta_xV(x)=0,~~~~~~~~~~~&\text{in}~B_{\frac \lambda2}(x_1),\\
V=c_1,~~~~~~~~~~~&\text{on}~\partial B_{\frac \lambda2}(x_1).
\end{cases}
\end{equation} If follows that $V \equiv c_1$ in $B_{\frac \lambda2}(x_1)$.
It implies
%\begin{equation*}
%\begin{cases}
%\Delta_xU_4(x)=-c_2,~~~~~~~~~~~&\text{in}~B_{\frac d2}(x_1),\\
%U_4=1,~~~~~~~~~~~&\text{on}~\partial B_{\frac d2}(x_1).
%\end{cases}
%\end{equation*}Using the uniqueness of solutions on the above problem, we know
\[U_4(x)=c_1+c_2|x-x_1|^2.
\]
for some constant $c_1$ and  $c_2\ne 0$. %where $a_1=\frac{8N+c_2d^2}{4N},b_1=\frac{c_2}{N}$. Thus
Thus for an extremal function $f(y)$ (so that $U_4(x)=f_{x_1, \lambda}$),
$$
E_4(f)(x)=c\int_{\partial \mathbb{R}^n_+}(\frac 1{|y-y_0|^2+\lambda^2})^{\frac{n+2}2} \cdot\frac{1}{|x-y|^{n-4}}dy,
$$
is not in the form of (\ref{standard_buble}).

%This can also be easily seen from below (section \ref{subsect 5.2}).
%For convenience, we denote $F(\eta)=
%\big(\frac\lambda{\eta-x_0}\big)^\frac{n+\alpha-2}f(\eta^{x_0,\lambda}), G(\xi)=
%\big(\frac\lambda{\xi-x_0}\big)^\frac{n+\alpha}f(\xi^{x_0,\lambda})$ for $\xi\in B_\frac12(x_1),\eta\in\partial B_\frac12(x_1),\lambda>0.$

%certainly is a new phenomenon.
%That is, the following biharminic  equation
%\begin{equation}\label{PI-14}
%\begin{cases}
%(-\Delta_x)^2u(x)=0,~~~~~~~~~~~&\text{in}~B_{\frac d2}(x_1),\\
%u(x)=c,~~~~~~~~~~~&\text{on}~\partial B_{\frac d2}(x_1)
%\end{cases}
%\end{equation}
%has new solutions.

\medskip

As a simple consequence, we show inequality (\ref{HLSD-3}) implies standard trace inequality with $p$-biharmonic operator.

 %{\bf not check the following yet. Wait for $\alpha=1$}

%{\bf I am here 2-16-2012}

\begin{crl}\label{Trace inequality}
For $n>2$ and $p\in(1,n/2),$ and $q=p(n-1)/(n-2p)$, there is a constant $C(p, n)>0,$ such that,  for all $f\in
W^{2,p}(\mathbb{R}^{n}_+),$
\begin{eqnarray*}
\|f\|_{L^q(\partial\mathbb{R}^{n}_+)}\le C(p,n)\|\Delta f \|_{L^p(\mathbb{R}^{n}_+)}.
\end{eqnarray*}
\end{crl}

 To this end, we need the following representation formula.
\begin{lm} \label{two formulas} For  any $f\in C^\infty_0(\mathbb{R}^{n}_+)$ and $x \in \partial \mathbb{R}^{n}_+$,
\begin{eqnarray*}\label{Tr-1}
f(x)=\frac{1}{C(n)}\int_{\mathbb{R}^n_+}\frac{\langle\nabla f(y),x-y\rangle}{|x-y|^{n}}dy=\frac{1}{(2-n)C(n)}\int_{\mathbb{R}^n_+}\frac{\Delta f(y) }{|x-y|^{n-2}}dy,
\end{eqnarray*} where $C(n)=n\omega_{n}/2 $ is the half of the surface area of the unit sphere $S^n$.
\end{lm}
\begin{proof}
Let $z\in \partial B^+_1(x)$ and  write
\begin{eqnarray*}
f(x)=-\int^\infty_0\frac{d}{dt}f(x+tz)dt=-\int^\infty_0\langle\nabla f(x+tz),z\rangle  dt.
\end{eqnarray*}
Integrating both sides on $\partial B^+_1(x)$ with respect to $z$ variable, we have
\begin{eqnarray*}
C(n)f(x)&=&-\int_{\partial B_1^+(x)}\int^\infty_0\langle\nabla f(x+tz),z\rangle  dt dS_{z} \\
% & & -\int_0^{\frac\pi 2}\sin^{n-2}\theta d\theta \int_{S^{n-2}}\int^\infty_0\langle\nabla_{x+tz} f(x+tz),z\rangle t^{n-1+1-n}dt dS_{z'}\\
&=&\int_{\mathbb{R}^n_+}\frac{\langle\nabla_{y} f(y),x-y\rangle}{|x-y|^n}dy.
\end{eqnarray*}

% and then change from the polar
%coordinates $t$ and $z$ to $y$ which is related to $t$ and $z$ by $y-x = tz$ with $z=(z',z_n),z_n=\sin\theta$, we obtain
%\begin{eqnarray*}
%C(n)f(x)&=&-\int_0^{\frac\pi 2}\sin^{n-2}\theta d\theta \int_{S^{n-2}}\int^\infty_0\langle\nabla_{x+tz} f(x+tz),z\rangle t^{n-1+1-n}dt dS_{z'}\\
%&=&\int_{\mathbb{R}^n_+}\frac{\langle\nabla_{y} f(y),x-y\rangle}{|x-y|^n}dy.
%\end{eqnarray*}Thus, we obtain the first formula.

Further, using integrating by part, we have
\begin{eqnarray*}
\int_{\mathbb{R}^n_+}\frac{\langle\nabla_{y} f(y),x-y\rangle}{|x-y|^n}dy
&=&\frac1{n-2}\int_{\mathbb{R}^n_+}\langle\nabla_{y} f(y),\nabla_{y}(|x-y|^{2-n}) \rangle  dy\\
&=&-\frac1{n-2}\int_{\mathbb{R}^n_+}\frac{\Delta f(y)}{ |x-y|^{n-2}}dy.
\end{eqnarray*}
%Combining this and the above, we conclude the second formula. The proof is completed.
\end{proof}

 {\bf Proof of corollary \ref{Trace inequality}.}
From Lemma \ref{two formulas}, we have
for any $f\in C^\infty_0(\mathbb{R}^{n}_+)$,
\begin{eqnarray*}
 |f(x)|\le\frac{1}{(n- 2)C(n)}\int_{\mathbb{R}^n_+}\frac{|\Delta f(y)| }{|x-y|^{n-2}}dy.
\end{eqnarray*}
Let $\frac1{q}=\frac n{n-1}\big(\frac1{p}-\frac 2n\big)$.   It follows from the above and \eqref{HLSD-3} that
\begin{eqnarray*}
\|f\|_{L^{q}(\partial\mathbb{R}^{n}_+)}\le C(p,n)\|\Delta f\|_{L^{p}(\mathbb{R}^{n}_+)}.
\end{eqnarray*}

\subsection{Inequality in limit case $\alpha=n$}\label{subsect 5.2}
Inequality \eqref{B-1} is equivalent to the following inequality
\begin{equation}\label{HLSD-B}
\int_{B_1}\int_{\partial B_1}\frac{G(\xi) F(\eta)}{|\xi-\eta|^{n-\alpha}}dS_\eta d\xi\le C_e (n,\alpha,p)\|F\|_{L^p(\partial B_1)}\|G\|_{L^t(B_1)},
\end{equation}  with $p=\frac{2(n-1)}{n+\alpha-2}$ and $ t=\frac{2n}{n+\alpha}$.

Standard limiting argument for $\alpha \to n^-$ (see, for example, Beckner \cite{B1993})
and tedious computations yield
%Now we consider the limit case for inequality \eqref{HLSD-B}. That is, for $\alpha=n$, we can obtain the following %inequalities via taking limit in inequality \eqref{HLSD-B}.
\begin{crl}\label{folklore}
Assume that $F$ and $G$ are nonnegative $L\ln L$  functions on $\partial B_1$ and $B_1$, respectively, with $\int_{\partial B_1}F(\eta)dS_\eta=\int_{ B_1}G(\xi)d\xi=1.$ Then
\begin{eqnarray*}%\label{log-H-1}
& &-2n\omega_n \int_{B_1}\int_{\partial B_1} G(\xi) \ln{|\xi-\eta|}F(\eta)dS_\eta d\xi\nonumber\\
&\le&
\frac1{n}\int_{B_1}G(\xi)\ln(G(\xi))d\xi+\frac1{n-1}\int_{\partial B_1}F(\eta)\ln(F(\eta))dS_\eta+C_n,
\end{eqnarray*}
where $C_n=\frac{\ln(n\omega_n)}{n-1}+\frac{1}{n}\ln \int_{B_1}e^{I_n(\xi)}d\xi$  and $I_n(\xi)=-2\omega_n^{-1}\int_{\partial B_1}\ln |\xi-\eta|dS_\eta.$
%checked with Jingbo 3-26-2012.
%and
% the extremal functions must be the following form
%\begin{eqnarray*}
% F(\eta)&=&c_1(|\eta-\eta_0|^2+d^2)^{-(n-1)},\eta,\eta_0\in \partial B_1,\\
% G(\xi',0)&=&c_2(|\xi'-\eta_0|^2+d^2)^{-n},  \xi'\in \partial B_1,
%\end{eqnarray*}where $c_1,c_2,d>0$ are the some constants.
%$C_n=\frac1n\ln\int_{ B_1}|\xi-\eta|^{-\frac{2n(n-1)}{2n-1}}d\xi+\frac1{n-1}\ln\int_{\partial %B_1}|\xi-\eta|^{-\frac{2n(n-1)}{2n-1}}d\eta, \xi=(\xi',\xi_n)\in B_1.$

Equivalently, assume that $f(y)$ and $g(x)$ are nonnegative $L\ln L$  functions on $\partial\mathbb{R}^n_+$ and $\mathbb{R}^n_+$, respectively, with $\int_{\partial \mathbb{R}^n_+}f(y)d y=\int_{ \mathbb{R}^n_+}g(x)dx=1.$ Then
\begin{eqnarray}\label{log-H-2}
& &-2n\omega_n \int_{\mathbb{R}^n_+}\int_{\partial \mathbb{R}^n_+} g(x) \ln{|x-y|}f(y)dy dx\nonumber\\
&\le&\frac1{n}\int_{\mathbb{R}^n_+}g(x)\ln(g(x))dx
+\frac1{n-1}\int_{\partial \mathbb{R}^n_+}f(y)\ln(f(y))dy+C_n.
\end{eqnarray}
\end{crl}
%Using the method of moving sphere of  Li and Zhu \cite{LZ1995},
%one may check that the extremal functions to inequality \eqref{log-H-2} must be the following form
%\begin{eqnarray*}
% f(y)&=&c_1(|y-y_0|^2+d^2)^{-(n-1)},\, y, \, y_0\in \partial \mathbb{R}^n_+,\\
% g(x',0)&=&c_2(|x'-y_0|^2+d^2)^{-n},  \, x'\in \partial \mathbb{R}^n_+,
%\end{eqnarray*}where $c_1,c_2,d>0$ are the some constants.
%
%The case $\alpha>n$ will be addressed in another paper \cite{DZ3}.

\subsection{Fractional Lapalacian operators}
%{\bf will modify 2-27-2012}
Let
\begin{equation}\label{1-8_au}
u(x)=\frac 1{c(n,\alpha)} \int_{\partial \mathbb{R}^{n}_+}\frac{ f(y)}{|x-y|^{n-\alpha}}dy.
\end{equation}
Direct computation shows that $u(x)$  satisfies the following equation for $\alpha= 2m$.

\begin{prop} \label{prop1-1}  Assume that $f(y)\in C^\infty(\mathbb{R}^{n-1})$. If
$u(x)$ satisfies \eqref{1-8_au} for $\alpha=2m$ ($m\in \mathbb{Z})$, then
\begin{equation}\label{1-8}
\begin{cases}
(-\Delta)^{\frac{\alpha}2} u=0 ~&\mbox{for\ } x_n>0 ~ \mbox{and\ }
x' \in \mathbb{R}^{n-1},\\
\frac{\partial }{\partial x_n}[(-\Delta)^{k}u]=0,\frac{\partial}{\partial
x_n}[(-\Delta)^{m-1}u]=(-1)^{m}f(x'),~&x' \in \mathbb{R}^{n-1}, %~\text{or}\\
%\frac{\partial^k u}{\partial x_n^k}=0,\frac{\partial^{m-1} u}{\partial x_n^{m-1}}=(-1)^{m-1}f(x'),~&x' \in \mathbb{R}^{n-1}.
\end{cases}
\end{equation}where $k=0,1,2,3,\cdots,m-2.$
% Conversely, if $u(x)$ is a solution to (\ref{1-8}), it then satisfies (\ref{1-7}).
\end{prop}

Note that
\begin{equation}\label{equivalent}
c(n-1,\alpha-1)^{-1}E_\alpha f(x',0) =I_\alpha f(x').
\end{equation}
As a simple consequence, we know that the global defined equation
$$
(-\Delta)^{\frac{\beta}2} u(x)=f(x),\quad x\in \mathbb{R}^{n-1}
$$
is equivalent to a pointwise defined equation \eqref{1-8}, where $\beta=\alpha-1.$ In particular,
for $\beta=1$ (i.e. $\alpha=2$), the above equivalent relation is a well known fact, and was pointed out in, e.g. Caffarelli and Silverster \cite{CS2007}, and Cabre and Tan
\cite{CT2010}, and others.

\subsection{Nonexistence of positive solutions to the integral system with subcritical powers}

As a byproduct of using  method of moving sphere, we will show system \eqref{sys-1} with subcritical powers:
\begin{equation}\label{subcritical}
1<\theta\le\frac{n+\alpha-2}{n-\alpha}, \, \, \, 1<\kappa\le\frac{n+\alpha}{n-\alpha},  \, \, \mbox{and} \, \, \, \theta+\kappa<\frac{2n+2\alpha-2}{n-\alpha} \end{equation}
has only a pair of trivial non-negative solutions.
Note that $\theta+\kappa=\frac{2n+2\alpha-2}{n-\alpha}$ is equivalent to the critical case
\[\frac{1}{ \theta+1 }+\frac{n }{(\kappa+1)( n-1)}=\frac{n-\alpha}{n-1},
\]
which is a necessary condition for inequality \eqref{HLSD-2}.
\begin{thm}\label{sys-nonexistence}
Let $1<\alpha<n.$ Suppose $(u,v)\in L_{loc}^{\frac{(n-1)(\theta-1)}{\alpha-1}}(\partial\mathbb{R}^{n}_+)\times L_{loc}^{\frac{n(\kappa-1)}{\alpha}}(\overline{\mathbb{R}^{n}_+})$ is a pair of
non-negative solutions to system \eqref{sys-1} with $\theta, \  \kappa$ satisfying (\ref{subcritical}). Then $u=v=0$.
\end{thm}

%{\bf I am here now 2-13-2012}
Note  (\ref{subcritical}) is equivalent to
\[
\frac{2(n-1)}{n+\alpha-2}\le p<2,  \, 2<q\le \frac{2n}{n-\alpha}, \, \, \mbox{and} \, \, \frac{1}{p}-\frac{n}{q (n-1)}<\frac{\alpha-1}n .\]
Theorem \ref{sys-nonexistence} yields the proof for Theorem \ref{Nonexistence}.
%Suppose $$, we have the inequality \[\frac{n-1}{np}-\frac{1}q<\frac{\alpha-1}n,
%\]
% and there exists only trivial solution for equation \eqref{equ-1}. That is,  for $\theta<\frac{n+\alpha-2}{n-\alpha},
%\kappa<\frac{n+\alpha}{n-\alpha}$ in  satisfying
%\[\frac{n-1}{n(\theta+1)}+\frac{1}{\kappa+1}>\frac{n-\alpha}n,
%\]
%we have the following nonexistence result:

To prove the theorem, we first show

\begin{lm}\label{K-lm-no-1} Under the same  assumptions in Theorem \ref{sys-nonexistence},
if $(u,v)$ is  a pair of nonnegative solutions to system \eqref{sys-1}, then for any  $x\in\partial\mathbb{R}^{n}_+$, there exists
$\lambda_0(x)>0$ such that : $0<\lambda<\lambda_0(x),$
\begin{eqnarray*}
u_{x,\lambda}(\xi)&\le& u(\xi),\quad \forall\, \xi\in\Sigma_{x,\lambda}^{n-1},\\
v_{x,\lambda}(\eta)&\le& v(\eta),\quad\forall\,\eta\in\Sigma_{x,\lambda}^n.
\end{eqnarray*}
\end{lm}
\begin{proof}  The proof is similar to that of Lemma \ref{K-lm-1}.
 Since $\kappa\le\frac{n+\alpha}{n-\alpha}$, it is obvious that $\tau_1\ge0,$ and $\big(\frac{\lambda}{|\eta-x|}\big)^{\tau_1}\le1$ for $\eta\in\Sigma_{x,\lambda}^n.$ Thus,
for any $\xi\in\Sigma_{x,\lambda}^{u}$, using \eqref{K-3} and mean value theorem it has
\begin{eqnarray*}
0\le u_{x,\lambda}(\xi)-u(\xi)&=&\int_{\Sigma_{x,\lambda}^n}P(x,\lambda;\xi,\eta)
\big[\big(\frac{\lambda}{|\xi-x|}\big)^{\tau_1}v_{x,\lambda}^{\kappa}(\eta)-v^{\kappa}(\eta)\big]d\eta,\\
&\le&\int_{\Sigma_{x,\lambda}^n}\frac{v_{x,\lambda}^{\kappa}(\eta)-v^{\kappa}(\eta)}
{|\xi-\eta|^{n-\alpha}}d\eta,\\
&\le&\kappa\int_{\Sigma_{x,\lambda}^v}\frac{v_{x,\lambda}^{\kappa-1}(\eta)(v_{x,\lambda}(\eta)-v(\eta))
}{|\xi-\eta|^{n-\alpha}}d\eta.
\end{eqnarray*}
For any $t\in(1,\frac n\alpha),$ take $k=\frac{(n-1)t}{n-\alpha t}, s=\frac{nt}{n-\alpha t}$.  Similar to Lemma \ref{K-lm-1}, integrating on
$\partial\mathbb{R}^{n}_+$ and using inequality \eqref{HLSD-3} and H\"{o}lder inequality,  from the above inequality, we have
\begin{eqnarray}\label{K-13}
\|u_{x,\lambda}-u\|_{L^k(\Sigma_{x,\lambda}^{u})} &\le&c\|v_{x,\lambda}^{\kappa-1}(v_{x,\lambda}-v)\|_{L^t(\Sigma_{x,\lambda}^v)}\nonumber\\
&\le&c\|v_{x,\lambda}\|^{\kappa-1}_{L^{\frac{n(\kappa-1)}{\alpha}}(\Sigma_{x,\lambda}^v)}
\|v_{x,\lambda}-v\|_{L^s(\Sigma_{x,\lambda}^v)}.
\end{eqnarray}
On the other hand, since $\theta\le\frac{n+\alpha-2}{n-\alpha}, $ it holds $\tau_2\ge0,$ and
$\big(\frac{\lambda}{|\xi-x|}\big)^{\tau_2}\le1$ for $\xi\in\Sigma_{x,\lambda}^{n-1}.$   Thus, for any
$\eta\in\Sigma_{x,\lambda}^{v}$, it  follows from \eqref{K-4} that
\begin{eqnarray*}
v_{x,\lambda}(\eta)-v(\eta) &\le&\theta\int_{\Sigma_{x,\lambda}^u}\frac{u_{x,\lambda}^{\theta-1}(\xi)(u_{x,\lambda}(\xi)-u(\xi))
}{|\eta-\xi|^{n-\alpha}}d\xi,
\end{eqnarray*}
 Since $s>\frac{n}{n-\alpha}$ for any $t\in(1,\frac n\alpha)$, we can use HLS inequality \eqref{HLSD-2} and H\"{o}lder inequality to arrive at
 \begin{eqnarray}\label{K-14}
\|v_{x,\lambda}-v\|_{L^s(\Sigma_{x,\lambda}^v)}
&\le&c\|u_{x,\lambda}^{\theta-1}(u_{x,\lambda}-u)\|_{L^{\frac{(n-1)s}{n+(\alpha-1)s}}
(\Sigma_{x,\lambda}^u)}\nonumber\\
&\le&c\|u_{x,\lambda}\|^{\theta-1}_{L^{\frac{(n-1)(\theta-1)}{\alpha-1}} (\Sigma_{x,\lambda}^u)}\|u_{x,\lambda}-u\|_{L^{k}
(\Sigma_{x,\lambda}^u)}.
\end{eqnarray}
 Combining \eqref{K-13} with \eqref{K-14}, we obtain
\begin{eqnarray}
\|u_{x,\lambda}-u\|_{L^k(\Sigma_{x,\lambda}^{u})}
&\le&c\|v_{x,\lambda}\|^{\kappa-1}_{L^{\frac{n(\kappa-1)}{\alpha}}(\Sigma_{x,\lambda}^v)}
\|u_{x,\lambda}\|^{\theta-1}_{L^{\frac{(n-1)(\theta-1)}{\alpha-1}} (\Sigma_{x,\lambda}^u)}\|u_{x,\lambda}-u\|_{L^{k}
(\Sigma_{x,\lambda}^u)}\nonumber\\
&\le&c\|v\|^{\kappa-1}_{L^{\frac{n(\kappa-1)}{\alpha}}(B^+_\lambda(x))}
 \|u\|^{\theta-1}_{L^{\frac{(n-1)(\theta-1)}{\alpha-1}}
(B^{n-1}_\lambda(x))}\|u_{x,\lambda}-u\|_{L^{k} (\Sigma_{x,\lambda}^u)}.\nonumber\\ \label{K-15}
\end{eqnarray}
Similarly, we have
\begin{eqnarray}
\|v_{x,\lambda}-v\|_{L^s(\Sigma_{x,\lambda}^v)} &\le&c\|v \|^{\kappa-1}_{L^{\frac{n(\kappa-1)}{\alpha}}(B^+_\lambda(x))}
 \|u \|^{\theta-1}_{L^{\frac{(n-1)(\theta-1)}{\alpha-1}}
(B^{n-1}_\lambda(x))}\|v_{x,\lambda}-v\|_{L^s(\Sigma_{x,\lambda}^v)}.\nonumber\\
\label{K-16}
\end{eqnarray}
Since $u\in L_{loc}^{\frac{(n-1)(\theta-1)}{\alpha-1}}(\partial\mathbb{R}^{n}_+)$ and $v\in
L_{loc}^{\frac{n(\kappa-1)}{\alpha}}(\mathbb{R}^{n}_+),$ we can choose $\lambda_0$ small enough such that $0<\lambda<\lambda_0$, and
\[c\|v \|^{\kappa-1}_{L^{\frac{n(\kappa-1)}{\alpha}}(B^+_\lambda(x))}
 \|u \|^{\theta-1}_{L^{\frac{(n-1)(\theta-1)}{\alpha-1}}
(B^{n-1}_\lambda(x))}\le\frac12.
\]
Combining the above with \eqref{K-15} and \eqref{K-16}, we get
\begin{eqnarray*}
\|u_{x,\lambda}-u\|_{L^k(\Sigma_{x,\lambda}^{u})}&\le&\frac12\|u_{x,\lambda}-u\|_{L^k(\Sigma_{x,\lambda}^{u})},\\
\|v_{x,\lambda}-v\|_{L^s(\Sigma_{x,\lambda}^v)} &\le&\frac12\|v_{x,\lambda}-v\|_{L^s(\Sigma_{x,\lambda}^v)},
\end{eqnarray*}
which imply $\|u_{x,\lambda}-u\|_{L^k(\Sigma_{x,\lambda}^{u})}=\|v_{x,\lambda}-v\|_{L^s(\Sigma_{x,\lambda}^v)}=0$. Thus, both
$\Sigma_{x,\lambda}^{u}$ and $\Sigma_{x,\lambda}^v$ have measure zero.  We complete the proof of the lemma.
\end{proof}

Define
\[\bar{\lambda}(x)=sup\{\mu>0\,|\,u_{x,\lambda}(\xi)\le u (\xi), \mbox{and} \ v_{x,\lambda}(\eta)\le v
(\eta), \  \forall \lambda\in (0, \mu), \forall \xi\in\Sigma_{x,\lambda}^{n-1}, \forall \eta\in \Sigma_{x,\lambda}^{n}\}.
\]
We will show  the sphere will never stop.

\begin{lm}\label{K-lm-no-2} $\bar{\lambda}(x)=\infty$ for all $x\in\partial\mathbb{R}^{n}_+.$
\end{lm}
\begin{proof} We prove it by contradiction argument. Suppose the contrary, there exists some $x_0\in\partial\mathbb{R}^{n}_+$ such that $\bar{\lambda}(x_0)<\infty.$
By the definition of $\bar{\lambda}$,
\begin{eqnarray*}
u_{x_0,\bar{\lambda}}(\xi)&\le&u(\xi)\quad
\text{for}~\xi\in\Sigma_{x_0,\bar{\lambda}}^{n-1},\\
v_{x_0,\bar{\lambda}}(\eta)&\le&v(\eta)\quad \text{for}~\eta\in \Sigma_{x_0,\bar{\lambda}}^{n}.
\end{eqnarray*}
From \eqref{K-3} and \eqref{K-4} with $x=x_0,\lambda=\bar{\lambda}$ and the fact that at least one of $\tau_1$ and $\tau_2$ is positive, we have
\begin{eqnarray*}
u_{x_0,\bar{\lambda}}(\xi)&<&u(\xi)\quad
\text{for}~\xi\in\Sigma_{x_0,\bar{\lambda}}^{n-1},\\
v_{x_0,\bar{\lambda}}(\eta)&<&v(\eta)\quad \text{for}~\eta\in \Sigma_{x_0,\bar{\lambda}}^{n}.
\end{eqnarray*}
Similar to proof process of Lemma \ref{K-lm-1}, for $
\lambda\in[\bar{\lambda},\bar{\lambda}+\varepsilon) $, we can conclude that
$\Sigma_{x_0,\lambda}^{u}$ and
 $\Sigma_{x_0,\lambda}^v$ must have measure zero. Thus we obtain that
\begin{eqnarray*}
u_{x_0, {\lambda}}(\xi)&\le&u(\xi)\quad
\text{for}~\xi\in\Sigma_{x_0, {\lambda}}^{n-1},\\
v_{x_0, {\lambda}}(\eta)&\le&v(\eta)\quad \text{for}~\eta\in \Sigma_{x_0, {\lambda}}^{n}
\end{eqnarray*}for $
\lambda\in[\bar{\lambda},\bar{\lambda}+\varepsilon) $,  which contradicts the definition of $\bar{\lambda}$.
\end{proof}

\textbf{Proof of Theorem \ref{sys-nonexistence}.}
According to
Lemma \ref{K-lm-no-2},  $\bar{\lambda}(x)=\infty$ for all
$x\in\partial \mathbb{R}^n_+$, that is, for all $\lambda>0$ and $x\in\partial \mathbb{R}^n_+$,
\begin{eqnarray*}
u_{x, {\lambda}}(\xi)&\le&u(\xi)\quad
\text{for}~\xi\in\Sigma_{x, {\lambda}}^{n-1},\\
v_{x, {\lambda}}(\eta)&\le&v(\eta)\quad \text{for}~\eta\in
\Sigma_{x, {\lambda}}^{n}.
\end{eqnarray*} As being shown in the proof of Theorem
\ref{sys-critical}, this is impossible. This
completes the proof. \hfill$\Box$

 \vskip 1cm
\noindent {\bf Acknowledgements}\\
\noindent %The author would like to thank the referee for his/her careful reading of the manuscript and many good suggestions.
M. Zhu would like to thank Y.Y. Li and E. H. Lieb for their interests in the paper and their comments related to Lemma 3.5 and 3.6. He also thanks R. Frank and P. Yang for valuable discussions that lead to the general extension of sharp HLS inequality.
This paper is written while J. Dou was visiting the Department of
Mathematics, The University of Oklahoma. He thanks the department
for its hospitality. The work of J.Dou is partially supported by the
National Natural Science Foundation of China (Grant No. 11101319) and
CSC project for visiting The University of Oklahoma.

\small
\begin{thebibliography}{a}\small
\bibitem{ACW2000} J. Ai, K-S. Chou, J. Wei, Self-similar solutions for the anisotropic affine curve shortening problem,  Calc. Var. Partial Differential Equations 13 (2000), 311-337.
\bibitem{B1993} W. Beckner, Sharp Sobolev inequalities on the sphere and the Moser-Trudinger inequality,
Ann. of Math. 138 (1993), 213-242.
\bibitem{BLL1974} H. J. Brascamp; E. H. Lieb; J. M. Luttinger, A general rearrangement inequality for multiple integrals. J. Functional Analysis  17  (1974), 227?37.
\bibitem{BR1992} H. Brezis, Uniform estimates for solutions of $-\Delta u=V(x)u^p$, in Partial Differential Equations and related subjects, Trento, 1990, (M. Miranda, ed.), Longman, 1992, p. 38-52.
\bibitem{BL1983} H. Brezis, E. Lieb, A relation between pointwise convergence of functions and convergence of functionals, Proc. Amer. Math. Soc. 88 (1983), 486-490.
\bibitem{BK1979} H. Brezis, T. Kato, Remarks on the Schr\"{o}dinger operator with singular complex potentials, J. Math. Pures Appl. 58 (1979), 137-151.
\bibitem{CT2010} X. Cabr\'{e}, J. Tan, Positive solutions of nonlinear problems involving the square root of the Laplacian, Adv. Math. 224 (2010), 2052-2093.
\bibitem{CS2007} L. Caffarelli, L. Silvestre, An extension problem related to the fractional Laplacian, Comm. Partial. Diff. Equ. 32 (2007), 1245-1260.
\bibitem{CG2011} S.-Y. Chang, M. Gonzalez, Fractional Laplacian in conformal geometry, Adv. Math. 226 (2011), 1410-1432.
\bibitem{CL2010} W. Chen, C. Li, Methods on Nonlinear Elliptic Equations, AIMS Ser. Differ. Equ. Dyn. Syst.  vol. 4, 2010.
\bibitem{CL2011} W. Chen, C. Li, Super polyharmonic property of solutions for PDE systems and its applications, arxiv: 1110.2539.
\bibitem{CLO2006} W. Chen, C. Li, B. Ou, Classification of solutions for an integral equation, Comm. Pure Appl. Math.  59 (2006), 330-343.
\bibitem{CW91} Y. Chen,  L. Wu, Second Order Elliptic Equations and Elliptic Systems, Translations of Mathematical Monographs. AMS, 1998.
\bibitem{DZ2012} J. Dou, M. Zhu, Two dimensional $L_p$ Minkowski problem and  nonlinear equations with negative exponents, Adv. Math.
 230 (2012), 1209-1221.
\bibitem{DZ2} J. Dou, M. Zhu, On the extension of Hardy-Littewood-Sobolev inequality II: The negative exponent,  in preparation.
\bibitem{DZ3} J. Dou, M. Zhu, Negative power Hardy-Littewood-Sobolev inequality on the upper half space, in preparation.
\bibitem{E1988} J. F. Escobar, Sharp constant in a Sobolev trace inequality, Indiana Univ. Math. J.  37,  (1988), 687-698.
\bibitem{E1992-1} J. F. Escobar, The Yamabe problem on manifolds with boundary, J. Differential Geom. 35 (1992),  21-84.
\bibitem{E1992} J. F. Escobar, Conformal deformation of a Riemannian metric to a scalar flat metric with constant mean curvature on the boundary, Ann. of Math. 136 (1992), 1-50.
\bibitem{FL2010} R. L. Frank, E. H. Lieb, Inversion positivity and the sharp Hardy-Littlewood-Sobolev inequality, Calc. Var. Partial Differential Equations  39 (2010),  85-99.
    \bibitem{FL2012} R. L. Frank,  E. H. Lieb, Sharp constants in several inequalities on the Heisenberg group, Ann. of Math.  to appear.
\bibitem{GM2011} M. Gonz\'{a}lez, R. Mazzeo, Y. Sire, Singular solutions of fractional order conformal Laplacians,
J. Geom. Anal. DOI: 10.1007/s12220-011-9217-9.
\bibitem{GQ2011} M. Gonz\'{a}lez, J. Qing, Fractional conformal Laplacians and fractional Yamabe problems, arXiv: 1012.0579v1
    \bibitem{Gr} L. Gross, Logarithmic Sobolev Inequalities,  Amer.  J.  Math.   97  (1976), 1061-1083.
\bibitem{HL1999} Z. Han,  Y. Y. Li, The Yamabe problem on manifolds with boundary: existence and
compactness results, Duke Math. J.  99 (1999), 489-542.

\bibitem{HY2004} F. Hang, P. Yang, The Sobolev inequality for Paneitz operator on three manifolds,  Calc. Var. Partial Differential Equations  21 (2004), 57-83.
\bibitem{HWY2008} F. Hang,  X. Wang,  X. Yan, Sharp integral inequalities for harmonic functions, Comm. Pure Appl. Math. 61 (2008), 0054-0095.
\bibitem{HWY2009}  F. Hang, X. Wang, X. Yan, An integral equation in conformal geometry, Ann. Inst. H.  Poincar\'{e} Analyse Non Lin\'{e}aire 26 (2009),  1-21.
\bibitem{HL1928} G. H. Hardy,  J. E. Littlewood, Some properties of fractional integrals (1), Math. Zeitschr. 27 (1928), 565-606.
\bibitem{HL1930} G. H. Hardy, J. E. Littlewood, On certain inequalities connected with the calculus of variations, J. London Math. Soc. 5 (1930), 34-39.
\bibitem{JLX2011a} T. Jin, Y. Y. Li, J.  Xiong, On a fractional Nirenberg problem, part I: blow up analysis and compactness of solutions, arXiv: 1111.1332v1
\bibitem{JLX2011b} T. Jin, Y. Y. Li, J.  Xiong, On a fractional Nirenberg problem, part II: existence of solutions, in preparation, 2011.
\bibitem{JX2011} T. Jin, J. Xiong, A fractional Yamabe flow and some applications, arXiv: 1110. 5664v1
\bibitem{Li2004} Y. Y. Li, Remark on some conformally invariant integral equations:
the method of moving spheres, J. Eur. Math. Soc.  6 (2004), 153-180.
\bibitem{LZ2003} Y. Y. Li, L. Zhang, Liouville type theorems and Harnack type
inequalities for semilinear elliptic equations, J. D'Anal. Math. 90 (2003), 27-87.
\bibitem{LZ1995} Y. Y. Li, M. Zhu, Uniqueness theorems
through the method of moving spheres, Duke Math. J.  80 (1995),
383-417.
\bibitem{LZ1997} Y. Y. Li, M. Zhu, Sharp Sobolev trace inequalities on
Riemannian manifolds with boundaries, Comm. Pure Appl. Math.  50 (1997), 449-487.
\bibitem{Lieb1977} E. H.  Lieb, Existence and Uniqueness of the Minimizing Solution of Choquard's Non-Linear Equation, Studies in Appl. Math. 57, 93-105 (1977).
\bibitem{Lieb1983} E. Lieb, Sharp constants in the Hardy-Littlewood-Sobolev and related inequalities, Ann. of Math. 118 (1983), 349-374.
\bibitem{LL2001} E. Lieb, M. Loss, Analysis, 2nd ed. Graduate Studies in Mathematics, 14. American Mathematical Society, Providence, R.I. 2001.
\bibitem{NZ1} Y. Ni, M. Zhu, Steady states for one dimensional curvature flows, Comm. Contemp. Math. 10 (2008), 155-179.
\bibitem{ONeil} R. O'Neil, Convolution operators and $L(p, q)$ spaces, Duke Math. J. 30 (1963), 129-142.
\bibitem{SY1988} R. Schoen, S.-T. Yau, Conformally flat manifolds, Kleinian groups and scalar curvature, Inventiones Mathematicae,  92 (1988), 47-71.
\bibitem{So1963} S. L. Sobolev, On a theorem of functional analysis, Mat. Sb. (N.S.) 4 (1938), 471-479. A. M. S. Transl. Ser. 2, 34 (1963), 39-68.
\bibitem{Stein} E. M. Stein, Singular integrals and differentiability properties of functions, Princeton Mathematical
Series, 30. Princeton University Press, Princeton, N.J. 1970.
\bibitem{SW1958} E. M. Stein, G. Weiss, Fractional integrals in n-dimensional Euclidean space,
J. Math. Mech. 7 (1958), 503-514.
\bibitem{WZ2012} X. Wang, M. Zhu, Personal communication, June 2012.
\bibitem{YZ2004} P. Yang, M. Zhu, On the Paneitz energy on standard three sphere, ESAIM Control Optim. Calc. Var. 10 (2004), 211-223.
\end{thebibliography}
\end{document}